\pdfoutput=1 
\documentclass[a4paper]{amsart}
\usepackage[margin=30mm]{geometry}
\usepackage{microtype}

\usepackage[all]{xy}

\usepackage{amssymb, amsthm, amsmath, graphics, indentfirst, paralist, sseq, mathtools, float}

\usepackage[scr=boondox]{mathalfa}

\theoremstyle{plain}
\newtheorem*{thmx}{Theorem}
\newtheorem{thm}{Theorem}[section]
\newtheorem{prop}[thm]{Proposition}
\newtheorem{lem}[thm]{Lemma}
\newtheorem{cor}[thm]{Corollary}

\theoremstyle{definition}
\newtheorem{defin}[thm]{Definition}

\theoremstyle{remark}
\newtheorem{rem}[thm]{Remark}

\begin{document}

\newcommand{\Ker}{\mathop{\mathrm{Ker}}\nolimits}
\newcommand{\Image}{\mathop{\mathrm{Im}}\nolimits}
\newcommand{\Coker}{\mathop{\mathrm{Coker}}\nolimits}
\newcommand{\hofib}{\mathop{\mathrm{hofib}}\nolimits}
\newcommand{\rk}{\mathop{\mathrm{rk}}\nolimits}
\newcommand{\sk}{\mathop{\mathrm{sk}}\nolimits}
\newcommand{\fr}{\mathrm{fr}}
\newcommand{\pr}{\mathop{\mathrm{pr}}\nolimits}
\newcommand{\lcm}{\mathop{\mathrm{lcm}}\nolimits}
\newcommand{\id}{\mathop{\mathrm{id}}\nolimits}
\newcommand{\ad}{\mathop{\mathrm{ad}}\nolimits}
\newcommand{\interior}{\mathop{\mathrm{int}}\nolimits}
\newcommand{\Hom}{\mathop{\mathrm{Hom}}\nolimits}
\newcommand{\Tor}{\mathop{\mathrm{Tor}}\nolimits}
\newcommand{\Ext}{\mathop{\mathrm{Ext}}\nolimits}
\newcommand{\Aut}{\mathop{\mathrm{Aut}}\nolimits}
\newcommand{\RU}{\mathop{\mathrm{RU}}\nolimits}
\newcommand{\st}{\mathop{\mathrm{st}}\nolimits}
\newcommand{\HH}{\mathcal{H}}
\newcommand{\Z}{\mathbb{Z}}
\newcommand{\R}{\mathbb{R}}
\newcommand{\C}{\mathbb{C}}
\newcommand{\quat}{\mathbb{H}}
\newcommand{\oct}{\mathbb{O}}
\newcommand{\UM}{\underline{M}}
\newcommand{\UMp}{\underline{M}{}'}
\newcommand{\UMpp}{\underline{M}{}''}
\newcommand{\UMs}{\underline{M}{}^*}
\newcommand{\UtM}{\underline{\tilde{M}}}
\newcommand{\UtMp}{\underline{\tilde{M}}{}'}
\newcommand{\UbM}{\underline{\bar{M}}}
\newcommand{\UbMp}{\underline{\bar{M}}{}'}
\newcommand{\UbMpp}{\underline{\bar{M}}{}''}
\newcommand{\UN}{\underline{N}}
\newcommand{\UNs}{\underline{N}{}^*}
\newcommand{\UbN}{\underline{\bar{N}}}
\newcommand{\UbNp}{\underline{\bar{N}}{}'}
\newcommand{\UH}{\underline{H}}
\newcommand{\UHp}{\underline{H}{}'}
\newcommand{\UR}{\underline{R}}
\newcommand{\UE}{\underline{E}}
\newcommand{\UEp}{\underline{E}{}'}
\newcommand{\UEs}{\underline{E}{}^*}
\newcommand{\UbE}{\underline{\bar{E}}}
\newcommand{\UbEs}{\underline{\bar{E}}{}^*}
\newcommand{\UV}{\underline{V}}
\newcommand{\UVp}{\underline{V}{}'}
\newcommand{\UbV}{\underline{\bar{V}}}
\newcommand{\UbVp}{\underline{\bar{V}}{}'}
\newcommand{\UtVp}{\underline{\tilde{V}}{}'}
\newcommand{\UbY}{\underline{\bar{Y}}}
\newcommand{\UbYp}{\underline{\bar{Y}}{}'}
\newcommand{\UbZ}{\underline{\bar{Z}}}
\newcommand{\UbZp}{\underline{\bar{Z}}{}'}
\newcommand{\gc}{\mathfrak{g}}
\newcommand{\gcb}{\bar{\mathfrak{g}}}
\newcommand{\gs}{\mathscr{g}}
\newcommand{\s}{\mathscr{s}}
\newcommand{\si}{\mathrel{\underset{^{\scriptstyle{s}}}{\cong}}}
\newcommand{\Str}{\mathcal{S}}
\newcommand{\SI}{\Str^{\st}}

\title{Extended surgery theory for simply-connected $4k$-manifolds}
\author{Csaba Nagy}
\email{nagy@mpim-bonn.mpg.de}
\address{Max Planck Institute for Mathematics, Vivatsgasse 7, 53111 Bonn, Germany}

\begin{abstract}
Kreck proved that two $2q$-manifolds are stably diffeomorphic if and only if they admit normally bordant normal $(q{-}1)$-smoothings over the same normal $(q{-}1)$-type $(B,\xi)$. We show that stable diffeomorphism can be replaced by diffeomorphism if the normal smoothings have isomorphic Q-forms (which consists of the intersection form of the manifold and the induced homomorphism on $H_q$), when the manifolds are simply-connected, $q=2k$ is even and $H_q(B)$ is free. This proves a special case of Crowley's Q-form conjecture. The basis of the proof is the construction of an \emph{extended surgery obstruction} associated to a normal bordism. 

As an application, we identify the inertia group of a $(2k{-}1)$-connected $4k$-manifold with the kernel of a certain bordism map. By the calculations of Senger-Zhang and earlier results, these kernels are now known in all cases. For $k=2,4$, the combination of these results determines the inertia groups. 

We also obtain, for a simply-connected $4k$-manifold $M$ with normal $(2k{-}1)$-type $(B,\xi)$ such that $H_{2k}(B)$ is free, an algebraic description of the stable class of $M$, that is, the set of diffeomorphism classes of manifolds stably diffeomorphic to $M$. Using this description, we explicitly compute the stable class of manifolds $M$ with rank-$2$ hyperbolic intersection form.
\end{abstract}

\maketitle

\section{Introduction}

\subsection{Main results}

Diffeomorphism classification of high-dimensional smooth manifolds is one of the basic problems of differential topology. While classical surgery theory provides a way for obtaining a diffeomorphism classification from a homotopy classification, the latter is only known for a few special classes of manifolds, leaving the problem largely open.

In the case of even dimensional manifolds, classification up to a weaker equivalence relation, stable diffeomorphism, has proven to be more attainable. Recall that two closed connected smooth $2q$-manifolds $M$ and $N$ are stably diffeomorphic if there are integers $k, l \geq 0$ such that $M \# k(S^q \times S^q)$ is diffeomorphic to $N \# l(S^q \times S^q)$. By Kreck's theorem \cite[Corollary 3]{kreck99} we have the following: 
\begin{thmx}[Kreck \cite{kreck99}]
A pair of $2q$-manifolds are stably diffeomorphic if and only if they have the same normal $(q{-}1)$-type and they admit bordant normal $(q{-}1)$-smoothings over it.
\end{thmx}

\noindent
(See \cite[Section 2]{kreck99} and Section \ref{ss:nm} for the definition of normal smoothings etc., and our conventions.)

In this paper we will consider simply-connected $2q$-manifolds, with $q \geq 2$ even. Our main result is that the addition of a single homological invariant is enough to obtain a classification up to diffeomorphism (h-cobordism, if $q=2$), when the following condition is satisfied:
\begin{flalign} \label{free}
\text{The normal $(q{-}1)$-type of the manifolds has free homology in degree $q$.} && \tag{$\ast$}
\end{flalign}
Namely, we define the \emph{Q-form} associated to a closed, oriented, simply-connected $2q$-manifold $M$, with $q$ even, and a map $f : M \rightarrow B$ (for example, when $f$ is a normal $(q{-}1)$-smoothing):
\[
E_q(M,f) = (H_q(M), \lambda_M, H_q(f))
\] 
which consists of the intersection form $\lambda_M : H_q(M) \times H_q(M) \rightarrow \Z$ of $M$ together with the induced homomorphism $H_q(f) : H_q(M) \rightarrow H_q(B)$ (see Section \ref{ss:qf}). This is an example of an \emph{extended quadratic form} over the abelian group $H_q(B)$ (see Section \ref{s:eqf}). Isomorphism of extended quadratic forms over a fixed abelian group is defined in the obvious way, and we have the following: 

\begin{thm} \label{thm:main}
Suppose that $q \geq 2$ is even. A pair of simply-connected $2q$-manifolds satisfying \eqref{free} are diffeomorphic (h-cobordant, if $q=2$) if and only if they have the same normal $(q{-}1)$-type and they admit bordant normal $(q{-}1)$-smoothings with isomorphic Q-forms over it.
\end{thm}

Theorem \ref{thm:main} follows from Theorem \ref{thm:qfc} below, which proves Crowley's Q-form conjecture (see \cite[Problem 11]{bowden-cdfrt20}) in the special case when $H_q(B)$ is free. In fact we get a stronger statement, also showing that any normal bordism between the normal smoothings is bordant to an h-cobordism.

\begin{thm}[Q-form conjecture for simply-connected $B$ and even $q$, with free $H_q(B)$] \label{thm:qfc}
Let $B$ be a simply-connected space, $\xi$ a stable bundle over it and $q \geq 2$ an even integer. Let $M_0$ and $M_1$ be two closed oriented $2q$-manifolds. Let $f_i : M_i \rightarrow B$ be a normal $(q{-}1)$-smoothing over $(B,\xi)$ ($i=0,1$), and $F_0 : W_0 \rightarrow B$ a normal bordism between $f_0$ and $f_1$. Suppose that $H_q(B)$ is free. If $E_q(M_0,f_0) \cong E_q(M_1,f_1)$, then $F_0$ is normally bordant (rel boundary) to a normal bordism $F' : W' \rightarrow B$ such that $W'$ is an h-cobordism. 
\end{thm}

\begin{rem} \label{rem:nec}
In the other direction, if there is a normal bordism $F' : W' \rightarrow B$ between $f_0$ and $f_1$ such that $W'$ is an h-cobordism, then the inclusions $M_i \rightarrow W'$ determine a homotopy equivalence $H : M_0 \rightarrow M_1$ such that $f_0 \simeq f_1 \circ H : M_0 \rightarrow B$. This shows that the condition $E_q(M_0,f_0) \cong E_q(M_1,f_1)$ is necessary.
\end{rem}

To prove Theorem \ref{thm:qfc} we develop a new version of surgery theory. This is the main achievement of the paper, and we will summarise it in Section \ref{ss:est} below. We will also present two applications of Theorem \ref{thm:qfc}. 

In the first application we consider the classification of $(q{-}1)$-connected $2q$-manifolds for $q=4,8$, more specifically, the computation of their inertia groups. Recall that $\Theta_n$ denotes the group of homotopy $n$-spheres, and the \emph{inertia group} of an $n$-manifold $M$ is the subgroup
\[
I(M) = \{ \Sigma \in \Theta_n \mid M \# \Sigma \approx M \} \leq \Theta_n \text{\,.}
\]
That is, the inertia group of a manifold is its stabiliser under the action of $\Theta_n$, which acts on $n$-manifolds by taking connected sum.

Wall \cite{wall62} started the classification of $(q{-}1)$-connected $2q$-manifolds based on the action of $\Theta_{2q}$, for $q \geq 3$. Two manifolds are in the same orbit of $\Theta_{2q}$ if and only if they become diffeomorphic after cutting out a disc $D^{2q}$. Thus the set of orbits can be understood by studying punctured manifolds, which have a relatively simple handlebody decomposition. Wall has identified a set of invariants that classify $(q{-}1)$-connected $2q$-manifolds with boundary a homotopy $(2q{-}1)$-sphere up to diffeomorphism. In particular, this includes the above mentioned punctured manifolds, hence the given invariants classify closed $(q{-}1)$-connected $2q$-manifolds up to connected sum with homotopy $2q$-spheres. For $q=4$ (resp.\ $8$) the relevant invariants are the intersection form together with the Pontryagin class $p_1$ (resp.\ $p_2$). 

In order to precisely characterise the set of orbits of the action of $\Theta_{2q}$, one needs to determine which of the possible values of the invariants can be realised by closed $(q{-}1)$-connected $2q$-manifolds. The answer is now known for all $q \neq 63$, see \cite{burklund-senger24} for an overview. The $q=4,8$ cases were solved by Wall \cite{wall62} as follows. There exists a closed $3$-connected $8$-manifold (resp.\ $7$-connected $16$-manifold) with a given intersection form of signature $\sigma$ and Pontryagin class $p_1$ (resp.\ $p_2$) if and only if $\frac{1}{8} \left[ \left( \frac{p_1}{2} \right)^{\smash{2}} - \sigma \right]$ (resp.\ $\frac{1}{8} \left[ \left( \frac{p_2}{6} \right)^{\smash{2}} - \sigma \right]$) is divisible by $28$ (resp.\ $8128$).

Finally, we need to understand each individual orbit. This requires computing the inertia groups of $(q{-}1)$-connected $2q$-manifolds, which we will discuss in more detail in Section \ref{ss:ig}. Early progress in this area showed that for most values of $q$ the inertia groups of all $(q{-}1)$-connected $2q$-manifolds are trivial (see \cite{wall62b}, \cite{stolz85}), while \cite{kramer-stolz07} proved that some $8$- and $16$-dimensional manifolds, including ${\quat}P^2$ and ${\oct}P^2$, have non-trivial inertia groups. Stable homotopy computations together with modified surgery theory have led to the determination of inertia groups for all $q \neq 4,8$, see \cite{senger-zhang25}.

By combining Theorem \ref{thm:qfc} and the calculations of Senger-Zhang \cite{senger-zhang25}, we can determine the inertia groups for $q = 4,8$ as follows. See Theorems \ref{thm:bordism-ker}--\ref{thm:inert-ker} for more details, \cite[Section 7]{senger-zhang25} for the treatment of Senger-Zhang, and the announcement \cite{ann-3c8m} in the $q=4$ case.

\begin{thm} \label{thm:inert-pontr}
The inertia group of a $3$-connected $8$-manifold (resp.\ $7$-connected $16$-manifold) is trivial if and only if its Pontryagin class $p_1$ (resp.\ $p_2$) is divisible by $8$ (resp.\ $24$). 
\end{thm}

\noindent
Note that $\Theta_8 \cong \Theta_{16} \cong \Z_2$, so when the inertia group is non-trivial, it is equal to $\Theta_8$ or $\Theta_{16}$. 
Theorem \ref{thm:inert-pontr} thus completes the final step in the classification of $3$-connected $8$-manifolds and $7$-connected $16$-manifolds.

In the second application we consider the \emph{stable class} $\Str^{\st}(M)$ of a manifold $M$, see Section \ref{ss:sc}. It is the set of diffeomorphism classes (s-cobordism classes in dimension $4$) within the stable diffeomorphism class of $M$. In other words, the stable class measures the difference between the stable diffeomorphism and diffeomorphism classifications. Theorem \ref{thm:main} shows that the Q-form is the extra (and only) invariant responsible for this difference, and this allows us to obtain an algebraic description of the stable class, see Theorem \ref{thm:sc-bij}. To illustrate the use of this description in explicit computations, we consider manifolds whose intersection form is hyperbolic of rank $2$. 

\begin{thm} \label{thm:hyp-sss}
Let $M$ be a simply-connected $2q$-manifold, with $q \geq 2$ even. Let $(B,\xi)$ be its normal $(q{-}1)$-type, and $f : M \rightarrow B$ a normal $(q{-}1)$-smoothing. Suppose that $H_q(B)$ is free, $\rk H_q(M)=2$ and the induced intersection form on $H_q(M) / \Tor(H_q(M)) \cong \Z^2$ is isomorphic to $\bigl[ \begin{smallmatrix} 0 & 1 \\ 1 & 0 \end{smallmatrix} \bigr]$. Then
\[
|\Str^{\st}(M)| = 
\begin{cases}
1 & \text{if $\rk H_q(B)=0$ or $2$} \\
1 & \text{if $\rk H_q(B)=1$ and $|ab| \leq 1$} \\
2^{r-1} & \text{if $\rk H_q(B)=1$ and $|ab| \geq 2$}
\end{cases}
\]
where in the $\rk H_q(B)=1$ case we fix an identification $H_q(B) \cong \Z$ and elements $x,y \in H_q(M)$ such that $[x],[y]$ is a basis of $H_q(M) / \Tor(H_q(M))$ in which the intersection form is given by the matrix $\bigl[ \begin{smallmatrix} 0 & 1 \\ 1 & 0 \end{smallmatrix} \bigr]$, we define $a = f_*(x), b=f_*(y) \in H_q(B) \cong \Z$, and $r$ is the number of primes dividing $|ab|$.
\end{thm}

Conway-Crowley-Powell-Sixt \cite[Proposition 2.2, Theorem 3.3]{CCPS23} constructed two families of manifolds with rank-$2$ hyperbolic intersection form, $N_{a,b}$ and $M_{a,b}$, and gave bounds for the size of their stable class. The manifolds $N_{a,b}$ are simply-connected, $4k$-dimensional and stably parallelisable with homology concentrated in even dimensions, while the $M_{a,b}$ are $(4k-1)$-connected $8k$-manifolds. Theorem \ref{thm:hyp-sss} strengthens these results by determining the exact size of the stable class, and covering a more general class of manifolds.

\subsection{Background}

Kreck's theorem is the basis of a strategy for classifying even-dimensional manifolds up to diffeomorphism, which starts with a stable diffeomorphism classification. Explicit stable diffeomorphism classification has been obtained for some classes of manifolds, for example see Hambleton-Kreck-Teichner \cite{HKT09}, Kasprowski-Land-Powell-Teichner \cite{KLPT17} for $4$-manifolds with certain fundamental groups. Then the next step is understanding manifolds in a given stable diffeomorphism class, ie.\ determining the stable class. The main tool for this is Kreck's modified surgery obstruction (see \cite[Theorem 4]{kreck99}).

In some cases it can be used to show that the diffeomorphism classification is determined by the stable diffeomorphism classification. By \cite[Theorem D]{kreck99} two simply-connected $2q$-manifolds, with $q$ odd, are diffeomorphic if and only if they are stably diffeomorphic and have the same Euler-characteristic. This is not true for arbitrary fundamental group and $q$. In general, the stable class can be arbitrarily large, or infinite, as shown by examples constructed in Conway-Crowley-Powell-Sixt \cite{CCPS23}, \cite{CCPS21} (including the manifolds $N_{a,b}$ and $M_{a,b}$ mentioned above).

Nevertheless, under some assumptions, the stable diffeomorphism classification still determines the diffeomorphism classification for manifolds that are already sufficiently stabilised. It follows from \cite[Corollary 4]{kreck99} that stably diffeomorphic $2q$-manifolds with finite fundamental group and the same Euler-characteristic are in fact diffeomorphic if they split off two copies of $S^q \times S^q$, or just one copy, if $q$ is even and the manifolds are simply-connected. These cancellation results have been generalised to other fundamental groups (see eg.\ Crowley-Sixt \cite{crowley-sixt11}), and also strengthened in the special case of $4$-manifolds (where diffeomorphism is replaced by s-cobordism or homeomorphism), see eg.\ Hambleton-Kreck \cite{hambleton-kreck93b}, Hambleton \cite{hambleton23}. 

There are few results where unstable classification has been obtained from the stable classification without assuming that the manifolds are already stabilised. As, in general, the stable and unstable classifications do not coincide, additional invariants are needed for the latter. Wall \cite{wall64}, Freedman-Quinn \cite{freedman-quinn90} and Hambleton-Kreck \cite{hambleton-kreck93c} showed that stably homeomorphic $4$-manifolds with trivial or cyclic fundamental group are classified up to s-cobordism/homeomorphism by their equivariant intersection form. However, by Kreck-Schafer \cite{kreck-schafer84}, this is not true for arbitrary $4$-manifolds. Our Theorem \ref{thm:main} provides a new condition of this type for simply-connected $2q$-manifolds satisfying \eqref{free}, for all even $q \geq 2$ (which, in the $q=2$ case, recovers the earlier results on spin simply-connected smooth $4$-manifolds). This can be used to complete the second step of this classification strategy and determine the stable class of a manifold.

\subsection{Extended surgery theory} \label{ss:est}

We will prove Theorem \ref{thm:qfc} (the special case of the Q-form conjecture), and hence Theorem \ref{thm:main}, by introducing and studying a new type of surgery obstruction associated to a normal bordism.

We start the proof by applying surgery below the middle dimension to replace $F_0$ with a $q$-connected normal bordism $F : W \rightarrow B$. Note that for this $F$ Kreck's modified surgery obstruction is defined \cite[Section 6]{kreck99}, with the property that the obstruction is elementary if and only if $W$ is bordant (over $(B,\xi)$, rel boundary) to an h-cobordism. It is defined using the intersection form on $H_q(\partial U)$, where $U \subset W$ is a collection of embedded copies of $S^q \times D^{q+1}$ such that $\Image(H_q(U) \rightarrow H_q(W))$ is the surgery kernel $\Ker(H_q(F) : H_q(W) \rightarrow H_q(B))$. Then \cite[Proposition 8]{kreck99} relates the obstruction to the intersection form on $\Ker(H_q(f_0) : H_q(M_0) \rightarrow H_q(B))$. This is the basis of \cite[Theorem D]{kreck99} and other cancellation results. 

That is not what we will use though, instead, we will define a new obstruction $[\theta_{W,F}]$ (see Section \ref{s:obstr-def}). While the definition of $[\theta_{W,F}]$ will follow the overall logic of \cite{kreck99}, we will make significant changes. The main novelty of $[\theta_{W,F}]$, compared to the modified surgery obstruction and the obstructions in classical surgery theory (see Browder \cite{browder72}, Wall \cite{wall-scm}), is that it is no longer defined on the surgery kernel. Quite the opposite, we will use a codimension-$0$ submanifold $U \subset W$ such that the map $H_q(U) \rightarrow H_q(W)$ (and hence the composition $H_q(U) \rightarrow H_q(W) \rightarrow H_q(B)$) is surjective. Furthermore, the monoid $\ell_{2q+1}(B,\xi)$ in which $[\theta_{W,F}]$ lives is no longer determined by the fundamental group alone, but also depends on $H_q(B)$ and the Wu class $v_q(\xi)$. Its elements are represented by quasi-formations on extended quadratic forms over $H_q(B)$ (see Sections \ref{s:eqf} and \ref{s:l-mon}). 

We will define a group $L_{2q+1}(B,\xi)$, which is the analogue of the classical L-group, and prove that it is the group of invertible elements of $\ell_{2q+1}(B,\xi)$. This will rely on a generalisation of Wall's \cite[Lemma 6.2]{wall-scm} to metabolic forms (see Theorem \ref{thm:ru-wall}). Moreover, we will show that if $H_q(B)$ is free, then $L_{2q+1}(B,\xi)$ is trivial.

We will also define the subsemigroup of elementary elements of $\ell_{2q+1}(B,\xi)$, and prove that $F$ is normally bordant to an h-cobordism if and only if $[\theta_{W,F}]$ is elementary. Along with its definition, this step is where the main technical difficulty of working with $[\theta_{W,F}]$ lies (see Lemma \ref{lem:theta-surg-real}), due to our choice of a more complicated $U$. 

Then, similarly to \cite[Proposition 8]{kreck99}, we will relate $[\theta_{W,F}]$ to the the Q-form of $f_0$. However, since the intersection form of $M_0$ is nonsingular (in contrast to its restriction to $\Ker H_q(f_0)$), we will find that $[\theta_{W,F}]$ has better algebraic properties. In fact, under some assumptions, we will be able to compute $[\theta_{W,F}]$ in terms of the Q-forms of $f_0$ and $f_1$, up to the action of $L_{2q+1}(B,\xi)$ (see Theorem \ref{thm:theta-calc}). In the case when $H_q(B)$ is free and the Q-forms are isomorphic, this will allow us to deduce that $[\theta_{W,F}]$ is elementary and complete the proof of Theorem \ref{thm:qfc} (see Section \ref{s:qfc}).

\subsection{Inertia groups of highly-connected manifolds} \label{ss:ig}

As an application, next we consider the inertia groups of $(q{-}1)$-connected $2q$-manifolds. When $q=4$ or $8$, we will use Theorem \ref{thm:qfc} to identify the inertia group with the kernel of a bordism map.

Given an oriented manifold $M$, let $\nu_M$ denote the stable normal bundle of $M$, and also its classifying map $\nu_M : M \rightarrow BSO$. If $M$ is $(q{-}1)$-connected and $A = \Image(\pi_q(\nu_M): \pi_q(M) \rightarrow \pi_q(BSO))$, then the normal $(q{-}1)$-type of $M$ is $(BSO\left< q, A\right>,\xi_{q,A})$, see Definition \ref{def:BSO-qa} and Proposition~\ref{prop:nm-unique}.

\begin{defin} \label{def:eta-A}
For a subgroup $A \leq \pi_q(BSO)$ let $\eta_{2q}^A : \Theta_{2q} \rightarrow \Omega_{2q}(BSO\left< q, A\right>; \xi_{q,A})$ be the map which sends a homotopy $2q$-sphere $\Sigma$ to the bordism class of a normal map $\Sigma \rightarrow BSO\left< q, A\right>$. (This map is a well-defined homomorphism by Propositions \ref{prop:nm-unique} and \ref{prop:eta1}.)
\end{defin}

It follows from a standard obstruction theory argument that if $M$ is a $(q{-}1)$-connected $2q$-manifold and $A = \Image(\pi_q(\nu_M))$, then $I(M) \leq \Ker(\eta_{2q}^A)$, see Proposition \ref{prop:inert-upper}. Based on this observation, the inertia group $I(M)$ can be determined by completing the following two steps: 
\begin{compactenum}[1.]
\item Computing $\Ker(\eta_{2q}^A)$. 
\item Showing that $I(M) = \Ker(\eta_{2q}^A)$.
\end{compactenum}

The subgroup $\Ker(\eta_{2q}^A)$ is now known for every $q$ and $A$, due to the work of several authors. The final cases were resolved by Senger-Zhang \cite{senger-zhang25}, see \cite{senger-zhang25} also for the history and earlier results. In particular, $\Ker(\eta_{2q}^A) = 0$ for every $A$ if $q \neq 4,8,9$, and then the inertia group is also trivial by the relation $I(M) \leq \Ker(\eta_{2q}^A)$. Moreover, Kreck's \cite[Theorem D]{kreck99} implies that if $q$ is odd, then $I(M) = \Ker(\eta_{2q}^A)$, see Theorem \ref{thm:inert-lower} a). Hence the second step is settled and the inertia group is determined by the above results in all cases, except when $q=4,8$. 

In these cases $\Ker(\eta_{2q}^A)$ is computed as follows: 

\begin{thm}[$q=4$: Senger-Zhang {\cite[Theorem 7.14]{senger-zhang25}}, Crowley-Nagy \cite{ann-3c8m}, $q=8$: Senger-Zhang {\cite[Theorem 7.15]{senger-zhang25}}, Crowley-Olbermann] \label{thm:bordism-ker}
Suppose that $q = 4$ or $8$ and $A \leq \pi_q(BSO) \cong \Z$. Then $\Ker(\eta_{2q}^A) = 0$ if $A \leq 4\Z$ and $\Ker(\eta_{2q}^A) = \Theta_{2q} \cong \Z_2$ if $A \not\leq 4\Z$.
\end{thm}

Now the Q-form conjecture can be used to show that $I(M) = \Ker(\eta_{2q}^A)$, see Theorem \ref{thm:inert-lower} b). Namely, given a $(q{-}1)$-connected $2q$-manifold $M$ with $A = \Image(\pi_q(\nu_M))$, if $\Sigma \in \Ker(\eta_{2q}^A)$, then the (unique) normal $(q{-}1)$-smoothings of $M$ and $M \# \Sigma$ over $(BSO\left< q, A\right>, \xi_{q,A})$ are normally bordant. If $q$ is even, then their Q-forms are defined and isomorphic, and, unless $q \equiv 2$ mod $8$, every subgroup of $\pi_q(BSO)$ is free, so the manifolds satisfy \eqref{free}. Therefore by Theorem \ref{thm:qfc}, $M$ is diffeomorphic to $M \# \Sigma$. These arguments are summarised as follows: 

\begin{thm} \label{thm:inert-ker}
Suppose that $q \equiv 0$, $4$ or $6$ mod $8$. Let $M$ be a $(q{-}1)$-connected $2q$-manifold and $A = \Image(\pi_q(\nu_M): \pi_q(M) \rightarrow \pi_q(BSO))$. Then 
\[
I(M) = \Ker(\eta_{2q}^A : \Theta_{2q} \rightarrow \Omega_{2q}(BSO\left< q, A\right>; \xi_{q,A})) .
\]
\end{thm}

So if $q=4,8$, then the combination of Theorems \ref{thm:bordism-ker} and \ref{thm:inert-ker} determines the inertia groups. Theorem \ref{thm:inert-pontr} is then obtained by rephrasing the condition of Theorem \ref{thm:bordism-ker} on $A = \Image \pi_q(\nu_M)$ in terms of a Pontryagin class of $M$ (see Section \ref{s:inertia}).

\subsection{The stable smoothing set and stable class} \label{ss:sc}

We can also apply Theorem \ref{thm:qfc} to compute the stable smoothing set and the stable class of a manifold. Fix a $2q$-manifold $M$, let $(B,\xi)$ be its normal $(q{-}1)$-type, and let $f : M \rightarrow B$ be a normal $(q{-}1)$-smoothing. 

\begin{defin}
The \emph{stable class} of $M$, denoted $\Str^{\st}(M)$, is the set of s-cobordism classes of $2q$-manifolds $N$ such that $N$ is stably diffeomorphic to $M$ and $\chi(N)=\chi(M)$ (where $\chi(M)$ denotes the Euler-characteristic of $M$). 

The \emph{stable smoothing set} of $M$, denoted $\Str^{\st}(M,f)$, consists of equivalence classes of normal $(q{-}1)$-smoothings $g : N \rightarrow B$ such that $g$ is normally bordant to $f$ and $\chi(N)=\chi(M)$. Two such normal $(q{-}1)$-smoothings, $g : N \rightarrow B$ and $g' : N' \rightarrow B$, are equivalent if there is a normal bordism $G : W \rightarrow B$ between them such that $W$ is an s-cobordism between $N$ and $N'$.
\end{defin}

Note that the normal $(q{-}1)$-smoothing $f$ is unique up to an automorphism of $(B,\xi)$ (see Section \ref{ss:nm}), so if $f' : M \rightarrow B$ is another normal $(q{-}1)$-smoothing, then $\Str^{\st}(M,f) \cong \Str^{\st}(M,f')$.

By Kreck's theorem a manifold $N$ represents an element of $\Str^{\st}(M)$ if and only if there is a normal map $g : N \rightarrow B$ that represents an element of $\Str^{\st}(M,f)$. Hence forgetting the normal map defines a surjective map $\Str^{\st}(M,f) \rightarrow \Str^{\st}(M)$. Since for a fixed $N$ the normal map $g : N \rightarrow B$ is unique up to an automorphism of $(B,\xi)$, the fibres of the forgetful map $\Str^{\st}(M,f) \rightarrow \Str^{\st}(M)$ are the orbits of the action of $\Aut(B,\xi)_{[M,f]}$ on $\Str^{\st}(M,f)$, where $\Aut(B,\xi)_{[M,f]}$ denotes the stabiliser of the normal bordism class of $f : M \rightarrow B$ in $\Aut(B,\xi)$. Therefore $\Str^{\st}(M) \cong \Str^{\st}(M,f) / \Aut(B,\xi)_{[M,f]}$. An automorphism of $(B,\xi)$ determines an automorphism of $H_q(B)$, and we will denote the image of the map $\Aut(B,\xi)_{[M,f]} \rightarrow \Aut(H_q(B))$ by $A^{(B,\xi)}_{[M,f]}$.

From now on assume that $q$ is even and $M$ is simply-connected. We will next consider extended quadratic forms over abelian groups, see Section \ref{s:eqf} for the definitions. For example, the Q-form $E_q(N,g)$ of a normal $(q{-}1)$-smoothing $g : N \rightarrow B$ is an extended quadratic form over $H_q(B)$.

\begin{defin}
Let $\UE$ be an extended quadratic form over an abelian group $Q$. We define $\SI(\UE)$ to be the set of isomorphism classes of extended quadratic forms $\UEp$ over $Q$ such that $\UE \oplus \UH_{2k} \cong \UEp \oplus \UH_{2k}$ for some $k \geq 0$.
\end{defin}

\begin{thm} \label{thm:sc-bij}
a) Taking the Q-form of a normal $(q{-}1)$-smoothing induces a well-defined surjection $E_q : \Str^{\st}(M,f) \rightarrow \SI(E_q(M,f))$. If $H_q(B)$ is free, then this map is a bijection. 

b) The group $A^{(B,\xi)}_{[M,f]}$ acts on $\SI(E_q(M,f))$, and the map $E_q$ induces a surjection $\Str^{\st}(M) \rightarrow \SI(E_q(M,f)) / A^{(B,\xi)}_{[M,f]}$. If $H_q(B)$ is free, then this map is a bijection. 
\end{thm}

Thus, when $H_q(B)$ is free, the Q-form conjecture reduces the computation of the stable smoothing set and the stable class to algebraic problems. The former only depends on the Q-form of a normal $(q{-}1)$-smoothing $f : M \rightarrow B$, for the latter we further need to know the subgroup $A^{(B,\xi)}_{[M,f]} \leq \Aut(H_q(B))$ and understand its action on $\SI(E_q(M,f))$. 
Theorem \ref{thm:sc-bij} will be proved in Section \ref{s:sss}. There we also prove Theorem \ref{thm:hyp-sss} by solving the corresponding algebraic problem.

\quad

The paper is organised as follows. Sections \ref{s:eqf}--\ref{s:l-mon} are concerned with the algebra underlying our constructions, introducing extended quadratic forms over an abelian group, the group $\RU_{\st}(\UM,L)$ and the monoid $\ell_{2q+1}(Q,v)$, respectively. Section \ref{s:nm-qf} contains preliminaries for the geometric results. In Section \ref{s:obstr-def} we define the extended surgery obstruction associated to a normal bordism, and prove some of its main properties. In Section \ref{s:obstr-comp} we give a formula for the surgery obstruction in terms of the Q-forms of $f_0$ and $f_1$. In Section \ref{s:qfc} we apply the results of the earlier sections to prove Theorem \ref{thm:qfc}. Finally, Sections \ref{s:inertia} and \ref{s:sss} contain the proofs of Theorems \ref{thm:inert-ker}, \ref{thm:inert-pontr}, \ref{thm:sc-bij} and \ref{thm:hyp-sss}. 

This paper is based on Chapter 4 of the author's thesis \cite{csn-thesis}. Here we will often omit straightforward arguments, and refer to \cite{csn-thesis} for those details. Compared to \cite{csn-thesis}, a new Section \ref{s:ru} is added, introducing the groups $\RU_{\st}(\UM,L)$, which allow us to prove Theorem \ref{thm:jacobi} (\cite[Theorem 4.2.21]{csn-thesis}) without the assumption that $Q$ is free. Sections \ref{s:l-mon} and \ref{s:obstr-comp} are updated accordingly. Section \ref{s:inertia} is based on \cite[Section 5.2.3]{csn-thesis}.

All abelian groups will be assumed to be finitely generated and all manifolds will be assumed to be smooth, compact (closed, unless indicated otherwise) and oriented.

\subsection*{Acknowledgements}

I would like to thank my PhD supervisor, Diarmuid Crowley, for his ongoing support and helpful comments on the paper. I am also grateful to Mark Powell for discussions on the groups $\RU_{\st}(\UM,L)$. I thank the anonymous referee for the many helpful suggestions. 
I was supported by the Melbourne Research Scholarship and the EPSRC New Investigator grant EP/T028335/2.

\section{Extended quadratic forms} \label{s:eqf}

In this section we define and study (symmetric) extended quadratic forms over an abelian group. 

\begin{defin}
Let $Q$ be an abelian group. An \emph{extended quadratic form over $Q$} is a triple
\[
\UM = (M, \lambda, \mu)
\]
where $M$ is an abelian group, $\lambda : M \times M \rightarrow \Z$ is a symmetric bilinear function, and $\mu : M \rightarrow Q$ is a homomorphism. 
\end{defin}

\begin{defin}
Let $\UM$ be an extended quadratic form as above. 
\begin{compactitem}
\item $\UM$ is called \emph{free}, if the abelian group $M$ is free.
\item $\UM$ is \emph{nonsingular}, if the adjoint of $\lambda$, $\ad \lambda : M \rightarrow \Hom(M, \Z) = M^*$ induces an isomorphism $M / \Tor M \cong M^*$.
\item A subgroup $V \leq M$ is called a \emph{half-rank direct summand}, if it is a direct summand in $M$ and $\rk V = \frac{1}{2} \rk M$.
\item A subgroup $L \leq M$ is a \emph{free lagrangian}, if it is a free half-rank direct summand in $M$, $\lambda \big| _{L \times L} = 0$ and $\mu \big| _L = 0$. 
\item A subgroup $L \leq M$ is a \emph{T-lagrangian}, if it is a half-rank direct summand in $M$, $\Tor M \leq L$, $\lambda \big| _{L \times L} = 0$ and $\mu \big| _L = 0$. 
\item $\UM$ is \emph{metabolic}, if it is nonsingular and there is a T-lagrangian $L \leq M$. 
\item $\UM$ is \emph{hyperbolic}, if it is free, nonsingular and there are two lagrangians $L, L' \leq M$ such that $M = L \oplus L'$ (as an internal direct sum).
\item $\UM$ is \emph{full}, if $\mu$ is surjective.
\item $\UM$ is \emph{even}, if $\lambda(x,x)$ is even for every $x \in M$.
\item $\UM$ is \emph{geometric} with respect to a homomorphism $v : Q \rightarrow \Z_2$, if for every $x \in M$ we have $\varrho_2 \circ \lambda(x,x) = v \circ \mu(x)$, where $\varrho_2 : \Z \rightarrow \Z_2$ is reduction mod $2$.
\item The \emph{rank} of $\UM$, denoted by $\rk \UM$, is the rank of the group $M$. 
\end{compactitem}
\end{defin}

If $\UM$ is free, T-lagrangians coincide with free lagrangians, and we simply call them lagrangians.

\begin{rem} \label{rem:eqf}
Baues \cite{baues94} and Ranicki \cite{ranicki01} introduced the notion of an extended quadratic form over a quadratic form parameter $P$. Using the terminology and notation from Crowley-Nagy \cite{dc-csn-eqf}, extended quadratic forms over an abelian group $A$, as defined here, are equivalent to extended quadratic forms over the quadratic form parameter $Q^+ \oplus A$. More specifically, the form $(M, \lambda, \mu)$ over $A$ corresponds to the form $(M, \lambda, (q_{\lambda},\mu))$ over $Q^+ \oplus A$, where $q_{\lambda} : M \rightarrow \Z$ is defined by $q_{\lambda}(x) = \lambda(x,x)$. 

Given a symmetric quadratic form parameter $P$, there is a canonical ``extended symmetrisation" morphism of quadratic form parameters $P \rightarrow Q^+ \oplus SP$. It induces a map that sends forms over $P$ to forms over $Q^+ \oplus SP$. By \cite[Lemma 2.29]{dc-csn-eqf} this map is a bijection between forms over $P$ and those forms over $Q^+ \oplus SP$ that correspond to geometric (with respect to $v_P : SP \rightarrow \Z_2$) forms over the abelian group $SP$. For example, ordinary quadratic forms used in classical simply-connected surgery are forms over the quadratic form parameter $Q_+$, and they are equivalent to geometric extended quadratic forms over the abelian group $SQ_+ = 0$ (with respect to $v_{Q_+} = 0 : 0 \rightarrow \Z_2$).
\end{rem}

\subsection{Preliminaries}

In this section we will collect a few lemmas which are well-known for ordinary ``symmetric forms"  (ie.\ pairs $(M, \lambda)$, where $M$ is an abelian group and $\lambda : M \times M \rightarrow \Z$ is a symmetric bilinear function), see eg.\ \cite{milnor-husemoller73}, and can be directly applied in the extended setting (see \cite[Lemmas 4.1.10--13]{csn-thesis}). Before those, we recall the following elementary lemma, which we will often use to check that a subgroup of an abelian group is a direct summand. 

\begin{lem} \label{lem:summand}
Let $A$ be an abelian group and $B \leq A$ a subgroup. The following are equivalent:
\begin{compactenum}
\item $B$ is a direct summand in $A$ and $\Tor A \leq B$.
\item $A / B$ is free.
\item For every $a \in A$ and $k \in \Z \setminus \{ 0 \}$ if $ka \in B$ then $a \in B$. \qed
\end{compactenum}
\end{lem}

\begin{defin}
Let $\UM = (M, \lambda, \mu)$ be a nonsingular form, and $X \subseteq M$ a subset. Then $X^{\perp}$ denotes the subgroup $\{ y \in M \mid \lambda(x,y) = 0 \text{ for all $x \in X$} \}$.
\end{defin}

For any subset $X$, the subgroup $X^{\perp}$ contains $\Tor M$ and it is a direct summand in $M$. Moreover, if $X$ is a subgroup of $M$, then $\rk X + \rk X^{\perp} = \rk M$. Therefore, if $X$ is a subset of $M$ such that $\mu \big| _X = 0$, then $X$ is a T-lagrangian if and only if $X = X^{\perp}$.

\begin{lem} \label{lem:nonsing-summand}
Let $\UM = (M, \lambda, \mu)$ be a free nonsingular form. If $X \leq M$ is a subgroup such that $\lambda \big| _{X \times X}$ is nonsingular, then $M = X \oplus X^{\perp}$ (as an internal direct sum).
\qed
\end{lem}

\begin{lem} \label{lem:perp-nonsing}
Let $\UM = (M, \lambda, \mu)$ be a free nonsingular form. Let $X, Y \leq M$ be subgroups such that $M = X \oplus Y$ (as an internal direct sum). Then 

a) $\ad \lambda$ defines an isomorphism $Y^{\perp} \cong X^*$.

b) $M = X^{\perp} \oplus Y^{\perp}$ (as an internal direct sum). 

c) And if $X \leq X^{\perp}$ (hence $X \oplus Y^{\perp}$ is a subgroup of $M$ by part b)), then $\lambda \big| _{(X \oplus Y^{\perp}) \times (X \oplus Y^{\perp})}$ is nonsingular.
\qed
\end{lem}

\begin{lem} \label{lem:lagr-dual}
Let $\UM = (M, \lambda, \mu)$ be a metabolic form. If $L \leq M$ is a T-lagrangian and $N \leq M$ is a direct complement of $L$, then $\ad \lambda$ defines an isomorphism $N \cong L^*$.
\qed
\end{lem}

\begin{lem} \label{lem:met-basis}
Let $\UM = (M, \lambda, \mu)$ be a free metabolic form of rank $2k$, and let $L \leq M$ be a lagrangian. Then $M$ has a basis $e_1, e_2, \ldots , e_k, f_1, f_2, \ldots , f_k$ such that $L = \left< e_1, e_2, \ldots , e_k \right>$, and in this basis $\lambda$ is given by the block matrix
\[
\begin{bmatrix}
0 & I_k \\
I_k & D
\end{bmatrix}
\]
where $I_k$ is the $k \times k$ identity matrix and $D$ is a diagonal matrix such that each of its entries is either $0$ or $1$.
\qed
\end{lem}

\subsection{Basic constructions} \label{ss:constr}

\begin{defin}
Let $Q$ and $N$ be abelian groups, let $\UM = (M, \lambda, \mu)$ be an extended quadratic form over $Q$, and let $h : N \rightarrow M$ be a homomorphism. The \emph{pullback} of $\UM$ by $h$ is 
\[
h^*(\UM) = (N, h^*\lambda, h^*\mu)
\] 
where $h^*\lambda(x,y) = \lambda(h(x),h(y))$ and $h^*\mu = \mu \circ h$, this is also an extended quadratic form over $Q$. 
\end{defin}

\begin{defin} \label{def:eqf-mor}
Let $\UM$ and $\UN$ be extended quadratic forms over some abelian group $Q$. A \emph{morphism} $h : \UN \rightarrow \UM$ of extended quadratic forms is a homomorphism $h : N \rightarrow M$ such that $\UN = h^*(\UM)$. 
\end{defin}

If $\UM = (M, \lambda_M, \mu_M)$ and $\UN = (N, \lambda_N, \mu_N)$, then $\UN = h^*(\UM)$ if and only if $\lambda_M(h(x),h(y)) = \lambda_N(x,y)$ and $\mu_M(h(x)) = \mu_N(x)$ for every $x,y \in N$. 

Composition of morphisms can be defined in the obvious way, and extended quadratic forms over a fixed group $Q$ together with the above defined morphisms form a category. In this category a morphism $h : \UN \rightarrow \UM$ is an isomorphism if and only if the underlying homomorphism $h : N \rightarrow M$ is an isomorphism.

\begin{defin}
The \emph{standard rank-$2k$ hyperbolic form} is $\UH_{2k} = (\Z^{2k}, \bigl[ \begin{smallmatrix} 0 & I_k \\ I_k & 0 \end{smallmatrix} \bigr], 0)$. 
\end{defin}

Since the third component of $\UH_{2k}$ is zero, it can be regarded as an extended quadratic form over any abelian group $Q$.

\begin{lem} \label{lem:hyp-equiv}
Let $\UM = (M, \lambda, \mu)$ be an extended quadratic form over some abelian group $Q$. The following are equivalent: 
\begin{compactenum}
\item \label{el:he1} $\UM$ is hyperbolic. 
\item \label{el:he2} $\UM$ is free, metabolic, even and $\mu = 0$.
\item \label{el:he3} $\UM$ is isomorphic to $\UH_{2k}$, where $2k = \rk \UM$. 
\end{compactenum}
\end{lem}

\begin{proof}
$\eqref{el:he1} \Rightarrow \eqref{el:he2} \Rightarrow \eqref{el:he3} \Rightarrow \eqref{el:he1}$. The $\eqref{el:he2} \Rightarrow \eqref{el:he3}$ step follows from Lemma \ref{lem:met-basis}.
\end{proof}

\begin{rem}
If $Q=0$, then automatically $\mu=0$, and conversely if $\mu=0$, then $Q$ can be replaced with $0$. Further, an extended quadratic form $\UM$ over $0$ is even if and only if it is geometric (ie.\ it corresponds to an ordinary quadratic form). Thus Lemma \ref{lem:hyp-equiv} adapts the classical result that an ordinary quadratic form is hyperbolic if and only if it is metabolic (see eg.\ \cite[Theorem 11.51]{ranicki02}) to the setting of extended quadratic forms.
\end{rem}

\begin{defin}
Let $\UM = (M, \lambda_M, \mu_M)$ and $\UN = (N, \lambda_N, \mu_N)$ be extended quadratic forms over some abelian group $Q$. Their \emph{direct sum} is 
\[
\UM \oplus \UN = (M \oplus N, \lambda_M \oplus \lambda_N, \mu_M \oplus \mu_N)
\] 
where $(\lambda_M \oplus \lambda_N)((x_1,y_1),(x_2,y_2)) = \lambda_M(x_1,x_2) + \lambda_N(y_1,y_2)$ and $(\mu_M \oplus \mu_N)(x,y) = \mu_M(x) + \mu_N(y)$, this is also an extended quadratic form over $Q$.  
\end{defin}

\begin{defin}
Let $\UM = (M, \lambda, \mu)$ be an extended quadratic form over some abelian group $Q$. We define
\[
\begin{aligned}
-\UM &= (M, -\lambda, -\mu) \\
\UMs &= (M, \lambda, -\mu)
\end{aligned}
\] 
where $(-\lambda)(x,y) = -\lambda(x,y)$ and $(-\mu)(x) = -\mu(x)$.  
\end{defin}

If $\UM$ and $\UN$ are free/nonsingular/metabolic/hyperbolic/even/geometric (with respect to some homomorphism), then the same property also holds for $\UM \oplus \UN$, $-\UM$ and $\UMs$. If $\UM$ is full, then $\UM \oplus \UN$, $-\UM$ and $\UMs$ are full too.

\begin{defin}
Let $I : A \rightarrow B$ be an isomorphism of abelian groups. We define the \emph{diagonal} and \emph{anti-diagonal with respect to $I$}:  
\[
\begin{aligned}
\Delta_I &= \{ (x,I(x)) \mid x \in A \} \leq A \oplus B \\
\Delta^*_I &= \{ (x,-I(x)) \mid x \in A \} \leq A \oplus B
\end{aligned}
\] 
\end{defin}

\begin{lem} \label{lem:diag-lagr}
Let $\UM = (M, \lambda_M, \mu_M)$ and $\UN = (N, \lambda_N, \mu_N)$ be free extended quadratic forms over some abelian group $Q$ and let $I : \UM \rightarrow \UN$ be an isomorphism. Then 

a) $\Delta_I$ is a lagrangian in $\UM \oplus (-\UN)$.

b) $\Delta^*_I$ is a lagrangian in $\UM \oplus (-\UNs)$.
\qed
\end{lem}

Verifying that $\Delta_I$ and $\Delta^*_I$ have the required properties is straightforward; or see \cite[Lemma 4.1.22]{csn-thesis}. See also \cite[Lemmas 4.1.23-25]{csn-thesis} for more detailed proofs of Lemmas \ref{lem:met-star-isom}-\ref{lem:met-isom} below.

\begin{lem} \label{lem:met-star-isom}
Let $\UM = (M, \lambda, \mu)$ be an extended quadratic form. The homomorphism $-\id : \UM \rightarrow \UMs$ is an isomorphism of extended quadratic forms. 
\qed
\end{lem}

\begin{lem} \label{lem:met-neg-isom}
Let $\UM = (M, \lambda, \mu)$ be a free metabolic form, and let $L \leq M$ be a lagrangian. There is an isomorphism $J : \UM \rightarrow -\UM$ such that $J \big| _L = \id_L$. 
\end{lem}

\begin{proof}
Let $e_1, e_2, \ldots , e_k, f_1, f_2, \ldots , f_k$ be a basis of $M$ given by Lemma \ref{lem:met-basis}, and let $d_i = \lambda(f_i,f_i)$ (so $d_1, d_2, \ldots , d_k$ are the diagonal entries of $D$). Let $f'_i = d_ie_i - f_i$, and define $J : M \rightarrow M$ by setting $J(e_i)=e_i$ and $J(f_i)=f'_i$. 
\end{proof}

\begin{lem} \label{lem:met-isom}
Let $\UM = (M, \lambda, \mu)$ be a free metabolic form of rank $2k$, and let $L \leq M$ be a lagrangian. There is an isomorphism $I : \UM \oplus \UM \rightarrow \UM \oplus \UH_{2k}$ such that $I(L \oplus L) = L \oplus (\{ 0 \} \times \Z^k)$. 
\end{lem}

\begin{proof}
Let $e_1, e_2, \ldots , e_k, f_1, f_2, \ldots , f_k$ be a basis of $M$ given by Lemma \ref{lem:met-basis}, and let $d_i = \lambda(f_i,f_i)$. We will also use the notation $e_i$, $f_i$ for the basis of the first component in $M \oplus M$ and $M \oplus \Z^{2k}$, and the basis of the second component of $M \oplus M$ will be denoted by $\bar{e}_i$ and $\bar{f}_i$. Let $a_1, a_2, \ldots , a_k, b_1, b_2, \ldots , b_k$ denote the standard basis of $\Z^{2k}$. With this notation let $I(e_i)=e_i+b_i$, $I(f_i)=f_i+d_ib_i$, $I(\bar{e}_i)=-b_i$ and $I(\bar{f}_i)=f_i-a_i$.
\end{proof}

\subsection{Full geometric metabolic forms}

In this section we prove Theorem \ref{thm:fund-met}, which states that free full geometric metabolic forms are unique up to stable isomorphism. We will need two lemmas.

\begin{defin}
Let $\UM = (M, \lambda, \mu)$ be a metabolic form over an abelian group $Q$. For a T-lagrangian $L \leq M$ we define the homomorphism $\hat{\mu}_L : L^* \rightarrow Q$ as follows.
\[
\xymatrix{
L^* & M^* \ar[l] & M \ar[l]_-{\ad \lambda} \ar[r]^-{\mu} \ar[d] & Q \\ 
 & & M/L \ar[ull]^-{\cong} \ar[ur] & 
}
\]
The composition $M \rightarrow L^*$ of $\ad \lambda$ and the dual of the inclusion $L \rightarrow M$ has kernel $L$, hence induces an isomorphism $M/L \cong L^*$. Since $\mu \big| _L = 0$, $\mu$ induces a map $M/L \rightarrow Q$. The composition of the inverse of the isomorphism $M/L \cong L^*$ and the map $M/L \rightarrow Q$ will be denoted by $\hat{\mu}_L$.
\end{defin}

\begin{lem} \label{lem:isom-extend}
Let $Q$ be an abelian group and $v : Q \rightarrow \Z_2$ a homomorphism. Let $\UM = (M, \lambda, \mu)$ and $\UMp = (M', \lambda', \mu')$ be extended quadratic forms over $Q$. Suppose that they are metabolic, free and geometric (with respect to $v$), and let $L \leq M$ and $L' \leq M'$ be lagrangians. Suppose that $I_0 : L \rightarrow L'$ is an isomorphism of abelian groups. Then there is an isomorphism $I : \UM \rightarrow \UMp$ such that $I \big| _L = I_0$ if and only if $\hat{\mu}'_{L'} = \hat{\mu}_L \circ I_0^* : (L')^* \rightarrow Q$.
\end{lem}

\begin{proof}
If an isomorphism $I$ exists, it induces a commutative diagram 
\[
\xymatrix@=10pt{
L^* & & M^* \ar[ll] & & M \ar[ll]_-{\ad \lambda} \ar[rr]^-{\mu} \ar[dl] \ar[dd]^(.6){I} |!{[dl];[rr]}\hole & & Q \ar@{=}[dd] \\ 
 & & & M/L \ar[ulll]^(.65){\cong} \ar[urrr] \ar[dd] & & & \\
(L')^* \ar[uu]^(.4){I_0^*} & & (M')^* \ar[ll] \ar[uu]^(.4){I^*} |!{[ur];[uull]}\hole & & M' \ar[ll]_(.25){\ad \lambda'} |!{[ul];[dl]}\hole \ar[rr]^-{\mu'} \ar[dl] & & Q \\ 
 & & & M'/L' \ar[ulll]^-{\cong} \ar[urrr] & & & 
}
\]
hence $\hat{\mu}'_{L'} = \hat{\mu}_L \circ I_0^*$.

For the other direction suppose that $\hat{\mu}'_{L'} = \hat{\mu}_L \circ I_0^* : (L')^* \rightarrow Q$. Since $L \leq M$ and $L' \leq M'$ are direct summands, they have direct complements $N \leq M$ and $N' \leq M'$. By Lemma \ref{lem:lagr-dual} $N \cong L^*$, and similarly $N' \cong (L')^*$. Hence $I_0^*$ can be regarded as an isomorphism $I_0^* : N' \rightarrow N$, and the condition means that $\mu' \big| _{N'} = \mu \big| _N \circ I_0^* : N' \rightarrow Q$.

Let $r = \rk N (= \rk N' = \rk L = \rk L')$. Let $f'_1, f'_2, \ldots , f'_r$ be a basis of $N'$, and let $f_i = I_0^*(f'_i)$. The basis $f_1, f_2, \ldots , f_r$ of $N \cong L^*$ determines a dual basis $e_1, e_2, \ldots , e_r$ of $L$ (satisfying $\lambda(e_i,f_j)=0$ if $i \neq j$ and $\lambda(e_i,f_i)=1$). Similarly, $L'$ has a basis $e'_1, e'_2, \ldots , e'_r$ which is dual to the basis $f'_1, f'_2, \ldots , f'_r$ of $N'$. 

The isomorphism $N \cong L^*$ is given by the composition $N \rightarrow M \rightarrow M^* \rightarrow L^*$, hence $f_j \in N$ corresponds to $\ad \lambda(f_j) \big| _L \in L^*$. So we have $\lambda'(f'_j,I_0(e_i)) = (\ad \lambda'(f'_j) \big| _{L'})(I_0(e_i)) = (I_0^*(\ad \lambda'(f'_j) \big| _{L'}))(e_i) = (\ad \lambda(f_j) \big| _L)(e_i) = \lambda(f_j,e_i) = \lambda'(f'_j,e'_i)$ for every $i,j$ (using the definition of $f_j$ and the dual bases). Since $N' \cong (L')^*$, this implies that $I_0(e_i)=e'_i$ for every $i$.

Next we define the elements $\bar{f}_1, \bar{f}_2, \ldots , \bar{f}_r \in M$ recursively. Let $\bar{f}_1 = f_1$, and let $\bar{f}_{i+1} = f_{i+1} - \sum_{j=1}^i \lambda(\bar{f}_j,f_{i+1})e_j - \left\lfloor \frac{\lambda(f_{i+1},f_{i+1})}{2} \right\rfloor e_{i+1}$. By induction we see that $e_1, e_2, \ldots , e_r, \bar{f}_1, \bar{f}_2, \ldots , \bar{f}_i$ is a basis of $L \oplus \left< f_1, f_2, \ldots , f_i \right>$, therefore $e_1, e_2, \ldots , e_r, \bar{f}_1, \allowbreak \bar{f}_2, \ldots , \bar{f}_r$ is a basis of $M$. By construction $\lambda(e_i,\bar{f}_j)=\lambda(e_i,f_j)$ for every $i,j$, $\lambda(\bar{f}_i,\bar{f}_j)=0$ if $i \neq j$, and $\lambda(\bar{f}_i,\bar{f}_i)=0$ if $\lambda(f_i,f_i)$ is even and $\lambda(\bar{f}_i,\bar{f}_i)=1$ if $\lambda(f_i,f_i)$ is odd. Furthermore $\mu(\bar{f}_i)=\mu(f_i)$. We define the elements $\bar{f}'_1, \bar{f}'_2, \ldots , \bar{f}'_r \in M'$ analogously. 

Let $I : M \rightarrow M'$ be the homomorphism given by $I(e_i) = e'_i$ and $I(\bar{f}_i) = \bar{f}'_i$. This is an isomorphism of abelian groups, and $I \big| _L = I_0$. We need to check that $I$ is a morphism of extended quadratic forms. For every $i$ we have $\mu(e_i) = 0 = \mu'(e'_i)$ and $\mu(\bar{f}_i) = \mu(f_i) = \mu \circ I_0^*(f'_i) = \mu'(f'_i) = \mu'(\bar{f}'_i)$, therefore $\mu = \mu' \circ I$. For every $i,j$ $\lambda(e_i,e_j) = 0 = \lambda'(e'_i,e'_j)$. If $i \neq j$, then $\lambda(e_i,\bar{f}_j) = 0 = \lambda'(e'_i,\bar{f}'_j)$ and $\lambda(\bar{f}_i,\bar{f}_j) = 0 = \lambda'(\bar{f}'_i,\bar{f}'_j)$. For every $i$ we have $\lambda(e_i,\bar{f}_i) = 1 = \lambda'(e'_i,\bar{f}'_i)$. Finally, both $\lambda(\bar{f}_i,\bar{f}_i)$ and $\lambda'(\bar{f}'_i,\bar{f}'_i)$ are either $0$ or $1$. And since $\UtM$ and $\UtMp$ are geometric, $\varrho_2 \circ \lambda(\bar{f}_i,\bar{f}_i) = v \circ \mu(\bar{f}_i) = v \circ \mu' \circ I(\bar{f}_i) = v \circ \mu'(\bar{f}'_i) = \varrho_2 \circ \lambda'(\bar{f}'_i,\bar{f}'_i)$, so this means that $\lambda(\bar{f}_i,\bar{f}_i) = \lambda'(\bar{f}'_i,\bar{f}'_i)$. Therefore $\lambda(x,y) = \lambda'(I(x),I(y))$ for every $x,y \in M$, so $I$ is an isomorphism of extended quadratic forms.
\end{proof}

\begin{rem}
Lemma \ref{lem:isom-extend} generalises \cite[Corollary 5.3.1]{wall-scm} (see also Corollary \ref{cor:fund-hyp} below). Indeed, if $Q=0$ (ie.\ in the case of ordinary quadratic forms), $\hat{\mu}_L = \hat{\mu}'_{L'} = 0$, so every isomorphism $I_0$ satisfies the condition, and hence can be extended to an isomorphism $I$. However, if $Q \not\cong 0$, then, depending on $\hat{\mu}_L$ and $\hat{\mu}'_{L'}$, it may happen that no isomorphism $I_0$ satisfies the condition, so there is no isomorphism $I : \UM \rightarrow \UMp$ with $I(L) = L'$. So a geometric metabolic form is essentially unique when $Q=0$, but this is not true over arbitrary pairs $(Q,v)$ (cf.\ \cite[Remark 15.196]{lueck-macko24}).

We will obtain a uniqueness result for arbitrary $(Q,v)$ by requiring the forms to also be full, see Theorem \ref{thm:fund-met}. We will get uniqueness only in a weaker sense (up to stable isomorphism, and instead of extending a given isomorphism between lagrangians, we can only prescribe the image of a lagrangian), but this will still be sufficient for our applications in Section \ref{s:l-mon}.
\end{rem}

An immediate corollary is the following:

\begin{cor} \label{cor:fund-hyp}
Let $\UH = (H, \lambda, \mu)$ and $\UHp = (H', \lambda', \mu')$ be hyperbolic extended quadratic forms over some abelian group $Q$ such that $\rk \UH = \rk \UHp$. Let $L$ and $L'$ be lagrangians in $\UH$ and $\UHp$ respectively. Then there is an isomorphism $I : \UH \rightarrow \UHp$ such that $I(L) = L'$. 
\end{cor}

\begin{proof}
Since $\rk L = \frac{1}{2} \rk H = \frac{1}{2} \rk H' = \rk L'$, there is an isomorphism $I_0 : L \rightarrow L'$. Since $\mu = \mu' = 0$, we also have $\hat{\mu}_L = \hat{\mu}'_{L'} = 0$. So by Lemma \ref{lem:isom-extend} there is an isomorphism $I : \UH \rightarrow \UHp$ such that $I \big| _L = I_0$, in particular $I(L)=L'$.
\end{proof}

We will also need the following lemma, part a) of which can be extracted from the proof of \cite[Proposition 8. ii)]{kreck99}. 

\begin{lem} \label{lem:hom-stab-isom}
Let $F$, $G$ and $A$ be abelian groups, and suppose that $F$ and $G$ are free. Let $f : F \rightarrow A$ and $g : G \rightarrow A$ be surjective homomorphisms. Then 

a) There exist free abelian groups $F'$, $G'$, and an isomorphism $h : F \oplus F' \rightarrow G \oplus G'$ such that $(f+0) = (g+0) \circ h : F \oplus F' \rightarrow A$. 

b) If, in addition, $A$ is free and $\rk F = \rk G$, then there is an isomorphism $h : F \rightarrow G$ such that $f = g \circ h$. 
\end{lem}

\begin{proof}
a) Since $F$ is free and $g$ is surjective, the homomorphism $f$ can be lifted to a homomorphism $\bar{f} : F \rightarrow G$ such that $f = g \circ \bar{f}$. Let $\bar{f}_0 : F_0 \rightarrow \Ker g \subseteq G$ be a surjective homomorphism from some free abelian group $F_0$ to $\Ker g$. Let $F_1 = F \oplus F_0$ and $\bar{f}_1 = \bar{f} + \bar{f}_0 : F_1 = F \oplus F_0 \rightarrow G$. 

We prove that $\bar{f}_1$ is surjective. Let $y \in G$ be an arbitrary element, then there is an $x \in F$ such that $f(x) = g(y)$ (because $f$ is surjective). Let $y_0 = \bar{f}(x) \in \Image \bar{f} \subseteq \Image \bar{f}_1$, then $y-y_0 \in \Ker g = \Image \bar{f}_0 \subseteq \Image \bar{f}_1$, therefore $y = y_0 + (y-y_0) \in \Image \bar{f}_1$. 

Since $G$ is free, $\bar{f}_1$ has a right inverse, ie.\ there is a homomorphism $c : G \rightarrow F_1$ such that $\bar{f}_1 \circ c = \id_G$. This also implies that $F_1 = \Ker \bar{f}_1 \oplus \Image c \cong \Ker \bar{f}_1 \oplus G$. 

Let $F' = F_0 \oplus G$ and $G' = F_1$. Then $F \oplus F' = F \oplus (F_0 \oplus G) \cong F_1 \oplus G$ and $G \oplus G' = G \oplus F_1$, and we can define the homomorphism $h : F_1 \oplus G \rightarrow G \oplus F_1$ by the formula $h(x,y)= (\bar{f}_1(x),x+c(y))$. 

To show that $h$ is an isomorphism, we use the decomposition $F_1 \cong \Ker \bar{f}_1 \oplus G$. For $(x_1,y_2) \in \Ker \bar{f}_1 \oplus G \cong F_1$ we have $\bar{f}_1(x_1,y_2)=y_2$, and for $y \in G$ we have $c(y) = (0,y) \in \Ker \bar{f}_1 \oplus G \cong F_1$. Hence under the decompositions $F \oplus F' \cong F_1 \oplus G \cong (\Ker \bar{f}_1 \oplus G) \oplus G$ and $G \oplus G' = G \oplus F_1 \cong G \oplus (\Ker \bar{f}_1 \oplus G)$ the map $h$ is given by the formula $h((x_1,y_2),y) = (y_2, (x_1,y_2+y))$, therefore it is an isomorphism. 

Finally, for any $x \in F$, $x_0 \in F_0$ and $y \in G$ we have $(g+0) \circ h((x,x_0),y) = g \circ \bar{f}_1(x,x_0) = g \circ (\bar{f} + \bar{f}_0)(x,x_0) = g \circ \bar{f}(x) = f(x) = (f+0)(x,(x_0,y))$, where we used that $g \circ \bar{f}_0 = 0$, because $\Image \bar{f}_0 = \Ker g$. Therefore $(f+0) = (g+0) \circ h$. 

b) Since $A$ is free, there are maps $c : A \rightarrow F$ and $d : A \rightarrow G$ such that $f \circ c = g \circ d = \id_A$. We get the decompositions $F = \Ker f \oplus \Image c$ and $G = \Ker g \oplus \Image d$, where $\Image c \cong \Image d \cong A$. The subgroups $\Ker f$ and $\Ker g$ are free, and $\rk \Ker f = \rk F - \rk A = \rk G - \rk A = \rk \Ker g$, therefore there is an isomorphism $h_0 : \Ker f \rightarrow \Ker g$. Then we can take $h = h_0 \oplus (d \circ f) \big| _{\Image c} : F = \Ker f \oplus \Image c \rightarrow G = \Ker g \oplus \Image d$. 
\end{proof}

\begin{thm} \label{thm:fund-met}
Let $Q$ be an abelian group and $v : Q \rightarrow \Z_2$ a homomorphism. Let $\UM = (M, \lambda, \mu)$ and $\UMp = (M', \lambda', \mu')$ be extended quadratic forms over $Q$. Suppose that they are metabolic, free, full and geometric (with respect to $v$), and let $L \leq M$ and $L' \leq M'$ be lagrangians. Then 

a) For some integers $k,l \geq 0$ there is an isomorphism $I : \UM \oplus \UH_{2k} \rightarrow \UMp \oplus \UH_{2l}$ such that $I(L \oplus (\{ 0 \} \times \Z^k)) = L' \oplus (\{ 0 \} \times \Z^l)$. 

b) If, in addition, $Q$ is free and $\rk \UM = \rk \UMp$, then there is an isomorphism $I : \UM \rightarrow \UMp$ such that $I(L) = L'$. 
\end{thm}

\begin{proof}
a) Since the homomorphisms $\mu : M \rightarrow Q$ and $\mu' : M' \rightarrow Q$ are surjective, the same is true for the induced maps $M/L \rightarrow Q$ and $M'/L' \rightarrow Q$, and for $\hat{\mu}_L : L^* \rightarrow Q$ and $\hat{\mu}'_{L'} : (L')^* \rightarrow Q$. By Lemma \ref{lem:hom-stab-isom} a) there are free abelian groups $A, A'$ and an isomorphism $h : (L')^* \oplus A' \rightarrow L^* \oplus A$ such that $(\hat{\mu}'_{L'} +0) = (\hat{\mu}_L +0) \circ h : (L')^* \oplus A' \rightarrow Q$. Let $k = \rk A$ and $l = \rk A'$. 

Let $\UtM = (\tilde{M}, \tilde{\lambda}, \tilde{\mu})$ denote the direct sum $\UM \oplus \UH_{2k}$, that is $\tilde{M} = M \oplus \Z^{2k}$, $\tilde{\lambda} = \lambda \oplus \bigl[ \begin{smallmatrix} 0 & I_k \\ I_k & 0 \end{smallmatrix} \bigr]$ and $\tilde{\mu} = \mu \oplus 0$. Similarly let $\UtMp = (\tilde{M}', \tilde{\lambda}', \tilde{\mu}') = \UMp \oplus \UH_{2l}$. Let $\tilde{L} = L \oplus (\{ 0 \} \times \Z^k)$ and $\tilde{L}' = L' \oplus (\{ 0 \} \times \Z^l)$, these are lagrangians in $\UtM$ and $\UtMp$ respectively. Hence the new extended quadratic forms $\UtM$ and $\UtMp$ are again metabolic, free, (full) and geometric. Moreover, we have $\hat{\tilde{\mu}}_{\tilde{L}} = \hat{\mu}_L + 0 : \tilde{L}^* \cong L^* \oplus (\{ 0 \} \times \Z^k)^* \rightarrow Q$ and $\hat{\tilde{\mu}}'_{\tilde{L}'} = \hat{\mu}'_{L'} + 0 : (\tilde{L}')^* \cong (L')^* \oplus (\{ 0 \} \times \Z^l)^* \rightarrow Q$. 

If we fix some isomorphisms $A \cong (\{ 0 \} \times \Z^k)^*$ and $A' \cong (\{ 0 \} \times \Z^l)^*$, then $h$ can be regarded as an isomorphism $h : (\tilde{L}')^* \rightarrow \tilde{L}^*$ such that $\hat{\tilde{\mu}}'_{\tilde{L}'} = \hat{\tilde{\mu}}_{\tilde{L}} \circ h$. Let $I_0 = h^* : \tilde{L} \rightarrow \tilde{L}'$. Then by Lemma \ref{lem:isom-extend} there is an isomorphism $I : \UtM \rightarrow \UtMp$ such that $I \big| _{\tilde{L}} = I_0$, in particular $I(\tilde{L}) = \tilde{L}'$.

b) In this case we can use part b) of Lemma \ref{lem:hom-stab-isom} instead of part a). Hence we can take $A = A' = 0$, so the statement of part a) holds with $k = l = 0$. 
\end{proof}

\section{The group $\RU_{\st}(\UM,L)$} \label{s:ru}

Fix an abelian group $Q$ with a homomorphism $v : Q \rightarrow \Z_2$. In this section every extended quadratic form we consider will be over $Q$, and will be assumed to be free and geometric (with respect to $v$). 

For a metabolic form $\UM$ and a lagrangian $L$ we introduce the group $\RU(\UM,L)$, and its stable version $\RU_{\st}(\UM,L)$. The group $\RU(\UM,L)$ generalises the group $\RU^+(\Z^r)$ used in classical and modified surgery (see \cite[Section 6]{kreck99} and \cite[Section 6]{wall-scm}). In particular, $\RU(\UH_{2k},\{0\} \times \Z^k) = \RU^+(\Z^k)$ (cf.\ \cite[Remark 4.2.23]{csn-thesis}). Note, however, the dependence of $\RU(\UM,L)$ on the lagrangian $L$ (which is implicit in the classical/modified setting, but will be important in our arguments) and that the stable group $\RU_{\st}(\UM,L)$ consists of automorphisms of the same $\UM$, in contrast to $\RU^+(\Z) = \lim_r \RU^+(\Z^r)$.

We will show that $\Phi \oplus \Phi^{-1} \in \RU_{\st}(\UM \oplus \UM, L \oplus L)$ for any $\Phi \in \Aut(\UM)$, which will play a key role in the proof of Theorem \ref{thm:jacobi}.

\begin{defin} \label{def:ru}
Given a metabolic form $\UM$ and a lagrangian $L$ in $\UM$, we define the subgroup $\RU(\UM,L) \leq \Aut(\UM)$, which is generated by the following elements: 
\begin{compactitem}
\item Automorphisms $f : \UM \rightarrow \UM$ such that $f(L)=L$.
\item $I^{-1} \circ (\sigma \oplus \id_{M'}) \circ I$ for every isomorphism $I : \UM \rightarrow \UH_2 \oplus \UMp$ (for some metabolic form $\UMp$) such that $I(L) = (\{ 0 \} \times \Z) \oplus L'$ for some lagrangian $L'$ in $\UMp$, where $\sigma : \Z^2 \rightarrow \Z^2$ is the flip map (ie.\ $\sigma(a,b)=(b,a)$). 
\end{compactitem}
\end{defin}

\begin{defin} \label{def:ru-st}
Given a metabolic form $\UM$ and a lagrangian $L$ in $\UM$, let 
\[
\RU_{\st}(\UM,L) = \{ f \in \Aut(\UM) \mid \exists (\UN,K) : f \oplus \id_N \in \RU(\UM \oplus \UN, L \oplus K) \}
\]
where $\UN$ denotes some metabolic form with a lagrangian $K$.
\end{defin}

If $f \in \RU(\UM,L)$, then $f \oplus \id_N \in \RU(\UM \oplus \UN, L \oplus K)$ for any $(\UN,K)$, showing that $\RU_{\st}(\UM,L)$ is a subgroup of $\Aut(\UM)$, and $\RU(\UM,L) \leq \RU_{\st}(\UM,L)$. 

\begin{lem} \label{lem:ru-eq1}
Suppose that $\UM$ and $\UN$ are metabolic forms with lagrangians $L$ and $K$ respectively. If $F : \UM \rightarrow \UN$ is an isomorphism such that $F(L)=K$, then $\RU(\UM,L) = F^{-1} \circ \RU(\UN,K) \circ F$.
\end{lem}

\begin{proof}
Suppose that $f \in \Aut(\UN)$. Then $f$ is a generator of $\RU(\UN,K)$ if and only if $F^{-1} \circ f \circ F$ is a generator of $\RU(\UM,L)$.
\end{proof}

\begin{lem} \label{lem:ru-eq2}
Suppose that $\UM$ is a metabolic form, and $L$ and $K$ are lagrangians in $\UM$. If there is an $f \in \RU(\UM,L)$ such that $f(L)=K$, then $\RU(\UM,L) = \RU(\UM,K)$.
\end{lem}

\begin{proof}
We have $\RU(\UM,K) = f \circ \RU(\UM,L) \circ f^{-1} = \RU(\UM,L)$ by Lemma \ref{lem:ru-eq1} and because $f \in \RU(\UM,L)$. 
\end{proof}

\begin{lem} \label{lem:ru-diag}
Let $\UM$ be an extended quadratic form and $\Phi \in \Aut(\UM)$. Then we have $\Phi \oplus \Phi \in \RU(\UM \oplus (-\UM), \Delta_{\id})$. 
\end{lem}

\begin{proof}
Since $\Delta_{\id}$ is the diagonal, $(\Phi \oplus \Phi)(\Delta_{\id}) = \Delta_{\id}$.
\end{proof}

\begin{prop} \label{prop:ru-diag-eq}
Let $\UM$ be a metabolic form and $L$ a lagrangian in $\UM$. Then 
\[
\RU(\UM \oplus (-\UM), \Delta_{\id}) = \RU(\UM \oplus (-\UM), L \oplus L) \text{\,.}
\] 
\end{prop}

\begin{proof}
Denote the rank of $\UM$ by $2k$. Let $e_1, e_2, \ldots , e_k, f_1, f_2, \ldots , f_k$ be a basis of $M$ given by Lemma \ref{lem:met-basis}, and let $d_i = \lambda(f_i,f_i)$. We will use the notation $e_i$, $f_i$ when $M$ is regarded as the underlying abelian group of $\UM$, and the same elements will be denoted by $\bar{e}_i$ and $\bar{f}_i$ when $M$ is regarded as the underlying abelian group of $-\UM$. Let $a_1, a_2, \ldots , a_k, b_1, b_2, \ldots , b_k$ denote the standard basis of $\Z^{2k}$. Let $I : \UM \oplus \UM \rightarrow \UM \oplus \UH_{2k}$ and $J : \UM \rightarrow -\UM$ denote the isomorphisms from Lemmas \ref{lem:met-isom} and \ref{lem:met-neg-isom} respectively. 

The composition of the flip map $\UM \oplus (-\UM) \rightarrow (-\UM) \oplus \UM$, $J^{-1} \oplus \id_M$ and $I$ is an isomorphism $F : \UM \oplus (-\UM) \rightarrow \UM \oplus \UH_{2k}$ given by $F(e_i)=-b_i$, $F(f_i)=f_i-a_i$, $F(\bar{e}_i)=e_i+b_i$ and $F(\bar{f}_i)=d_ie_i-f_i$. Therefore 
\begin{multline*}
F(L \oplus L) = F(\left< e_1, \ldots , e_k, \bar{e}_1, \ldots , \bar{e}_k \right>) = \\ 
= \left< -b_1, \ldots , -b_k, e_1+b_1, \ldots , e_k+b_k \right> = \left< b_1, \ldots , b_k, e_1, \ldots , e_k \right> = L \oplus (\{0\} \times \Z^k)
\end{multline*}
and
\begin{multline*}
F(\Delta_{\id}) = F(\left< e_1+\bar{e}_1, \ldots , e_k+\bar{e}_k, f_1+\bar{f}_1, \ldots , f_k+\bar{f}_k \right>) = \\ 
= \left< e_1, \ldots , e_k, d_1e_1-a_1, \ldots , d_ke_k-a_k \right> = \left< e_1, \ldots , e_k, a_1, \ldots , a_k \right> = L \oplus (\Z^k \times \{0\}) \text{\,.}
\end{multline*}
If $\Sigma : \Z^{2k} \rightarrow \Z^{2k}$ denotes the flip map (ie.\ $\Sigma(a,b)=(b,a)$ for $a,b \in \Z^k$), then $\id_M \oplus \Sigma \in \RU(\UM \oplus \UH_{2k}, L \oplus (\{0\} \times \Z^k))$. Since $(\id_M \oplus \Sigma)(L \oplus (\{0\} \times \Z^k)) = L \oplus (\Z^k \times \{0\})$, we have $\RU(\UM \oplus \UH_{2k}, L \oplus (\{0\} \times \Z^k)) = \RU(\UM \oplus \UH_{2k}, L \oplus (\Z^k \times \{0\}))$ by Lemma \ref{lem:ru-eq2}. Therefore 
\[
\begin{aligned}
\RU(\UM \oplus (-\UM), \Delta_{\id}) &= F^{-1} \circ \RU(\UM \oplus \UH_{2k}, L \oplus (\Z^k \times \{0\})) \circ F = \\
 &= F^{-1} \circ \RU(\UM \oplus \UH_{2k}, L \oplus (\{0\} \times \Z^k)) \circ F = \\
 &= \RU(\UM \oplus (-\UM), L \oplus L)
\end{aligned}
\] 
by Lemma \ref{lem:ru-eq1}.
\end{proof}

The following is a generalisation of Wall's \cite[Lemma 6.2]{wall-scm}.

\begin{thm} \label{thm:ru-wall}
Let $\UM$ be  a metabolic form, and $L$ a lagrangian in $\UM$. For any $\Phi \in \Aut(\UM)$ we have 
\[
\Phi \oplus \Phi^{-1} \in \RU_{\st}(\UM \oplus \UM, L \oplus L) \text{\,.}
\]
\end{thm}

\begin{proof}
By Lemma \ref{lem:ru-diag} and Proposition \ref{prop:ru-diag-eq} we have $\Phi \oplus \Phi \in \RU(\UM \oplus (-\UM), L \oplus L)$. This implies that 
\[
\id_M \oplus \Phi \oplus \Phi \in \RU(\UM \oplus \UM \oplus (-\UM), L \oplus L \oplus L)
\]
and that 
\[
\Phi \oplus \id_M \oplus \Phi \in \RU(\UM \oplus \UM \oplus (-\UM), L \oplus L \oplus L)
\]
By composing the latter with the inverse of the former we get that 
\[
\Phi \oplus \Phi^{-1} \oplus \id_M \in \RU(\UM \oplus \UM \oplus (-\UM), L \oplus L \oplus L)
\]
hence $\Phi \oplus \Phi^{-1} \in \RU_{\st}(\UM \oplus \UM, L \oplus L)$.
\end{proof}

\section{Extended $\ell$-monoids} \label{s:l-mon}

In this section we will define and study extended $\ell$-monoids. 

Recall that in the setting of Kreck \cite{kreck99} the monoid $l_{2q+1}(e)$ consists of equivalence classes of pairs $(\UH_{2k},V)$, where $V$ is a half-rank direct summand in $\UH_{2k}$. Crowley and Sixt \cite{crowley-sixt11} gave an equivalent definition in terms of quasi-formations: triples $(\UM, L, V)$ consisting of an (ordinary) quadratic form $\UM$, a lagrangian $L$ and a half-rank direct summand $V$. 

Building on these ideas, we define extended $\ell$-monoids. First we define quasi-formations (over an abelian group $Q$), using extended quadratic forms over $Q$ instead of ordinary quadratic forms. An extended $\ell$-monoid will then consist of equivalence classes of quasi-formations. As our goal is to define an obstruction in this $\ell$-monoid (see the definitions of $\theta_{W,F}$ and $[\theta_{W,F}]$ in Section \ref{s:obstr-def}), we define the equivalence relation in a way which will be compatible with the effect of surgery on $\theta_{W,F}$. Because of this, the relation used in \cite[Definition 3.8 (i)]{crowley-sixt11} will instead be a theorem in our setting (see Theorem \ref{thm:jacobi}). 

For most of this section we will work in a purely algebraic setting. First we define the $\ell$-monoid associated to a group $Q$ and a homomorphism $v : Q \rightarrow \Z_2$. We will consider two versions. In the first version elements are represented by quasi-formations such that the torsion subgroup of the abelian group underlying the extended quadratic form is isomorphic to a fixed torsion group $R$, we will denote this by $\ell^T_{2q+1}(Q,v,R)$. In the second version (which is a special case of the first), denoted by $\ell_{2q+1}(Q,v)$, the extended quadratic forms are free. The second version is more suitable for computations, as it has a natural monoid structure given by direct sum (see Proposition \ref{prop:lmonoid-zero}). However, the surgery obstruction that we will define in Section \ref{s:obstr-def} lives in an $\ell$-monoid of the first type. We will prove that the two versions are equivalent, ie.\ there is a natural bijection between them (see Proposition \ref{prop:ell-bij}), this will allow us to perform computations with the surgery obstruction. 

Next we prove the important identity of Theorem \ref{thm:jacobi}. Then we define the group $L_{2q+1}(Q,v)$, which is a generalisation of the classical L-group $L_{2q+1}$. It is the group of invertible elements of $\ell_{2q+1}(Q,v)$, and we show (using the fact that $L_{2q+1} \cong 0$) that if $Q$ is free, then $L_{2q+1}(Q,v) \cong 0$. We will end the section by defining the $\ell$-monoid associated to a space $B$ with a stable bundle $\xi$ over it.

\subsection{Definitions}

\begin{defin}
A \emph{quasi-formation} over an abelian group $Q$ is a triple
\[
(\UM; L, V)
\]
where
\begin{compactitem}
\item $\UM = (M, \lambda, \mu)$ is a metabolic extended quadratic form over $Q$
\item $L \leq M$ is a T-lagrangian
\item $V \leq M$ is a free half-rank direct summand
\end{compactitem}
\end{defin}

\begin{defin}
A quasi-formation $(\UM; L, V)$ is called \emph{elementary}, if $M = L \oplus V$ (as an internal direct sum).
\end{defin}

\begin{defin}
The \emph{standard rank-$2k$ elementary hyperbolic quasi-formation} is $\HH_{2k} = (\UH_{2k}; \{ 0 \} \times \Z^k, \Z^k \times \{ 0 \})$.
\end{defin}

\begin{defin}
Let $(\UM; L, V)$ and $(\UMp; L', V')$ be quasi-formations over an abelian group $Q$. They are \emph{isomorphic} if there is an isomorphism $h : \UM \rightarrow \UMp$ such that $L' = h(L)$ and $V' = h(V)$. 
\end{defin}

\begin{defin}
Let $(\UM; L, V)$ and $(\UMp; L', V')$ be quasi-formations over an abelian group $Q$. Their \emph{direct sum} is $(\UM; L, V) \oplus (\UMp; L', V') = (\UM \oplus \UMp; L \oplus L', V \oplus V')$. 
\end{defin}

\begin{defin}
Let $(\UM; L, V)$ and $(\UMp; L', V')$ be quasi-formations over an abelian group $Q$. They are \emph{stably isomorphic} if there are integers $k,l \geq 0$ such that  $(\UM; L, V) \oplus \HH_{2k} \cong (\UMp; L', V') \oplus \HH_{2l}$. Stable isomorphism will be denoted by ${} \si {}$. 
\end{defin}

\begin{defin} \label{def:gc}
For an abelian group $Q$ with a homomorphism $v : Q \rightarrow \Z_2$, a torsion abelian group $R$ and an even integer $q \geq 0$ let
\[
\gc^T_{2q+1}(Q,v,R)
\]
denote the class of quasi-formations $(\UM; L, V)$ over $Q$ such that $\UM = (M, \lambda, \mu)$ is geometric (with respect to $v$), full and $\Tor M \cong R$. 
\end{defin}

Note that the class $\gc^T_{2q+1}(Q,v,R)$ does not depend on the choice of $q$. 

\begin{defin}
For an abelian group $Q$ with a homomorphism $v : Q \rightarrow \Z_2$, a torsion abelian group $R$ and an even integer $q \geq 0$ let
\[
\gs^T_{2q+1}(Q,v,R) = \gc^T_{2q+1}(Q,v,R) / {\cong}
\]
denote the set of isomorphism classes of quasi-formations $(\UM; L, V)$ over $Q$ such that $\UM = (M, \lambda, \mu)$ is geometric (with respect to $v$), full and $\Tor M \cong R$. 
\end{defin}

\begin{defin}
For an abelian group $Q$ with a homomorphism $v : Q \rightarrow \Z_2$, a torsion abelian group $R$ and an even integer $q \geq 0$ let
\[
\s^T_{2q+1}(Q,v,R) = \gc^T_{2q+1}(Q,v,R) / {\si}
\] 
denote the set of stable isomorphism classes of quasi-formations $(\UM; L, V)$ over $Q$ such that $\UM = (M, \lambda, \mu)$ is geometric (with respect to $v$), full and $\Tor M \cong R$. 
\end{defin}

\begin{defin} \label{def:ell-t}
For an abelian group $Q$ with a homomorphism $v : Q \rightarrow \Z_2$, a torsion abelian group $R$ and an even integer $q \geq 0$ let
\[
\ell^T_{2q+1}(Q,v,R) = \gc^T_{2q+1}(Q,v,R) / {\sim}
\]
be the set of quasi-formations $(\UM; L, V)$ over $Q$ such that $\UM = (M, \lambda, \mu)$ is geometric (with respect to $v$), full and $\Tor M \cong R$ up to the equivalence relation $\sim$, which is generated by the following elementary equivalences: 
\begin{compactitem}
\item $(\UM; L, V) \sim (\UMp; L', V')$ if $(\UM; L, V) \si (\UMp; L', V')$
\item $(\UM \oplus \UH_2; L \oplus (\{ 0 \} \times \Z), V) \sim (\UM \oplus \UH_2; L \oplus (\Z \times \{ 0 \}), V)$, where $L$ is a T-lagrangian in $\UM$ and $V$ is a free half-rank direct summand in $\UM \oplus \UH_2$
\end{compactitem}

The equivalence class of $(\UM; L, V)$ will be denoted by $[\UM; L, V]$. 
\end{defin}

\begin{rem} \label{rem:eq-nat}
Notice that the equivalence relation $\sim$ depends ``naturally" on $(Q,v)$. To be able to state this precisely, we define the class $\gcb^T_{2q+1}(Q,v,R)$ of quasi-formations $(\UM; L, V)$ in the same way as $\gc^T_{2q+1}(Q,v,R)$, except that $\UM$ is not required to be full. Then $\sim$ can also be defined on $\gcb^T_{2q+1}(Q,v,R)$, and if $x,y \in \gc^T_{2q+1}(Q,v,R) \subseteq \gcb^T_{2q+1}(Q,v,R)$, then $x \sim y$ in $\gc^T_{2q+1}(Q,v,R)$ if and only if $x \sim y$ in $\gcb^T_{2q+1}(Q,v,R)$.

Pairs $(Q,v)$ are the objects in the slice category of abelian groups over $\Z_2$, where a morphism $(Q,v) \rightarrow (Q',v')$ is a homomorphism $f : Q \rightarrow Q'$ such that $v' \circ f = v$. Post-composition with such an $f$ turns an extended quadratic form over $Q$ into one over $Q'$, and determines a map $f_* : \gcb^T_{2q+1}(Q,v,R) \rightarrow \gcb^T_{2q+1}(Q',v',R)$. Then $\sim$ is natural in the sense that if $x \sim y \in \gcb^T_{2q+1}(Q,v,R)$, then $f_*(x) \sim f_*(y) \in \gcb^T_{2q+1}(Q',v',R)$.
\end{rem}

\begin{defin}
For an abelian group $Q$ with a homomorphism $v : Q \rightarrow \Z_2$ and an even integer $q \geq 0$ let
\[
\begin{aligned}
\gc_{2q+1}(Q,v) &= \gc^T_{2q+1}(Q,v,0) \text{\,,} \\
\gs_{2q+1}(Q,v) &= \gs^T_{2q+1}(Q,v,0) \text{\,,} \\
\s_{2q+1}(Q,v) &= \s^T_{2q+1}(Q,v,0) \text{\,,} \\
\ell_{2q+1}(Q,v) &= \ell^T_{2q+1}(Q,v,0) \text{\,.} 
\end{aligned}
\]
\end{defin}

To prove the equivalence of $\ell^T_{2q+1}(Q,v,R)$ and $\ell_{2q+1}(Q,v)$ we will need the following notation: 

\begin{defin} \label{def:quot}
Suppose that $\UN = (N, \lambda_N, \mu_N)$ is an extended quadratic form over an abelian group $Q$ such that $\mu_N \big| _{\Tor N} = 0$ (note that this condition automatically holds if $Q$ is free). Let $\bar{N}$ denote $N / \Tor N$ and $\pi_N : N \rightarrow N / \Tor N$ the quotient map. Let $\bar{\lambda}_N$ and $\bar{\mu}_N$ denote the induced bilinear function and homomorphism on $\bar{N}$, and let $\UbN = (\bar{N}, \bar{\lambda}_N, \bar{\mu}_N)$. If $X \leq N$ is a subgroup, then $\bar{X}$ will denote the subgroup $\pi_N(X) \leq N / \Tor N$.
\end{defin}

Note that $\pi_N : \UN \rightarrow \UbN$ is a morphism of extended quadratic forms. If $\UN$ is nonsingular, metabolic, full or geometric (with respect to some homomorphism), then the same property also holds for $\UbN$.

\begin{prop} \label{prop:ell-bij}
For any abelian group $Q$ with a homomorphism $v : Q \rightarrow \Z_2$, torsion abelian group $R$ and even integer $q \geq 0$, there is a canonical bijection between $\ell_{2q+1}(Q,v)$ and $\ell^T_{2q+1}(Q,v,R)$.
\end{prop}

\begin{proof}
First we define a function $F : \ell_{2q+1}(Q,v) \rightarrow \ell^T_{2q+1}(Q,v,R)$. Let $\UR = (R, 0, 0)$, it is a metabolic extended quadratic form of rank $0$ over any group, geometric with respect to any homomorphism. If $(\UM; L, V) \in \gc_{2q+1}(Q,v)$, then $(\UM; L, V) \oplus (\UR; R, 0) \in \gc^T_{2q+1}(Q,v,R)$, let $F([\UM; L, V]) = [\UM \oplus \UR; L \oplus R, V \oplus 0]$.

Next we define a function $G : \ell^T_{2q+1}(Q,v,R) \rightarrow \ell_{2q+1}(Q,v)$. Suppose that $(\UN; K, W) \in \gc^T_{2q+1}(Q,v,R)$, where $\UN = (N, \lambda, \mu)$. Since $\mu \big| _{\Tor N} = 0$ (because $\UN$ has a T-lagrangian), we can define $\UbN$. Moreover, $\bar{K} = K / \Tor N$ is a lagrangian in $\UbN$, and $\bar{W} = (W \oplus \Tor N) / \Tor N$ is a half-rank direct summand, hence $(\UbN; \bar{K}, \bar{W}) \in \gc_{2q+1}(Q,v)$. Let $G([\UN; K, W])=[\UbN; \bar{K}, \bar{W}]$. 

It is immediate from the definitions that $G \circ F = \id_{\ell_{2q+1}(Q,v)}$. For the other composition, let $[\UN; K, W] \in \ell^T_{2q+1}(Q,v,R)$, then $F \circ G([\UN; K, W]) = [\UbN \oplus \UR; \bar{K} \oplus R, \bar{W} \oplus 0]$. Since $W$ is a free direct summand, $\pi_N$ has a right inverse $I_0 : \bar{N} \rightarrow N$ such that $W \leq \Image I_0$. By choosing an isomorphism $\Tor N \cong R$ we can extend $I_0$ to an isomorphism $I : \bar{N} \oplus R \rightarrow N$. This is also an isomorphism $\UbN \oplus \UR \cong \UN$ of extended quadratic forms, and by construction $I(\bar{W} \oplus 0) = I_0(\bar{W})=W$. Since $K$ is a T-lagrangian, we have $K = (K \cap \Image I_0) \oplus \Tor N$, so $I(\bar{K} \oplus R) = I_0(\bar{K}) \oplus \Tor N = (K \cap \Image I_0) \oplus \Tor N = K$. This shows that $[\UbN \oplus \UR; \bar{K} \oplus R, \bar{W} \oplus 0] = [\UN; K, W]$, hence $F \circ G = \id_{\ell^T_{2q+1}(Q,v,R)}$.
\end{proof}

\begin{rem}
The same proof shows that there are canonical bijections between $\gs_{2q+1}(Q,v)$ and $\gs^T_{2q+1}(Q,v,R)$ and between $\s_{2q+1}(Q,v)$ and $\s^T_{2q+1}(Q,v,R)$.
\end{rem}

Recall that a quasi-formation $(\UM; L, V)$ is called elementary if $M = L \oplus V$. This property is invariant under stable isomorphism (because $(\UM; L, V)$ is elementary if and only if $(\UM; L, V) \oplus \HH_{2k}$ is elementary, for any $k$). Therefore an element of $\gs^T_{2q+1}(Q,v,R)$ or $\s^T_{2q+1}(Q,v,R)$ has an elementary representative if and only if all of its representatives are elementary.

\begin{defin}
An element of $\gs^T_{2q+1}(Q,v,R)$ or $\s^T_{2q+1}(Q,v,R)$ is called \emph{elementary}, if its representatives are elementary.
\end{defin}

In contrast to the above, an elementary quasi-formation may be equivalent to a non-elementary one under the equivalence relation $\sim$. 

\begin{defin}
An element of $\ell^T_{2q+1}(Q,v,R)$ is called \emph{elementary}, if it has an elementary representative. 
\end{defin}

\begin{rem} \label{rem:elem-bij}
Under the bijection of Proposition \ref{prop:ell-bij} elementary elements of $\ell^T_{2q+1}(Q,v,R)$ correspond to elementary elements of $\ell_{2q+1}(Q,v) = \ell^T_{2q+1}(Q,v,0)$. 
\end{rem}

\subsection{The monoid $\ell_{2q+1}(Q,v)$}

From now on we will focus on the free version $\ell_{2q+1}(Q,v)$. We show that it has a natural monoid structure and, using the results on $\RU_{\st}(\UM,L)$ from Section \ref{s:ru}, we prove Theorem \ref{thm:jacobi}.

\begin{lem} \label{lem:ru-equiv}
Suppose that $(\UM; L, V) \in \gc_{2q+1}(Q,v)$. If $\Psi \in \RU(\UM,L)$, then we have $(\UM; L, V) \sim (\UM; L, \Psi(V))$.
\end{lem}

\begin{proof}
It is enough to prove the statement when $\Psi$ is a generator of $\RU(\UM,L)$. 

If $\Psi(L) = L$, then $\Psi$ is an isomorphism between $(\UM; L, V)$ and $(\UM; L, \Psi(V))$.

If there is an isomorphism $I : \UM \rightarrow \UH_2 \oplus \UMp$ such that $I(L) = ( \{0 \} \times \Z) \oplus L'$ for some lagrangian $L'$ in $\UMp$ and $\Psi = I^{-1} \circ (\sigma \oplus \id_{M'}) \circ I$, then $I \circ \Psi(L) = (\sigma \oplus \id_{M'}) \circ I(L) = (\sigma \oplus \id_{M'})(( \{0 \} \times \Z) \oplus L') = (\Z \times \{0 \} ) \oplus L'$. So 
\[
\begin{aligned}
(\UM; L, V) &\cong (\UH_2 \oplus \UMp; I \circ \Psi(L), I \circ \Psi(V)) = \\
 &= (\UH_2 \oplus \UMp; (\Z \times \{0 \} ) \oplus L', I \circ \Psi(V)) \sim \\
 &\sim (\UH_2 \oplus \UMp; ( \{0 \} \times \Z) \oplus L', I \circ \Psi(V)) = \\
 &= (\UH_2 \oplus \UMp; I(L), I \circ \Psi(V)) \cong \\
 &\cong (\UM; L, \Psi(V))
\end{aligned}
\]
where we used the isomorphisms $I \circ \Psi$ and $I$. Therefore the statement holds for every generator of $\RU(\UM,L)$.
\end{proof}

\begin{lem} \label{lem:hyp-zero}
If $(\UM; L, V) \in \gc_{2q+1}(Q,v)$ and $k \geq 0$, then we have $(\UM; L, V) \sim (\UM; L, V) \oplus (\UH_{2k}; \{ 0 \} \times \Z^k, \{ 0 \} \times \Z^k)$. 
\end{lem}

\begin{proof}
As in Proposition \ref{prop:ru-diag-eq}, if $\Sigma$ denotes the flip map, then $\id_M \oplus \Sigma \in \RU(\UM \oplus \UH_{2k}, L \oplus (\{0\} \times \Z^k))$ and $(\id_M \oplus \Sigma)(V \oplus (\Z^k \times \{0\})) = V \oplus (\{0\} \times \Z^k)$. Hence by Lemma \ref{lem:ru-equiv}, we have $(\UM; L, V) \sim (\UM; L, V) \oplus (\UH_{2k}; \{ 0 \} \times \Z^k, \Z^k \times \{ 0 \}) \sim (\UM; L, V) \oplus (\UH_{2k}; \{ 0 \} \times \Z^k, \{ 0 \} \times \Z^k)$. 
\end{proof}

\begin{prop} \label{prop:lmonoid-zero}
a) Direct sum of quasi-formations induces an operation $\oplus$ on $\ell_{2q+1}(Q,v)$, and $(\ell_{2q+1}(Q,v), \oplus)$ is a commutative monoid. 

b) If $(\UM; L, L) \in \gc_{2q+1}(Q,v)$, then $[\UM; L, L]$ is the zero element in $\ell_{2q+1}(Q,v)$.
\end{prop}

\begin{proof}
It is clear that there is a well-defined induced operation $\oplus$ on $\ell_{2q+1}(Q,v)$, which is associative and commutative (cf.\ \cite[Proposition 4.2.18]{csn-thesis}). We need to show that it has a zero element. (Note that $\UH_{2k}$ is not full if $Q \neq 0$, so $\HH_{2k} \not \in \gc_{2q+1}(Q,v)$.) We will do this in 3 steps. 

First, we show that there is a quasi-formation of the form $(\UM; L, L)$ in the class $\gc_{2q+1}(Q,v)$ (in particular $\gc_{2q+1}(Q,v)$ and $\ell_{2q+1}(Q,v)$ are non-empty). Let $\mu_0 : A \rightarrow Q$ be a surjective homomorphism from some free abelian group $A$ to $Q$. Let $k = \rk A$, and let $f_1, f_2, \ldots , f_k$ be a basis of $A$. Let $L$ be another rank-$k$ free abelian group with basis $e_1, e_2, \ldots , e_k$. Let $M = L \oplus A$. Let $\lambda : M \times M \rightarrow \Z$ be the bilinear function given by the block matrix
\[
\begin{bmatrix}
0 & I_k \\
I_k & D
\end{bmatrix}
\]
(in the basis $e_1, e_2, \ldots , e_k, f_1, f_2, \ldots , f_k$), where $D$ is a diagonal matrix such that the diagonal entry corresponding to $f_i$ is $0$ if $v \circ \mu_0(f_i)=0$ and $1$ if $v \circ \mu_0(f_i)=1$ (hence $\varrho_2 \circ \lambda(f_i,f_i) = v \circ \mu_0(f_i)$). Let $\mu : M \rightarrow Q$ be the homomorphism such that $\mu \big| _L = 0$ and $\mu \big| _A = \mu_0$. Then $\UM = (M, \lambda, \mu)$ is an extended quadratic form over $Q$, and it is free and full. It is also geometric, because for any $x = \sum_{i=1}^k r_ie_i + \sum_{i=1}^k s_if_i \in M$ we have 
\begin{multline*}
\varrho_2 \circ \lambda(x,x) = \varrho_2 \left( 2\sum_{i=1}^k r_is_i + \sum_{i=1}^k s_i^2 \lambda(f_i,f_i) \right) = \sum_{i=1}^k \varrho_2(s_i^2) \varrho_2 \circ \lambda(f_i,f_i) = \\ 
= \sum_{i=1}^k \varrho_2(s_i) \varrho_2 \circ \lambda(f_i,f_i) = \sum_{i=1}^k \varrho_2(s_i) v \circ \mu_0(f_i) = \sum_{i=1}^k v(s_i\mu_0(f_i)) = v \circ \mu(x)
\end{multline*}
Furthermore, $L$ is a lagrangian in $M$, therefore $(\UM; L, L) \in \gc_{2q+1}(Q,v)$. 

Second, we show that if $(\UM; L, L), (\UMp; L', L') \in \gc_{2q+1}(Q,v)$, then $(\UM; L, L) \sim (\UMp; L', L')$. By Theorem \ref{thm:fund-met} a) there is an isomorphism $I : \UM \oplus \UH_{2k} \rightarrow \UMp \oplus \UH_{2l}$ such that $I(L \oplus (\{ 0 \} \times \Z^k)) = L' \oplus (\{ 0 \} \times \Z^l)$ for some $k,l \geq 0$. So we have
\[
\begin{aligned}
(\UM; L, L) &\sim (\UM; L, L) \oplus (\UH_{2k}; \{ 0 \} \times \Z^k, \{ 0 \} \times \Z^k) = \\
 &= (\UM \oplus \UH_{2k}; L \oplus (\{ 0 \} \times \Z^k), L \oplus (\{ 0 \} \times \Z^k)) \cong \\
 &\cong (\UMp \oplus \UH_{2l}; L' \oplus (\{ 0 \} \times \Z^l), L' \oplus (\{ 0 \} \times \Z^l)) = \\
 &= (\UMp; L', L') \oplus (\UH_{2l}; \{ 0 \} \times \Z^l, \{ 0 \} \times \Z^l) \sim \\
 &\sim (\UMp; L', L')
\end{aligned}
\]
using Lemma \ref{lem:hyp-zero} and the isomorphism $I$. Therefore all quasi-formations of the form $(\UM; L, L)$ are in the same equivalence class, which we will denote by $\alpha \in \ell_{2q+1}(Q,v)$.

Third, we show that $\alpha$ is a zero element in $\ell_{2q+1}(Q,v)$. Suppose that $(\UM; L, V) \in \gc_{2q+1}(Q,v)$, then $(\UM; L, L) \in \gc_{2q+1}(Q,v)$ and $[\UM; L, L] = \alpha$. Let $I : \UM \oplus \UM \rightarrow \UM \oplus \UH_{2k}$ be the isomorphism constructed in the proof of Lemma \ref{lem:met-isom}, satisfying $I(L \oplus L) = L \oplus (\{ 0 \} \times \Z^k)$. We also define the basis elements $e_i,f_i,\bar{e}_i,\bar{f}_i,a_i,b_i$ and the integers $d_i$ as in Lemma \ref{lem:met-isom}. Let $v_1, v_2, \ldots , v_k$ be a basis of $V$, then $v_i = \sum_{j=1}^k (p_{ij}e_j + q_{ij}f_j)$ for some $p_{ij}, q_{ij} \in \Z$. We have $I(v_i) = \sum_{j=1}^k (p_{ij}I(e_j) + q_{ij}I(f_j)) = \sum_{j=1}^k (p_{ij}(e_j+b_j) + q_{ij}(f_j+d_jb_j)) = \sum_{j=1}^k (p_{ij}e_j + q_{ij}f_j) + \sum_{j=1}^k (p_{ij} + q_{ij}d_j)b_j = v_i + \sum_{j=1}^k (p_{ij} + q_{ij}d_j)b_j$ and $I(\bar{e}_i)=-b_i$. Thus 
\begin{multline*}
I(V \oplus L) = I(\left< v_1, \ldots , v_k, \bar{e}_1, \ldots , \bar{e}_k \right>) = \\
= \left< v_1 + \sum_{j=1}^k (p_{1j} + q_{1j}d_j)b_j, \ldots , v_k + \sum_{j=1}^k (p_{kj} + q_{kj}d_j)b_j, -b_1, \ldots , -b_k \right> = \\
= \left< v_1, \ldots , v_k, b_1, \ldots , b_k \right> = V \oplus (\{ 0 \} \times \Z^k) \text{\,.}
\end{multline*}
Using the isomorphism $I$ and Lemma \ref{lem:hyp-zero} we get the following: 
\[
\begin{aligned}
(\UM; L, V) \oplus (\UM; L, L) &= (\UM \oplus \UM; L \oplus L, V \oplus L) \cong \\
 &\cong (\UM \oplus \UH_{2k}; L \oplus (\{ 0 \} \times \Z^k), V \oplus (\{ 0 \} \times \Z^k)) = \\
 &= (\UM; L, V) \oplus (\UH_{2k}; \{ 0 \} \times \Z^k, \{ 0 \} \times \Z^k) \sim \\
 &\sim (\UM; L, V)
\end{aligned}
\]
This means that $[\UM; L, V] \oplus \alpha = [\UM; L, V]$. Since this is true for any $[\UM; L, V] \in \ell_{2q+1}(Q,v)$, $\alpha$ is a zero element (hence \textit{the} zero element) in $\ell_{2q+1}(Q,v)$. 
\end{proof}

\begin{rem}
The direct sum of elementary quasi-formations is obviously elementary, hence the elementary elements of the monoid $\ell_{2q+1}(Q,v)$ form a subsemigroup, which we denote by $\varepsilon\ell_{2q+1}(Q,v)$, following Crowley-Sixt \cite{crowley-sixt11}. In general $\varepsilon\ell_{2q+1}(Q,v)$ is not a submonoid of $\ell_{2q+1}(Q,v)$, because if $(\UM; L, V)$ is equivalent to an elementary quasi-formation, then $\mu(V) = Q$, therefore $[\UM; L, L] \not \in \varepsilon\ell_{2q+1}(Q,v)$ if $Q \neq 0$.  
\end{rem}

\begin{rem} \label{rem:ell-nat}
Using Remark \ref{rem:eq-nat}, we can see that $\ell_{2q+1}$ is a functor from the slice category of abelian groups over $\Z_2$ to commutative monoids. First, we define the monoid $\bar{\ell}_{2q+1}(Q,v) = \gcb_{2q+1}(Q,v) / \sim$, the analogue of $\ell_{2q+1}(Q,v)$ allowing not necessarily full forms, which has zero element $[\HH_{2k}]$. Then $\ell_{2q+1}(Q,v)$ is a subsemigroup of $\bar{\ell}_{2q+1}(Q,v)$. Its inclusion has a canonical left inverse $\bar{\ell}_{2q+1}(Q,v) \rightarrow \ell_{2q+1}(Q,v)$ given by $[\UN; K, V] \mapsto [(\UN; K, V) \oplus (\UM; L, L)]$ for any representative $(\UM; L, L)$ of the zero element of $\ell_{2q+1}(Q,v)$ (where $\UN \oplus \UM$ is full, because $\UM$ is full), which is a monoid morphism.

If $f : Q \rightarrow Q'$ is a morphism $(Q,v) \rightarrow (Q',v')$ in the slice category, then, by the naturality of $\sim$, the map $f_* : \gcb_{2q+1}(Q,v) \rightarrow \gcb_{2q+1}(Q',v')$ induces a map $\bar{\ell}_{2q+1}(Q,v) \rightarrow \bar{\ell}_{2q+1}(Q',v')$ between the quotients. By pre-composing this map with the inclusion $\ell_{2q+1}(Q,v) \rightarrow \bar{\ell}_{2q+1}(Q,v)$ and post-composing with $\bar{\ell}_{2q+1}(Q',v') \rightarrow \ell_{2q+1}(Q',v')$ we get the morphism $f_* : \ell_{2q+1}(Q,v) \rightarrow \ell_{2q+1}(Q',v')$ of $\ell$-monoids induced by $f$.
\end{rem}

The next lemma expands on Lemma \ref{lem:hyp-zero} and Proposition \ref{prop:lmonoid-zero} by showing that every quasi-formation of the form $(\UN; K, K)$ has a trivial effect on $\ell_{2q+1}(Q,v)$, even if $\UN$ is not full (so $(\UN; K, K)$ does not represent an element of $\ell_{2q+1}(Q,v)$).

\begin{lem} \label{lem:nkk-zero}
Suppose that $(\UM; L, V) \in \gc_{2q+1}(Q,v)$, $\UN$ is a free geometric (with respect to $v$) metabolic form over $Q$ and $K$ is a lagrangian in $\UN$ (so that $(\UM; L, V) \oplus (\UN; K, K) \in \gc_{2q+1}(Q,v)$). Then $(\UM; L, V) \sim (\UM; L, V) \oplus (\UN; K, K)$. 
\end{lem}

\begin{proof}
We have $(\UM; L, L) \in \gc_{2q+1}(Q,v)$ and by Proposition \ref{prop:lmonoid-zero} b) it represents the zero element in $\ell_{2q+1}(Q,v)$. The same is true for $(\UM \oplus \UN; L \oplus K, L \oplus K) = (\UM; L, L) \oplus (\UN; K, K)$. Therefore we have $(\UM; L, V) \sim (\UM; L, V) \oplus ((\UM; L, L) \oplus (\UN; K, K)) \cong ((\UM; L, V) \oplus (\UN; K, K)) \oplus (\UM; L, L) \sim (\UM; L, V) \oplus (\UN; K, K)$. 
\end{proof}

We can now prove the stable version of Lemma \ref{lem:ru-equiv}: 

\begin{lem} \label{lem:ru-st-equiv}
Suppose that $(\UM; L, V) \in \gc_{2q+1}(Q,v)$. If $\Psi \in \RU_{\st}(\UM,L)$, then 
\[
(\UM; L, V) \sim (\UM; L, \Psi(V)) \text{\,.}
\] 
\end{lem}

\begin{proof}
By Definition \ref{def:ru-st} there is a metabolic form $\UN$ with a lagrangian $K$ such that $\Psi \oplus \id_N \in \RU(\UM \oplus \UN, L \oplus K)$. By Lemma \ref{lem:nkk-zero} it is enough to show that $(\UM \oplus \UN; L \oplus K, V \oplus K) \sim (\UM \oplus \UN; L \oplus K, \Psi(V) \oplus K)$, which holds by Lemma \ref{lem:ru-equiv}.
\end{proof}

\begin{thm} \label{thm:jacobi}
Suppose that $(\UM; L, V) \in \gc_{2q+1}(Q,v)$, and let $K$ be a lagrangian in $\UM$ (so $(\UM; K, L), (\UM; K, V) \in \gc_{2q+1}(Q,v)$). Then 
\[
[\UM; K, L] \oplus [\UM; L, V] = [\UM; K, V] \in \ell_{2q+1}(Q,v) \text{\,.}
\]
\end{thm}

\begin{proof}
By Theorem \ref{thm:fund-met} a) there is an isomorphism $\Phi : \UM \oplus \UH_{2k} \rightarrow \UM \oplus \UH_{2k}$ such that $\Phi(L \oplus (\{ 0 \} \times \Z^k)) = K \oplus (\{ 0 \} \times \Z^k)$ for some $k \geq 0$. Let $\UtM = \UM \oplus \UH_{2k}$, $\tilde{L} = L \oplus (\{ 0 \} \times \Z^k)$, $\tilde{K} = K \oplus (\{ 0 \} \times \Z^k)$ and $\tilde{V} = V \oplus (\Z^k \times \{ 0 \})$, then $[\UM; L, V] = [\UtM; \tilde{L}, \tilde{V}]$, $[\UM; K, V] = [\UtM; \tilde{K}, \tilde{V}]$ and (by Lemma \ref{lem:hyp-zero}) $[\UM; K, L] = [\UtM; \tilde{K}, \tilde{L}]$.

By Theorem \ref{thm:ru-wall} $\Phi \oplus \Phi^{-1} \in \RU_{\st}(\UtM \oplus \UtM, \tilde{K} \oplus \tilde{K})$. Therefore we have 
\[
\begin{aligned}
(\UtM; \tilde{K}, \tilde{L}) \oplus (\UtM; \tilde{L}, \tilde{V}) &\cong (\UtM; \tilde{K}, \tilde{L}) \oplus (\UtM; \tilde{K}, \Phi(\tilde{V})) = \\
 &= (\UtM \oplus \UtM; \tilde{K} \oplus \tilde{K}, \tilde{L} \oplus \Phi(\tilde{V})) \sim \\
 &\sim (\UtM \oplus \UtM; \tilde{K} \oplus \tilde{K}, \tilde{K} \oplus \tilde{V}) = \\
 &= (\UtM; \tilde{K}, \tilde{K}) \oplus (\UtM; \tilde{K}, \tilde{V}) \sim \\ 
 &\sim (\UtM; \tilde{K}, \tilde{V}) 
\end{aligned}
\]
using the isomorphism $\Phi$ and Lemmas \ref{lem:ru-st-equiv} and \ref{lem:nkk-zero}. 
\end{proof}

These results generalise the relations that were used to define the classical L-groups and Kreck's monoid $l_{2q+1}(e)$. In \cite[Section 6]{kreck99} the equivalence relation on pairs $(\UH_{2k}, V)$ is given by the group $\RU^+(\Z)$, which is analogous to the relation in Lemma \ref{lem:ru-st-equiv} (see Proposition \ref{prop:L0} below for the correspondence between pairs and quasi-formations). The analogue of the identity of Theorem \ref{thm:jacobi} has been used to reformulate the definitions of the classical L-groups and the $l$-monoids in terms of (quasi-)formations (see eg.\ \cite{ranicki80}, \cite{crowley-sixt11}).

\subsection{The group $L_{2q+1}(Q,v)$}

In this section we consider the extended version of the L-groups.

\begin{defin}
For an abelian group $Q$ with a homomorphism $v : Q \rightarrow \Z_2$, a torsion abelian group $R$ and an even integer $q \geq 0$ let
\[
L^T_{2q+1}(Q,v,R) = \{ [\UM; L, K] \in \ell^T_{2q+1}(Q,v,R) \mid \text{$K$ is a free lagrangian in $\UM$} \}
\] 
and 
\[
L_{2q+1}(Q,v) = L^T_{2q+1}(Q,v,0)
\] 
\end{defin}

The first important application of Theorem \ref{thm:jacobi} is showing that the elements of $L_{2q+1}(Q,v)$ are invertible, hence it is a group:

\begin{prop} \label{prop:L-inverse}
$L_{2q+1}(Q,v)$ is the group of invertible elements of $\ell_{2q+1}(Q,v)$.
\end{prop}

\begin{proof}
First we show that every element of $L_{2q+1}(Q,v)$ has an inverse. If $[\UM; L, K] \in L_{2q+1}(Q,v)$, then $K$ is a lagrangian in $\UM$, so $(\UM; K, L) \in \gc_{2q+1}(Q,v)$. By Theorem \ref{thm:jacobi} $[\UM; L, K] \oplus [\UM; K, L] = [\UM; L, L]$, which is the zero element, therefore $[\UM; K, L]$ is the inverse of $[\UM; L, K]$. 

Next we show that every invertible element is contained in $L_{2q+1}(Q,v)$. For an $(\UM; L, V) \in \gc_{2q+1}(Q,v)$ let $\alpha(\UM; L, V)$ denote the subgroup $\mu(V) \leq Q$ (where $\UM = (M, \lambda, \mu)$). Then $\alpha(\UM; L, V)$ is invariant under the elementary equivalences that define $\sim$, so it induces a well-defined function on $\ell_{2q+1}(Q,v)$. Moreover, we have $\alpha([\UM; L, V] \oplus [\UMp; L', V']) = \alpha([\UM; L, V]) + \alpha([\UMp; L', V'])$ and $\alpha([\UM; L, L]) = 0 \leq Q$. This implies that if $[\UM; L, V]$ is invertible, then $\alpha([\UM; L, V]) = 0$, that is, $\mu \big| _V = 0$. Similarly, let $\beta(\UM; L, V) \leq \Z$ denote the image of the homomorphism $V \otimes V \rightarrow \Z$ corresponding to the bilinear function $\lambda \big| _{V \times V} : V \times V \rightarrow \Z$. Then $\beta$ also induces a well-defined function on $\ell_{2q+1}(Q,v)$, we have $\beta([\UM; L, V] \oplus [\UMp; L', V']) = \beta([\UM; L, V]) + \beta([\UMp; L', V'])$ and $\beta([\UM; L, L]) = 0 \leq \Z$. So if $[\UM; L, V]$ is invertible, then $\beta([\UM; L, V]) = 0$, that is, $\lambda \big| _{V \times V} = 0$. This means that if $[\UM; L, V]$ is invertible, then $V$ is a lagrangian, therefore $[\UM; L, V] \in L_{2q+1}(Q,v)$. 
\end{proof}

Next we will prove that if $Q$ is free, then $L_{2q+1}(Q,v) \cong 0$ (Theorem \ref{thm:l-group-zero}). The first step towards that is the following:

\begin{prop} \label{prop:L0}
There is a canonical isomorphism of monoids $\ell_{2q+1}(0,0) \cong l_{2q+1}(e)$, where $l_{2q+1}(e)$ is Kreck's $l$-monoid associated to the trivial group. It restricts to an isomorphism $L_{2q+1}(0,0) \cong L_{2q+1}$, where $L_{2q+1}$ denotes the classical L-group used in simply-connected surgery.
\end{prop}

\begin{proof}
Kreck \cite[Section 6]{kreck99} defines $l_{2q+1}(e)$ as follows: an element of $l_{2q+1}(e)$ is represented by a pair $(\UH_{2k}, V)$ for some $k \geq 0$, where $V$ is a half-rank summand in $\UH_{2k}$. The equivalence relation on such pairs is generated by the elementary equivalences $(\UH_{2k}, V) \sim (\UH_{2k} \oplus \UH_2, V \oplus (\Z \times \{0 \} ))$ and $(\UH_{2k}, V) \sim (\UH_{2k}, \Phi(V))$ for $\Phi \in RU(\UH_{2k}, \{ 0 \} \times \Z^k)$. 

We define $F : l_{2q+1}(e) \rightarrow \ell_{2q+1}(0,0)$ by $F([\UH_{2k}, V]) = [\UH_{2k}; \{0 \} \times \Z^k, V]$, this map is well-defined by Lemma \ref{lem:ru-equiv}. 

To define $G : \ell_{2q+1}(0,0) \rightarrow l_{2q+1}(e)$, note that if $[\UH; L, V] \in \ell_{2q+1}(0,0)$, then $\UH$ is hyperbolic (see Lemma \ref{lem:hyp-equiv}), so by Corollary \ref{cor:fund-hyp} there is a $k \geq 0$ and an isomorphism $I : \UH \rightarrow \UH_{2k}$ such that $I(L) = \{0 \} \times \Z^k$. Let $G([\UH; L, V]) = [\UH_{2k}, I(V)]$, we will check that this is well-defined. If $I : \UH \rightarrow \UH_{2k}$ and $I' : \UH \rightarrow \UH_{2k}$ are two isomorphisms as above, then $I' \circ I^{-1}(\{0 \} \times \Z^k) = \{0 \} \times \Z^k$, so $I' \circ I^{-1} \in RU(\UH_{2k}, \{ 0 \} \times \Z^k)$ and $[\UH_{2k}, I(V)] = [\UH_{2k}, I'(V)]$. If $(\UH; L, V) \cong (\UHp; L', V')$, then $I' : \UHp \rightarrow \UH_{2k}$ can be chosen to be the isomorphism $\UHp \rightarrow \UH$ composed with $I : \UH \rightarrow \UH_{2k}$, so that $(\UH_{2k}, I(V)) = (\UH_{2k}, I'(V'))$. If $(\UHp; L', V') = (\UH; L, V) \oplus \HH_2$, then we can choose $I' = I \oplus \id_{\Z^2}$, so that $(\UH_{2k+2}, I'(V')) = (\UH_{2k} \oplus \UH_2, I(V) \oplus (\Z \times \{0 \} )) \sim (\UH_{2k}, I(V))$. Finally consider the equivalence $(\UHp \oplus \UH_2; L \oplus (\{ 0 \} \times \Z), V) \sim (\UHp \oplus \UH_2; L \oplus (\Z \times \{ 0 \}), V)$. If $I : \UHp \oplus \UH_2 \rightarrow \UH_{2k}$ is an isomorphism such that $I(L \oplus (\{ 0 \} \times \Z)) = \{0 \} \times \Z^k$, then $I \circ (\id_{H'} \oplus \sigma)(L \oplus (\Z \times \{ 0 \})) = \{0 \} \times \Z^k$ and $I \circ (\id_{H'} \oplus \sigma) \circ I^{-1} \in RU(\UH_{2k}, \{ 0 \} \times \Z^k)$, therefore $[\UH_{2k}, I(V)] = [\UH_{2k}, I \circ (\id_{H'} \oplus \sigma)(V)]$.

$F$ and $G$ are obviously inverses of each other, and they preserve direct sums, therefore they are isomorphisms. 

The pair $(\UH_{2k}, V)$ represents an element in $L_{2q+1}$ if and only if $V$ is a lagrangian in $\UH_{2k}$, which holds if and only if $[\UH_{2k}; \{0 \} \times \Z^k, V] \in L_{2q+1}(0,0)$, showing that $F$ restricts to an isomorphism $L_{2q+1} \rightarrow L_{2q+1}(0,0)$. (Alternatively, note that $L_{2q+1}(0,0)$ and $L_{2q+1}$ are both equal to the group of invertible elements in the corresponding monoid.)
\end{proof}

\begin{rem} \label{rem:L0}
The classical L-group $L_{2q+1}$ is known to vanish -- this follows from the results of Kervaire-Milnor \cite{kervaire-milnor63}, which show that the surgery process can always be completed for simply-connected odd-dimensional manifolds. In particular, this implies that $L_{2q+1}(0,0) \cong 0$.
\end{rem}

\begin{rem} 
Notice that Proposition \ref{prop:L-inverse} generalises Kreck's construction of the inverse of an element of $L_{2q+1} \subset l_{2q+1}(e)$ (see \cite[p.\ 733]{kreck99}). Such an element is represented by a pair $(\UH_{2k}, L)$, with $L$ a lagrangian in $\UH_{2k}$, so there is an $A \in \Aut(\UH_{2k})$ such that $L = A(\{0 \} \times \Z^k)$ (see Corollary \ref{cor:fund-hyp}). Then \cite[Lemma 6.2]{wall-scm} is used to show that $[\UH_{2k}, A^{-1}(\{0 \} \times \Z^k)]$ is the inverse of this element. Under the map $F : l_{2q+1}(e) \rightarrow \ell_{2q+1}(0,0)$ the element and its inverse correspond to $[\UH_{2k}; \{0 \} \times \Z^k, L]$ and $[\UH_{2k}; \{0 \} \times \Z^k, A^{-1}(\{0 \} \times \Z^k)]$ respectively. We have $(\UH_{2k}; \{0 \} \times \Z^k, A^{-1}(\{0 \} \times \Z^k)) \cong (\UH_{2k}; A(\{0 \} \times \Z^k), A \circ A^{-1}(\{0 \} \times \Z^k)) = (\UH_{2k}; L, \{0 \} \times \Z^k)$ (using the automorphism $A$), coinciding with the inverse defined in Proposition \ref{prop:L-inverse}.
\end{rem}

\begin{lem} \label{lem:hyp-cancel}
Suppose that $(\UM; L, V) \in \gc_{2q+1}(Q,v)$. If $\UH$ is a hyperbolic form and $K$ and $J$ are lagrangians in $\UH$, then $(\UM; L, V) \oplus (\UH; K, J) \sim (\UM; L, V)$. 
\end{lem}

\begin{proof}
The quasi-formation $(\UH; K, J)$ represents an element of $L_{2q+1}(0,0)$. By Proposition \ref{prop:L0} and Remark \ref{rem:L0} $L_{2q+1}(0,0) \cong 0$, so $(\UH; K, J) \sim (\UH; K, K)$ in $\gc_{2q+1}(0,0)$. These elements are also equivalent in $\gcb_{2q+1}(0,0)$, and hence in $\gcb_{2q+1}(Q,v)$, by our observations in Remark \ref{rem:eq-nat} (in particular, the naturality of $\sim$ under the unique morphism $(0,0) \rightarrow (Q,v)$ in the slice category). Since $\sim$ respects direct sums, $(\UM; L, V) \oplus (\UH; K, J) \sim (\UM; L, V) \oplus (\UH; K, K)$ in $\gcb_{2q+1}(Q,v)$, and hence in $\gc_{2q+1}(Q,v)$. 
By applying Lemma \ref{lem:nkk-zero} we get that $(\UM; L, V) \oplus (\UH; K, J) \sim (\UM; L, V)$. 
\end{proof}

\begin{thm} \label{thm:l-group-zero}
If $Q$ is free, then $L_{2q+1}(Q,v) \cong 0$. 
\end{thm}

\begin{proof}
Let $[\UM; L, K]$ be an arbitrary element in $L_{2q+1}(Q,v)$. 

Let $J = K \cap L$. Since $K$ and $L$ are direct summands in $M$, the same holds for $J$ (see Lemma \ref{lem:summand}). Let $X$ be a direct complement to $J$ in $M$, $K' = K \cap X$ and $L' = L \cap X$, then $K = J \oplus K'$ and $L = J \oplus L'$.

Let $M' = J \oplus X^{\perp}$. It follows from Lemma \ref{lem:perp-nonsing} that this is a subgroup of $M$ and $\lambda \big| _{M' \times M'}$ is nonsingular. Let $H = (M')^{\perp}$, by Lemma \ref{lem:nonsing-summand} $M = M' \oplus H$. Let $\UMp = (M', \lambda \big| _{M' \times M'}, \mu \big| _{M'})$ and $\UH = (H, \lambda \big| _{H \times H}, \mu \big| _H)$, then $\UM = \UMp \oplus \UH$, and since $\UM$ is nonsingular, $\UH$ is nonsingular too.

Next we show that $J$ is a lagrangian in $\UMp$. It is a direct summand in $M'$ by the latter's definition. We have $\lambda \big| _{J \times J} = 0$ and $\mu \big| _J = 0$, because $J \leq L$. Finally $\rk M' = \rk J + \rk X^{\perp} = \rk J + (\rk M - \rk X) = \rk J + \rk J = 2 \rk J$. So $J$ is a lagrangian in $\UMp$, and this means that $\UMp$ is metabolic. 

Next we show that $L'$ is a lagrangian in $\UH$. We have $L' = L \cap X = L^{\perp} \cap X \leq J^{\perp} \cap X = (J \oplus X^{\perp})^{\perp} = H$. By construction $L'$ is a direct summand in $L$, hence in $M$, hence in $H$. We have $\lambda \big| _{L' \times L'} = 0$ and $\mu \big| _{L'} = 0$, because $L' \leq L$. Finally $\rk L' = \rk L - \rk J = \frac{1}{2} \rk M - \frac{1}{2} \rk M' = \frac{1}{2}(\rk M - \rk M') = \frac{1}{2} \rk H$. Therefore $L'$ is a lagrangian in $\UH$, and since $\UH$ is nonsingular, it is metabolic. We can show similarly that $K'$ is also a lagrangian in $\UH$.

Consider the subgroup $K' \oplus L' \leq H$ (note that $K' \cap L' = \{ 0 \}$). Since $K'$ and $L'$ are lagrangians in $H$, we have $\rk (K' \oplus L') = \rk H$, therefore $H / (K' \oplus L')$ is a torsion group. The homomorphism $\mu \big| _H : H \rightarrow Q$ induces a homomorphism $H / (K' \oplus L') \rightarrow Q$, because $\mu \big| _{K' \oplus L'} = 0$. Since $Q$ is assumed to be free, this induced homomorphism is trivial, therefore $\mu \big| _H = 0$. This implies that $\UMp$ is full (because $\UM$ is full). Moreover, since $\UM$ is geometric (with respect to $v$), the same holds for $\UH$, so $\UH$ is even. By applying Lemma \ref{lem:hyp-equiv}, we get that $\UH$ is hyperbolic.

Then we have $(\UM; L, K) = (\UMp \oplus \UH; J \oplus L', J \oplus K') = (\UMp; J, J) \oplus (\UH; L', K') \sim (\UMp; J, J)$ by Lemma \ref{lem:hyp-cancel}. By Proposition \ref{prop:lmonoid-zero} b) $(\UMp; J, J)$ represents the zero element in $\ell_{2q+1}(Q,v)$, and hence in $L_{2q+1}(Q,v)$. Therefore $[\UM; L, K]$ is the zero element in $L_{2q+1}(Q,v)$, showing that $L_{2q+1}(Q,v) \cong 0$. 
\end{proof}

\subsection{The $\ell$-monoid associated to a space}

\begin{defin}
For a stable bundle $\xi$ over a space $B$ and an integer $q \geq 0$ let $\hat{v}_q(\xi) : H_q(B) \rightarrow \Z_2$ denote the homomorphism determined by the Wu class $v_q(\xi) \in H^q(B;\Z_2)$ of $\xi$.
\end{defin}

\begin{defin} \label{def:l-space}
For a simply-connected space $B$, a stable bundle $\xi$ over it and an even integer $q \geq 0$ let 
\[
\begin{aligned}
\gc_{2q+1}(B,\xi) &= \gc_{2q+1}(H_q(B),\hat{v}_q(\xi)) \text{\,,} \\
\gs_{2q+1}(B,\xi) &= \gs_{2q+1}(H_q(B),\hat{v}_q(\xi)) \text{\,,} \\
\s_{2q+1}(B,\xi) &= \s_{2q+1}(H_q(B),\hat{v}_q(\xi)) \text{\,,} \\
\ell_{2q+1}(B,\xi) &= \ell_{2q+1}(H_q(B),\hat{v}_q(\xi)) \text{\,,} \\
L_{2q+1}(B,\xi) &= L_{2q+1}(H_q(B),\hat{v}_q(\xi)) \text{\,,} 
\end{aligned}
\]
and 
\[
\begin{aligned}
\gc^T_{2q+1}(B,\xi) &= \gc^T_{2q+1}(H_q(B),\hat{v}_q(\xi),\Tor H_{q-1}(B)) \text{\,,} \\
\gs^T_{2q+1}(B,\xi) &= \gs^T_{2q+1}(H_q(B),\hat{v}_q(\xi),\Tor H_{q-1}(B)) \text{\,,} \\
\s_{2q+1}^T(B,\xi) &= \s^T_{2q+1}(H_q(B),\hat{v}_q(\xi),\Tor H_{q-1}(B)) \text{\,,} \\
\ell^T_{2q+1}(B,\xi) &= \ell^T_{2q+1}(H_q(B),\hat{v}_q(\xi),\Tor H_{q-1}(B)) \text{\,,} \\
L_{2q+1}^T(B,\xi) &= L^T_{2q+1}(H_q(B),\hat{v}_q(\xi),\Tor H_{q-1}(B)) \text{\,.} 
\end{aligned}
\]
\end{defin}

\section{Normal maps and Q-forms} \label{s:nm-qf}

\subsection{Normal maps} \label{ss:nm}

We recall the terminology related to normal maps and normal smoothings, see \cite{kreck99}. Suppose that $B$ is a space with a stable vector bundle $\xi$ over it (equivalently, with a map $B \rightarrow BO$). 

\begin{defin} \label{def:nm}
Let $M$ be a manifold. A \emph{normal map} over $(B,\xi)$ is a pair $(f, \bar{f})$, where $f : M \rightarrow B$ is a continuous map and $\bar{f} : \nu_M \rightarrow \xi$ is a bundle map from the stable normal bundle $\nu_M$ of $M$ to $\xi$ covering $f$. We will usually suppress $\bar{f}$ from the notation and say that $f$ is a normal map.
\end{defin}

\begin{defin} \label{def:nb}
Let $M_i$ be a closed manifold and $f_i : M_i \rightarrow B$ be a normal map over $(B,\xi)$ ($i=0,1$). A \emph{normal bordism} between $f_0$ and $f_1$ is a normal map $F : Y \rightarrow B$ such that $Y$ is a cobordism between $M_0$ and $M_1$, and $F \big| _{M_i} = f_i$ (and $\bar{F} \big| _{\nu_{M_i}} = \bar{f}_i$). 

If the manifolds $M_i$ have boundaries and there is an identification $\partial M_0 \approx \partial M_1$ such that the normal maps $f_i \big| _{\partial M_i}$ coincide, then \emph{normal bordism rel boundary} can be defined analogously, with the additional condition that $F$ restricts to the trivial normal bordism between the boundaries.
\end{defin}

\begin{defin}  \label{def:njs}
Let $j \geq 0$ be an integer. A normal map $f : M \rightarrow B$ over $(B,\xi)$ is called a \emph{normal $j$-smoothing}, if it is $(j+1)$-connected (ie.\ $\pi_i(f) : \pi_i(M) \rightarrow \pi_i(B)$ is an isomorphism for $i \leq j$ and $\pi_{j+1}(f)$ is surjective).
\end{defin}

\begin{defin} \label{def:j-type}
Let $j \geq 0$ be an integer. We say that $(B,\xi)$ is the \emph{normal $j$-type} of a manifold $M$, if $M$ admits a normal $j$-smoothing over $(B,\xi)$, and the classifying map $g : B \rightarrow BO$ of $\xi$ is $(j+1)$-co-connected (ie.\ $\pi_i(g) : \pi_i(B) \rightarrow \pi_i(BO)$ is an isomorphism for $i > j+1$ and $\pi_{j+1}(g)$ is injective). 
\end{defin}

Equivalently, $(B,\xi)$ is the normal $j$-type of $M$, if $B$ is the $(j+1)$-stage of the Moore-Postnikov tower of the classifying map $M \rightarrow BO$ of $\nu_M$, and $\xi$ is the pullback of the universal bundle to $B$. In particular, the normal $j$-type of $M$ always exists, and it is unique up to homotopy equivalence that commutes with the maps $M \rightarrow B$ and $B \rightarrow BO$ (see Baues \cite[Section (5.3)]{baues77}). Consequently, if the normal $j$-type $(B,\xi)$ is fixed, then any two normal $j$-smoothings $M \rightarrow B$ differ by an automorphism of $(B,\xi)$.

\subsection{The Q-form of a map} \label{ss:qf}

Next we define the Q-form of a map, and prove some of its basic properties. Fix a topological space $B$ and let $q \geq 2$ be an even integer. Let $M$ be a closed oriented $2q$-manifold with a map $f : M \rightarrow B$.

\begin{defin} \label{def:qf}
The \emph{Q-form} of $f$ is the extended quadratic form
\[
E_q(M,f) = (H_q(M), \lambda, \mu)
\] 
over $H_q(B)$, where $\lambda : H_q(M) \times H_q(M) \rightarrow \Z$ is the intersection form of $M$ and $\mu = H_q(f) : H_q(M) \rightarrow H_q(B)$. (The bilinear function $\lambda$ is symmetric, because $q$ is even.)
\end{defin}

It follows from Poincar\'e-duality that $E_q(M,f)$ is nonsingular. 

From the definitions we easily get the following lemmas:  

\begin{lem} \label{lem:qf-hyp}
If $t : S^q \times S^q \rightarrow B$ is nullhomotopic, then $E_q(S^q \times S^q,t) \cong \UH_2$.
\qed
\end{lem}

\begin{lem} \label{lem:qf-sum}
If $N$ is another closed oriented $2q$-manifold with a map $g : N \rightarrow B$, then $E_q(M \sqcup N,f \sqcup g) = E_q(M \# N,f \# g) = E_q(M,f) \oplus E_q(N,g)$. 
\qed
\end{lem}

\begin{lem} \label{lem:qf-neg}
If $-M$ denotes the same manifold $M$ with opposite orientation, then $E_q(-M,f) = -E_q(M,f)^*$.
\end{lem}

\begin{proof}
Changing the orientation of $M$ changes the sign of the intersection form $\lambda$, therefore $E_q(-M,f) = (H_q(M), -\lambda, \mu)= -E_q(M,f)^*$. 
\end{proof}

\begin{lem} \label{lem:qf-full}
If $f$ is $q$-connected, then $E_q(M,f)$ is full.
\qed
\end{lem}

Now let $\xi$ be a stable bundle over $B$, and recall that $\hat{v}_q(\xi) : H_q(B) \rightarrow \Z_2$ denotes the homomorphism determined by the Wu class $v_q(\xi) \in H^q(B;\Z_2)$ of $\xi$.

\begin{lem} \label{lem:qf-geom}
If $f^*(\xi) \cong \nu_M$, then $E_q(M,f)$ is geometric with respect to $\hat{v}_q(\xi)$.
\end{lem}

\begin{proof}
By the assumption $f^*(v_q(\xi)) = v_q(f^*(\xi)) = v_q(\nu_M) = v_q(M) \in H^q(M; \Z_2)$.

Let $x \in H_q(M)$ be an arbitrary element, let $\alpha \in H^q(M)$ be its Poincar\'e-dual, and let $\alpha_2 \in H^q(M; \Z_2)$ be the mod $2$ reduction of $\alpha$. Since $\lambda(x,x) = \left< [M], \alpha \cup \alpha \right> = \left< [M], \alpha^2 \right>$ and $\alpha_2^2 = Sq^q(\alpha_2) = \alpha_2 \cup v_q(M)$, we have
\begin{multline*}
\varrho_2 \circ \lambda(x,x) = \varrho_2(\left< [M], \alpha^2 \right>) = \left< [M], \alpha_2^2 \right> = \left< [M], \alpha_2 \cup v_q(M) \right> = \left< x, v_q(M) \right> = \\
= \left< f_*(x), v_q(\xi) \right> = \left< \mu(x), v_q(\xi) \right> = \hat{v}_q(\xi) \circ \mu(x) \text{\,.}
\end{multline*}
Therefore $E_q(M,f)$ is geometric with respect to $\hat{v}_q(\xi)$. 
\end{proof}

\begin{lem} \label{lem:boundary-lagr}
Suppose that $M$ is the boundary of some oriented manifold $W$ and there is a map $F : W \rightarrow B$ such that $F \big| _M = f$. If $X = \Image(H_{q+1}(W, M) \rightarrow H_q(M))$ denotes the image of the boundary homomorphism, then 

a) $\rk X = \frac{1}{2} \rk H_q(M)$.

b) $\lambda \big| _{X \times X} = 0$.

c) $\mu \big| _X = 0$.
\end{lem}

\begin{proof}
a) and b) are well-known, see eg.\ \cite[\S8]{hirzebruch95}.

c) Regarding $F$ as a map of pairs $(W, M) \rightarrow (B, B)$, we get a morphism between the homological long exact sequences, and a commutative diagram: 
\[
\xymatrix{
H_{q+1}(W, M) \ar[r] \ar[d]_-{H_q(F)} & H_q(M) \ar[d]^{H_q(f)} \\
H_{q+1}(B, B) \ar[r] & H_q(B) 
}
\]
Since $H_{q+1}(B, B) \cong 0$, the composition $H_{q+1}(W, M) \rightarrow H_q(M) \stackrel{\mu}{\rightarrow} H_q(B)$ is trivial. 
\end{proof}

\begin{lem} \label{lem:bound-qf-met}
Suppose that $H_q(B)$ is free. If $M$ is the boundary of some oriented manifold $W$ and there is a map $F : W \rightarrow B$ such that $F \big| _M = f$, then $E_q(M,f)$ is metabolic. 
\end{lem}

\begin{proof}
Let $X = \Image(H_{q+1}(W, M) \rightarrow H_q(M))$, then $\rk X = \frac{1}{2} \rk H_q(M)$, $\lambda \big| _{X \times X} = 0$ and $\mu \big| _X = 0$ by Lemma \ref{lem:boundary-lagr}. 

Let $Y = \{ y \in H_q(M) \mid ky \in X \text{ for some $k \in \Z \setminus \{ 0 \}$} \}$. Then $Y$ is a direct summand in $H_q(M)$ and $\Tor H_q(M) \leq Y$ by Lemma \ref{lem:summand}. Since $X \leq Y$ and $Y/X$ is a torsion group, $\rk Y = \rk X$. We have $\lambda \big| _{Y \times Y} = 0$, because for any $y_1, y_2 \in Y$ there is $k_1, k_2 \in \Z \setminus \{ 0 \}$ such that $k_1y_1, k_2y_2\in X$, so $\lambda(k_1y_1, k_2y_2) = 0$. 

The homomorphism $\mu \big| _Y : Y \rightarrow H_q(B)$ induces a homomorphism $Y/X \rightarrow H_q(B)$, because $\mu \big| _X = 0$. Since $Y/X$ is a torsion group and $H_q(B)$ is free, this induced homomorphism is trivial, therefore $\mu \big| _Y = 0$. 

Therefore $Y$ is a T-lagrangian in $E_q(M,f)$. Since $E_q(M,f)$ is nonsingular, this means that it is metabolic. 
\end{proof}

\section{The extended surgery obstruction $\theta_{W,F}$} \label{s:obstr-def}

The following setting will be used throughout this section: Let $B$ be a simply-connected space, $\xi$ a stable vector bundle over $B$ and $q \geq 2$ an even integer. Let $M_0$ and $M_1$ be closed oriented $2q$-manifolds, and let $W$ be an oriented cobordism between $M_0$ and $M_1$, ie.\ $\partial W = M_0 \sqcup (-M_1)$. Let $f_0 : M_0 \rightarrow B$ and $f_1 : M_1 \rightarrow B$ be normal $(q{-}1)$-smoothings over $(B,\xi)$ (note that we suppress the bundle maps from the notation, see Section \ref{ss:nm}). Let $F : W \rightarrow B$ be a normal bordism between $f_0$ and $f_1$, and assume that $F$ is also a normal $(q{-}1)$-smoothing (ie.\ it is $q$-connected). Assume also that $\chi(M_0)=\chi(M_1)$. 

We will define an invariant $\theta_{W,F} \in \s^T_{2q+1}(B,\xi)$. This invariant, and its equivalence class $[\theta_{W,F}] \in \ell^T_{2q+1}(B,\xi)$ is the \emph{extended surgery obstruction} associated to the normal bordism $F : W \rightarrow B$.

To define $\theta_{W,F}$ we need to make two choices (and then we will show that $\theta_{W,F}$ is independent of these choices). First, we choose a CW-decomposition $c$ of $W$ (ie.\ a homeomorphism between a CW-complex and $W$). Let $\sk_q W$ denote the $q$-skeleton of $W$ with respect to the CW-decomposition $c$. The inclusion $\sk_q W \rightarrow W$ is isotopic to some embedding $\sk_q W \rightarrow \interior W$ into the interior of $W$. Our second choice is a closed regular neighbourhood $U \subset \interior W$ of such an embedding.

After making these two choices, we will define a quasi-formation $\theta_{U,F} \in \gc^T_{2q+1}(B,\xi)$. Then we define $\theta_{c,F} \in \gs^T_{2q+1}(B,\xi)$ to be the isomorphism class of $\theta_{U,F}$ and we show that it only depends on the choice of $c$, but not on the choice of $U$. Next we define $\theta_{W,F}$ to be the stable isomorphism class represented by $\theta_{c,F}$  and show that it does not depend on the choice of $c$. Finally we consider $[\theta_{W,F}]$ and show that it depends only on the normal bordism class (rel boundary) of $F : W \rightarrow B$. 

To sum up, we will consider the following invariants: 
\[
\begin{aligned}
\theta_{U,F} &\in \gc^T_{2q+1}(B,\xi) \\
\theta_{c,F} &\in \gs^T_{2q+1}(B,\xi) \\
\theta_{W,F} &\in \s_{2q+1}^T(B,\xi) \\
[\theta_{W,F}] &\in \ell^T_{2q+1}(B,\xi)
\end{aligned}
\]

We will prove that $W$ is an h-cobordism if and only if $\theta_{W,F}$ is elementary, and $F$ is normally bordant (rel boundary) to an $F' : W' \rightarrow B$ with $W'$ an h-cobordism if and only if $[\theta_{W,F}]$ is elementary, justifying the name ``obstruction" for the invariants $\theta_{W,F}$ and $[\theta_{W,F}]$.

\vspace{\baselineskip}

Our construction of the extended surgery obstruction is based on Kreck's modified surgery obstruction \cite[Section 6]{kreck99}, with changes that give it nicer algebraic properties. As a result, under some assumptions, we can completely determine the obstruction in terms of the Q-forms of $f_0$ and $f_1$ (see Theorem \ref{thm:theta-calc} and Remark \ref{rem:theta-calc}). This comes at the cost of added technical difficulties in proving the well-definedness and completeness of the obstruction (see Sections \ref{ss:def:tw} and \ref{ss:real-surg}). We summarise the two constructions and the main similarities and differences below. 

The definition of the modified surgery obstruction $\theta(W,\bar{\nu}) \in l_{2q+1}(e)$ starts with the choice of an embedding $\bigsqcup^k S^q \times D^{q+1} \rightarrow W$ representing a generating set of $\Ker(\pi_q(F) : \pi_q(W) \rightarrow \pi_q(B))$, and $U \subset W$ is taken to be its image. The embedding identifies the intersection form of $\partial U$ with that of $\bigsqcup^k S^q \times S^q$ (the standard hyperbolic form) such that $\Image(H_{q+1}(U, \partial U) \rightarrow H_q(\partial U))$ corresponds to one of the standard lagrangians, and the obstruction is given by the half-rank direct summand $V \leq \Z^{2k}$ corresponding to $\Image(H_{q+1}(W \setminus \interior U, \partial U \sqcup M_0) \rightarrow H_q(\partial U))$. 

We define $\theta_{U,F} \in \gc^T_{2q+1}(B,\xi)$ as a quasi-formation on $E_q(\partial U, F \big| _{\partial U})$ for $U$ a neighbourhood of a $q$-skeleton of $W$ (see Definitions \ref{def:thetaU-M} and \ref{def:thetaU}). While the submanifold $U$ is different, the lagrangian $L$ and half-rank summand $V$ are defined using the same formulas (see Definitions \ref{def:thetaU-L} and \ref{def:thetaU-V}). Because of the choice of a more complicated $U$, the form $E_q(\partial U, F \big| _{\partial U})$ is only metabolic, and the obstruction $[\theta_{W,F}]$ lives in the extended $\ell$-monoid $\ell^T_{2q+1}(B,\xi)$, rather than $l_{2q+1}(e)$. 

After representatives are defined, the obstructions are shown to be well-defined for a given $W$. In the case of $\theta(W,\bar{\nu})$ this is done by examining the steps needed to transform one generating set of $\Ker(\pi_q(F))$ to another. As that method relies on the special form of $U \approx \bigsqcup^k S^q \times D^{q+1}$, a new approach is needed in the extended setting (see Theorem \ref{thm:tw-welldef}). We will prove that the $q$-skeleton of $W$ is unique up to homotopy equivalence and stabilisation (see Lemma \ref{lem:cwstab-he}), and check that the stable isomorphism class of $\theta_{U,F}$ is preserved, to get an invariant $\theta_{W,F} \in \s_{2q+1}^T(B,\xi)$. A similar refinement of $\theta(W,\bar{\nu})$ that depends on $W$ (rather than its normal bordism class) has not been defined in \cite{kreck99}, but it would live in the set of equivalence classes of pairs $(\UH_{2k},V)$ taken up to stabilisation and automorphisms of $\UH_{2k}$ preserving the standard lagrangian.

The next step is showing that the obstructions are invariant under normal bordism (equivalently, surgeries) of $W$. In the case of $\theta(W,\bar{\nu})$ we can assume that the surgery is done along one of the chosen embeddings $S^q \times D^{q+1} \rightarrow W$, and then its effect on the obstruction is given by a transposition, a generator of $\RU^+(\Z)$. For $[\theta_{W,F}]$ we construct compatible CW-decompositions on $W$ and on the result of the surgery such that the surgery is performed along a cell $S^q \subset \sk_q W$, and show that the obstruction changes by an elementary equivalence of $\sim$ (see Theorem \ref{thm:theta-bordism-inv}). 

Finally, the obstructions are complete. For $[\theta_{W,F}]$ we divide the proof into two parts. First, we prove that $W$ is an h-cobordism if and only if $\theta_{W,F}$ is elementary (see Proposition \ref{prop:hcob-elem}) using the exact sequence of the triple $(W, W \setminus \interior U, M_0)$, similarly to the analogous part of the argument for $\theta(W,\bar{\nu})$. The main difference is in the second part, showing that if $\theta_{W,F}$ is equivalent to an elementary quasi-formation (ie.\ $[\theta_{W,F}]$ is elementary), then $F : W \rightarrow B$ is normally bordant to an h-cobordism. This requires realising an equivalence $\theta_{W,F} \sim x$ of quasi-formations geometrically, that is, finding an $F' : W' \rightarrow B$ with $\theta_{W',F'} \si x$ (see Theorem \ref{thm:theta-equiv-real}). In the case of $\theta(W,\bar{\nu})$, the transposition in $\RU^+(\Z)$ can be readily realised by a surgery along one of the chosen embeddings $S^q \times D^{q+1} \rightarrow W$. For $[\theta_{W,F}]$, to realise an elementary equivalence of $\sim$, we first need to find an embedding $S^q \times D^{q+1} \rightarrow W$ to do surgery, and then the larger part of the proof is concerned with checking that this has the expected effect on the obstruction (see Lemma \ref{lem:theta-surg-real}).

\subsection{The definition of $\theta_{U,F}$}

Let $c$ be a CW-decomposition of $W$ and $U \subset \interior W$ be a neighbourhood of an embedding $\sk_q W \rightarrow \interior W$ isotopic to the inclusion. We will define a quasi-formation $\theta_{U,F} \in \gc^T_{2q+1}(B,\xi)$. We start with some basic observations.

\begin{prop} \label{prop:fund-hom-groups}
a) The spaces $W$, $M_0$, $M_1$, $\sk_q W$, $U$, $\partial U$ and $W \setminus \interior U$ are all simply-connected. 

b) If $i \leq q-1$, then $H_i(B) \cong H_i(W) \cong H_i(M_0) \cong H_i(M_1) \cong H_i(\sk_q W) \cong H_i(U) \cong H_i(\partial U) \cong H_i(W \setminus \interior U)$. The isomorphisms are induced by the natural inclusions between these spaces and (the restrictions of) $F$. 

c) If $i \leq q$, then $H_i(W, U) \cong H_i(W \setminus \interior U, \partial U) \cong 0$. 

d) If $i \leq q$, then $H_i(W, W \setminus \interior U) \cong H_i(U, \partial U) \cong 0$. 

e) If $i \leq q-1$, then $H_i(W, M_0) \cong H_i(W, M_1) \cong H_i(W \setminus \interior U, M_0) \cong H_i(W \setminus \interior U, M_1) \cong 0$. 
\end{prop}

\begin{proof}
Use the fact that the maps $F$, $f_0$ and $f_1$, as well as the inclusion $\sk_q W \simeq U \rightarrow W$, are $q$-connected, $B$ is simply-connected, and $q \geq 2$. Moreover, since $\sk_q W$ is a codimension-$(q+1)$ subcomplex in $U$ and $W$, the inclusions $\partial U \simeq U \setminus \sk_q W \rightarrow U$ and $W \setminus \interior U \simeq W \setminus \sk_q W \rightarrow W$ are $q$-connected too. (For more details see \cite[Proposition 4.4.1]{csn-thesis}.)
\end{proof}

Since $W$ is oriented, the codimension-$0$ submanifold $U$ and its boundary $\partial U$ are oriented too. 

\begin{defin} \label{def:thetaU-M}
Let $\UM = E_q(\partial U, F \big| _{\partial U})$ (see Definition \ref{def:qf}), that is
\[
\UM = (H_q(\partial U), \lambda, \mu)
\]
where $\lambda : H_q(\partial U) \times H_q(\partial U) \rightarrow \Z$ is the intersection form of $\partial U$ and $\mu = H_q(F \big| _{\partial U}) : H_q(\partial U) \rightarrow H_q(B)$. This is an extended quadratic form over $H_q(B)$.
\end{defin}

In the following we will define subgroups $L$ and $V$ in $H_q(\partial U)$, and show that $L$ is a T-lagrangian and $V$ is a free half-rank direct summand, so that $(\UM; L, V)$ is a quasi-formation over $H_q(B)$. This will be $\theta_{U,F}$. We will also verify that $\theta_{U,F}$ is in the class $\gc^T_{2q+1}(B,\xi) = \gc^T_{2q+1}(H_q(B),\hat{v}_q(\xi),\Tor H_{q-1}(B))$, ie.\ $\UM$ is geometric (with respect to $\hat{v}_q(\xi)$), full and $\Tor H_q(\partial U) \cong \Tor H_{q-1}(B)$ (see Definitions \ref{def:gc} and \ref{def:l-space}).

\begin{lem} \label{lem:dU-tor}
$\Tor H_q(\partial U) \cong \Tor H_{q-1}(B)$.
\end{lem}

\begin{proof}
Since $U \simeq \sk_q W$ and $\sk_q W$ is $q$-dimensional, $H_q(U)$ is free and $H_{q+1}(U) \cong 0$. Using the exact sequence of the pair $(U, \partial U)$, we get $\Tor H_q(\partial U) \cong \Tor H_{q+1}(U, \partial U)$. By the universal coefficient formula we have $\Tor H_{q+1}(U, \partial U) \cong \Tor H^{q+2}(U, \partial U)$. By Poincar\'e-duality and Proposition \ref{prop:fund-hom-groups} b) $H^{q+2}(U, \partial U) \cong H_{q-1}(U) \cong H_{q-1}(B)$. 
\end{proof}

\begin{defin} \label{def:thetaU-L}
We define the subgroup $L \leq H_q(\partial U)$ to be the image of the boundary homomorphism $H_{q+1}(U, \partial U) \rightarrow H_q(\partial U)$.
\end{defin}

\begin{prop} \label{prop:l-lagr}
$L$ is a T-lagrangian in $\UM$. 
\end{prop}

\begin{proof}
Again, we have that $H_q(U)$ is free, so $\Coker(H_{q+1}(U, \partial U) \rightarrow H_q(\partial U)) \cong \Image(H_q(\partial U) \rightarrow H_q(U)) \leq H_q(U)$ is free too. By Lemma \ref{lem:summand} $L = \Image(H_{q+1}(U, \partial U) \rightarrow H_q(\partial U))$ is a direct summand in $H_q(\partial U)$, and it contains $\Tor H_q(\partial U)$. By applying Lemma \ref{lem:boundary-lagr} we also get $\rk L = \frac{1}{2} \rk H_q(\partial U)$, $\lambda \big| _{L \times L} = 0$ and $\mu \big| _L = 0$. 
\end{proof}

\begin{lem} \label{lem:skel-hom}
$H_q(\sk_q W) \cong H_q(\partial U) / L$, and the quotient homomorphism $H_q(\partial U) \rightarrow H_q(\partial U) / L \cong H_q(\sk_q W)$ is induced by the inclusion $\partial U \rightarrow U \simeq \sk_q W$. 
\end{lem}

\begin{proof}
The homomorphism $H_q(\partial U) \rightarrow H_q(U)$ is surjective, because $H_q(U, \partial U) \cong 0$ (see Proposition \ref{prop:fund-hom-groups} d)). Hence $H_q(\sk_q W) \cong H_q(U) \cong \Image(H_q(\partial U) \rightarrow H_q(U)) \cong H_q(\partial U) / \Ker(H_q(\partial U) \rightarrow H_q(U)) \cong H_q(\partial U) / \Image(H_{q+1}(U, \partial U) \rightarrow H_q(\partial U))$. 
\end{proof}

Consider the long exact sequence of the triple $(W \setminus \interior U, \partial U \sqcup M_0, M_0)$: 
\begin{multline*}
\xymatrix{
H_{q+1}(W \setminus \interior U, M_0) \ar[r] & H_{q+1}(W \setminus \interior U, \partial U \sqcup M_0) \ar[r] & 
}
\\
\xymatrix{
 \ar[r] & H_q(\partial U \sqcup M_0, M_0) \ar[r] & H_q(W \setminus \interior U, M_0) \ar[r] & H_q(W \setminus \interior U, \partial U \sqcup M_0)
}
\end{multline*}

\begin{defin} \label{def:thetaU-V}
We define the subgroup $V \leq H_q(\partial U)$ to be the image of the boundary homomorphism $H_{q+1}(W \setminus \interior U, \partial U \sqcup M_0) \rightarrow H_q(\partial U \sqcup M_0, M_0) \cong H_q(\partial U)$. 
\end{defin}

Next we will show that the relevant part of the above sequence is in fact a short exact sequence.

\begin{lem} \label{lem:hom-triv1}
The homomorphism $H_q(W \setminus \interior U, M_0) \rightarrow H_q(W \setminus \interior U, \partial U \sqcup M_0)$ is trivial. 
\end{lem}

\begin{proof}
It follows from Proposition \ref{prop:fund-hom-groups} b) that the boundary map $H_q(W, M_0) \rightarrow H_{q-1}(M_0)$ is trivial.

By Proposition \ref{prop:fund-hom-groups} c) $H_i(W, U) \cong 0$ for $i \leq q$, so the boundary homomorphism $H_q(W, U \sqcup M_0) \rightarrow H_{q-1}(U \sqcup M_0, U)$ in the exact sequence of the triple $(W, U \sqcup M_0, U)$ is an isomorphism. By excision we get the isomorphism in the commutative diagram 
\[
\xymatrix{
H_q(W \setminus \interior U, M_0) \ar[r] \ar[d] & H_q(W \setminus \interior U, \partial U \sqcup M_0) \ar[d]^-{\cong} \\ 
H_q(W, M_0) \ar[r]^-{0} & H_{q-1}(M_0)
}
\]
showing that the map $H_q(W \setminus \interior U, M_0) \rightarrow H_q(W \setminus \interior U, \partial U \sqcup M_0)$ is trivial. 
\end{proof}

\begin{lem} \label{lem:hom-triv2}
The homomorphism $H_{q+1}(W \setminus \interior U, M_0) \rightarrow H_{q+1}(W \setminus \interior U, \partial U \sqcup M_0)$ is trivial. 
\end{lem}

\begin{proof}
By Poincar\'e-duality it is enough to prove that the homomorphism $H^q(W \setminus \interior U, \partial U \sqcup M_1) \rightarrow H^q(W \setminus \interior U, M_1)$ is trivial. 

Since $F$ and $f_1$ are $q$-connected, the inclusion $M_1 \rightarrow W$ induces an isomorphism on $H^i$ for $i \leq q-1$, so the boundary map $H^{q-1}(M_1) \rightarrow H^q(W, M_1)$ is trivial.

Since $U \simeq \sk_q W$, we have $H^i(W, U) \cong 0$ for $i \leq q$, therefore the boundary homomorphism $H^{q-1}(U \sqcup M_1, U) \rightarrow H^q(W, U \sqcup M_1)$ in the exact sequence of the triple $(W, U \sqcup M_1, U)$ is an isomorphism. By excision we get the isomorphism in the commutative diagram 
\[
\xymatrix{
H^{q-1}(M_1) \ar[r]^-{0} \ar[d]_-{\cong} & H^q(W, M_1) \ar[d] \\ 
H^q(W \setminus \interior U, \partial U \sqcup M_1) \ar[r] & H^q(W \setminus \interior U, M_1)
}
\]
showing that the map $H^q(W \setminus \interior U, \partial U \sqcup M_1) \rightarrow H^q(W \setminus \interior U, M_1)$ is trivial. 
\end{proof}

By combining Lemmas \ref{lem:hom-triv1} and \ref{lem:hom-triv2} we get: 

\begin{thm} \label{thm:V-ses}
There is a short exact sequence
\[
\xymatrix{
0 \ar[r] & H_{q+1}(W \setminus \interior U, \partial U \sqcup M_0) \ar[r] & H_q(\partial U) \ar[r] & H_q(W \setminus \interior U, M_0) \ar[r] & 0
} 
\qed
\]
\end{thm}

The following two lemmas will imply that $V$ is a free half-rank direct summand.

\begin{lem} \label{lem:V-free}
The restriction of the homomorphism $H_q(\partial U) \rightarrow H_q(W \setminus \interior U, M_0)$ is an isomorphism $\Tor H_q(\partial U) \cong \Tor H_q(W \setminus \interior U, M_0)$. 
\end{lem}

\begin{proof}
The universal coefficient formula, applied dually to compute the integral homology of a space $X$ (with finitely generated cohomology groups) from its cohomology, gives a natural isomorphism $\Ext(H^{i+1}(X),\Z) \stackrel{\cong}{\rightarrow} \Tor H_i(X)$ for every $i$. Therefore it is enough to prove that the homomorphism $H^{q+1}(W \setminus \interior U, M_0) \rightarrow H^{q+1}(\partial U)$ is an isomorphism. 

Consider the diagram 
\[
\xymatrix{
H^{q+1}(W \setminus \interior U, M_0) \ar[r] \ar[d] & H_q(W \setminus \interior U, \partial U \sqcup M_1) \ar[r] \ar[d] & H_q(W, U \sqcup M_1) \ar[d] \\
H^{q+1}(\partial U) \ar[r] & H_{q-1}(\partial U) \ar[r] & H_{q-1}(U)
}
\]
By Poincar\'e-duality the first square commutates and the horizontal arrows in it are isomorphisms. The second square also commutates, because it is induced by a map $(W \setminus \interior U, \partial U \sqcup M_1, M_1) \rightarrow (W, U \sqcup M_1, M_1)$ of triples, and its horizontal arrows are isomorphisms by excision and Proposition \ref{prop:fund-hom-groups} b). 

It is now enough to prove that the map $H_q(W, U \sqcup M_1) \rightarrow H_{q-1}(U)$ is an isomorphism. We use the exact sequence of the triple $(W, U \sqcup M_1, M_1)$. We have $H_{q-1}(W, M_1) \cong 0$ by Proposition \ref{prop:fund-hom-groups} e). The map $H_q(W, M_1) \rightarrow H_q(W, U \sqcup M_1)$ is trivial, because $H_q(W, U \sqcup M_1) \cong H_{q-1}(M_1)$ (similarly to the isomorphism $H_q(W, U \sqcup M_0) \cong H_{q-1}(M_0)$ in the proof of Lemma \ref{lem:hom-triv1}) and the composition $H_q(W, M_1) \rightarrow H_{q-1}(M_1)$ vanishes by Proposition \ref{prop:fund-hom-groups} b).

It is now enough to prove that the map $H_q(W, U \sqcup M_1) \rightarrow H_{q-1}(U)$ is an isomorphism. This holds, because in the exact sequence of the triple $(W, U \sqcup M_1, M_1)$ we have $H_{q-1}(W, M_1) \cong 0$ and the map $H_q(W, M_1) \rightarrow H_q(W, U \sqcup M_1) \cong H_{q-1}(M_1)$ is trivial by Proposition \ref{prop:fund-hom-groups} e) and b) respectively (where we used the isomorphism from the proof of Lemma \ref{lem:hom-triv1}). 
\end{proof}

\begin{lem} \label{lem:V-rank}
We have $\rk H_{q+1}(W \setminus \interior U, \partial U \sqcup M_0) = \rk H_q(W \setminus \interior U, M_0)$.
\end{lem}

\begin{proof}
It follows from Poincar\'e-duality and the universal coefficient formula that for every $i$ we have $\rk H_i(W \setminus \interior U, \partial U \sqcup M_0) = \rk H^{2q+1-i}(W \setminus \interior U, M_1) = \rk H_{2q+1-i}(W \setminus \interior U, M_1)$, in particular $\rk H_{q+1}(W \setminus \interior U, \partial U \sqcup M_0) = \rk H_q(W \setminus \interior U, M_1)$. 

Since $\chi(M_0) = \chi(M_1)$, we have $\chi(W \setminus \interior U, M_0) = \chi(W \setminus \interior U, M_1)$, that is, 
\[
\sum_{i=0}^{2q+1} (-1)^i \rk H_i(W \setminus \interior U, M_0) = \sum_{i=0}^{2q+1} (-1)^i \rk H_i(W \setminus \interior U, M_1)
\] 
Hence
\begin{multline*}
\rk H_q(W \setminus \interior U, M_0) - \rk H_q(W \setminus \interior U, M_1) = \\
= \sum_{i \neq q} (-1)^{i+q+1} (\rk H_i(W \setminus \interior U, M_0) - \rk H_i(W \setminus \interior U, M_1))
\end{multline*}

If $i \leq q-1$, then $H_i(W \setminus \interior U, M_0) \cong H_i(W \setminus \interior U, M_1) \cong 0$ by Proposition \ref{prop:fund-hom-groups} e). 

If $i \geq q+1$, then $\rk H_i(W \setminus \interior U, M_1) = \rk H_{2q+1-i}(W \setminus \interior U, \partial U \sqcup M_0) = \rk H_{2q-i}(M_0) = \rk H_{2q-i}(B)$ (the second equality follows from the proof of Lemma \ref{lem:hom-triv1} and the third from Proposition \ref{prop:fund-hom-groups} b)). Similarly $\rk H_i(W \setminus \interior U, M_0) = \rk H_{2q-i}(B)$, therefore $\rk H_i(W \setminus \interior U, M_0) - \rk H_i(W \setminus \interior U, M_1) = 0$. 

It follows that $\rk H_q(W \setminus \interior U, M_0) - \rk H_q(W \setminus \interior U, M_1) = 0$.
\end{proof}

\begin{defin} \label{def:thetaU}
Let $\theta_{U,F}$ denote the triple $(\UM; L, V)$ (see Definitions \ref{def:thetaU-M}, \ref{def:thetaU-L} and \ref{def:thetaU-V}).
\end{defin}

\begin{prop}
$\theta_{U,F}$ is a quasi-formation over $H_q(B)$ and $\theta_{U,F} \in \gc^T_{2q+1}(B,\xi)$. 
\end{prop}

\begin{proof}
$\UM = E_q(\partial U, F \big| _{\partial U})$ is an extended quadratic form over $H_q(B)$. It is metabolic, because the intersection form $\lambda$ is nonsingular, and $L$ is a T-lagrangian (Proposition \ref{prop:l-lagr}). It follows from Theorem \ref{thm:V-ses} and Lemmas \ref{lem:V-free} and \ref{lem:V-rank} that $V$ is a free half-rank direct summand. Therefore $\theta_{U,F}$ is a quasi-formation over $H_q(B)$. 

Since $F$ is a normal $(q{-}1)$-smoothing, $F^*(\xi) \cong \nu_W$ and hence $(F \big| _{\partial U})^*(\xi) \cong \nu_{\partial U}$. By Lemma \ref{lem:qf-geom} this implies that $\UM$ is geometric with respect to $\hat{v}_q(\xi)$. It is also full, because $F \big| _{\partial U}$ is the composition $\partial U \rightarrow U \rightarrow W \stackrel{F}{\rightarrow} B$ of three $q$-connected maps (using Proposition \ref{prop:fund-hom-groups} d), the homotopy equivalence $U \simeq \sk_q W$ and the assumption on $F$). By Lemma \ref{lem:dU-tor} $\Tor H_q(\partial U) \cong \Tor H_{q-1}(B)$. Therefore $\theta_{U,F}$ is in the class $\gc^T_{2q+1}(B,\xi) = \gc^T_{2q+1}(H_q(B),\hat{v}_q(\xi),\Tor H_{q-1}(B))$. 
\end{proof}

\subsection{The definition of $\theta_{c,F}$}

Now suppose that only a CW-decomposition $c$ of $W$ is chosen, without a fixed regular neighbourhood $U$ of $\sk_q W$.

\begin{defin}
Let $\theta_{c,F} \in \gs^T_{2q+1}(B,\xi)$ denote the isomorphism class of $\theta_{U,F}$ for any regular neighbourhood $U$ of any embedding $\sk_q W \rightarrow \interior W$ isotopic to the inclusion $\sk_q W \rightarrow W$.
\end{defin}

The $c_1=c_2=c$ special case of the following lemma shows that $\theta_{c,F}$ is well-defined.

\begin{lem} \label{lem:cwhe-isom}
Let $c_1$ and $c_2$ be two CW-decompositions of $W$, and let $K_i \subset W$ denote the $q$-skeleton of $W$ with respect to $c_i$, for $i=1,2$. Let $U_i$ be a regular neighbourhood of an embedding $K_i \rightarrow \interior W$ isotopic to the inclusion $K_i \rightarrow W$. Suppose that there is a homotopy equivalence $j : K_1 \rightarrow K_2$ such that the diagram 
\[
\xymatrix{
K_1 \ar[rr]^-{j} \ar[rd] & & K_2 \ar[ld] \\
 & W & 
}
\]
commutes up to homotopy (where the maps $K_i \rightarrow W$ are the inclusions). 

Then $\theta_{U_1,F} \cong \theta_{U_2,F}$, via an isomorphism $I : H_q(\partial U_1) \rightarrow H_q(\partial U_2)$ that fits into the commutative diagram
\[
\xymatrix{
H_q(K_1) \ar[r]^-{\cong} \ar[d]_-{H_q(j)} & H_q(U_1) \ar[d] & H_q(\partial U_1) \ar[l] \ar[d]^-{I} \\
H_q(K_2) \ar[r]^-{\cong} & H_q(U_2) & H_q(\partial U_2) \ar[l]
}
\]
where the isomorphisms $H_q(K_i) \rightarrow H_q(U_i)$ are induced by the embeddings of $K_i$, the isomorphism $H_q(U_1) \rightarrow H_q(U_2)$ is the unique map that makes the left square commute, and the homomorphisms $H_q(\partial U_i) \rightarrow H_q(U_i)$ are induced by the inclusions. 
\end{lem}

\begin{proof}
By Theorem \ref{thm:wall-emb-uniq} a) (applied to the composition $K_1 \stackrel{j}{\rightarrow} K_2 \rightarrow U_2$), there is a compact codimension-$0$ submanifold $U_2' \subset U_2$ with simply-connected boundary and a homotopy equivalence $K_1 \rightarrow U_2'$ such that the diagram 
\[
\xymatrix{
K_1 \ar[r] \ar[d]_-{j} & U_2' \ar[d]  \\
K_2 \ar[r] & U_2 
}
\]
commutes up to homotopy. Since $j$ is a homotopy equivalence, the inclusion $U_2' \rightarrow U_2$ is a homotopy equivalence too. The compositions $K_1 \rightarrow U_1 \rightarrow W$ and $K_1 \rightarrow U_2' \rightarrow U_2 \rightarrow W$ are homotopic, so Theorem \ref{thm:wall-emb-uniq} b) can be applied to the homotopy equivalences $K_1 \rightarrow U_1$ and $K_1 \rightarrow U_2'$. We get that there is a diffeomorphism $g : U_1 \rightarrow U_2'$ such that the inclusion $U_1 \rightarrow W$ and the composition $U_1 \stackrel{g}{\rightarrow} U_2' \rightarrow W$ are isotopic, and the diagram 
\[
\xymatrix{
K_1 \ar[r] \ar@{=}[d] & U_1 \ar[d]^-{g}  \\
K_1 \ar[r] & U_2' 
}
\]
commutes up to homotopy. In particular, $U_2'$ is a regular neighbourhood of an embedding $K_1 \rightarrow \interior W$ isotopic to the inclusion $K_1 \rightarrow W$, so $\theta_{U_2',F}$ is defined.

We will show that $\theta_{U_1,F} \cong \theta_{U_2',F}$ and $\theta_{U_2',F} \cong \theta_{U_2,F}$. We introduce the notation $(\UM; L, V) = \theta_{U_1,F}$, $(\UMp; L', V') = \theta_{U_2',F}$ and $(\UMpp; L'', V'') = \theta_{U_2,F}$. 

First we consider $\theta_{U_1,F}$ and $\theta_{U_2',F}$. It follows from the isotopy extension theorem that there is a diffeomorphism $G : W \rightarrow W$ (isotopic to the identity) such that $G \big| _{U_1} = g$ (hence $G(U_1)=U_2'$). Then $H_q(G \big| _{\partial U_1}) = H_q(g \big| _{\partial U_1}) : H_q(\partial U_1) \rightarrow H_q(\partial U_2')$ is an isomorphism between $\UM$ and $\UMp$. Moreover, $G$ induces commutative diagrams
\[
\xymatrix{
H_{q+1}(U_1, \partial U_1) \ar[r] \ar[d]_{\cong} & H_q(\partial U_1) \ar[d]^{\cong} \\
H_{q+1}(U_2', \partial U_2') \ar[r] & H_q(\partial U_2') 
}
\quad \quad
\xymatrix{
H_{q+1}(W \setminus \interior U_1, \partial U_1 \sqcup M_0) \ar[r] \ar[d]_{\cong} & H_q(\partial U_1) \ar[d]^{\cong} \\
H_{q+1}(W \setminus \interior U_2', \partial U_2' \sqcup M_0) \ar[r] & H_q(\partial U_2') 
}
\]
Hence $H_q(G \big| _{\partial U_1})$ sends $L$ to $L'$ and $V$ to $V'$. Therefore $(\UM; L, V) \cong (\UMp; L', V')$, via the isomorphism $H_q(g \big| _{\partial U_1})$. 

Next we consider $\theta_{U_2',F}$ and $\theta_{U_2,F}$. Let $C = U_2 \setminus \interior U_2'$. It follows from the assumptions and the Van Kampen theorem that $\partial U_2'$, $\partial U_2$ and $C$ are all simply-connected. Since the inclusion $U_2' \rightarrow U_2$ is a homotopy equivalence, $H_*(C, \partial U_2') \cong H_*(U_2, U_2') \cong 0$, so the inclusion $\partial U_2' \rightarrow C$ is also a homotopy equivalence. It follows from Poincar\'e-duality that the same is true for the inclusion $\partial U_2 \rightarrow C$, so $C$ is an h-cobordism. This (and excision) implies that the vertical arrows are isomorphisms in the following diagrams, induced by the inclusions:
\[
\xymatrix{
H_{q+1}(U_2', \partial U_2') \ar[r] \ar[d]_{\cong} & H_q(\partial U_2') \ar[d]^{\cong} \\
H_{q+1}(U_2, C) \ar[r] & H_q(C) \\
H_{q+1}(U_2, \partial U_2) \ar[r] \ar[u]^{\cong} & H_q(\partial U_2) \ar[u]_{\cong} 
}
\quad \quad
\xymatrix{
H_{q+1}(W \setminus \interior U_2', \partial U_2' \sqcup M_0) \ar[r] \ar[d]_{\cong} & H_q(\partial U_2') \ar[d]^{\cong} \\
H_{q+1}(W \setminus \interior U_2', C \sqcup M_0) \ar[r] & H_q(C) \\
H_{q+1}(W \setminus \interior U_2, \partial U_2 \sqcup M_0) \ar[r] \ar[u]^{\cong} & H_q(\partial U_2) \ar[u]_{\cong} 
}
\]
Thus the composition of the isomorphism $H_q(\partial U_2') \rightarrow H_q(C)$ and the inverse of $H_q(\partial U_2) \rightarrow H_q(C)$ is an isomorphism $H_q(\partial U_2') \rightarrow H_q(\partial U_2)$ that sends $L'$ to $L''$ and $V'$ to $V''$. Moreover, it is an isomorphism between the extended quadratic forms $\UMp$ and $\UMpp$, because the inclusions induce isomorphisms of cohomology rings $H^*(\partial U_2') \cong H^*(C) \cong H^*(\partial U_2)$, and they commute with the homomorphisms $H_q(F \big| _{\partial U_2'})$, $H_q(F \big| _C)$ and $H_q(F \big| _{\partial U_2})$. Therefore $(\UMp; L', V') \cong (\UMpp; L'', V'')$.

We take $I$ to be the composition of the above isomorphisms $H_q(\partial U_1) \rightarrow H_q(\partial U_2') \rightarrow H_q(\partial U_2)$. To complete the proof, we verify the commutativity of the following diagram:
\[
\xymatrix{
H_q(K_1) \ar[r]^-{\cong} \ar@{=}[d] & H_q(U_1) \ar[d]^-{H_q(g)} & H_q(\partial U_1) \ar[l] \ar[d]^-{H_q(g | _{\partial U_1})} \\
H_q(K_1) \ar[r]^-{\cong} \ar[d]_-{H_q(j)} & H_q(U_2') \ar[d] & H_q(\partial U_2') \ar[l] \ar[d] \\
H_q(K_2) \ar[r]^-{\cong} & H_q(U_2) & H_q(\partial U_2) \ar[l]
}
\]
The squares on the left are obtained by applying $H_q$ to the earlier homotopy commutative diagrams. The top right square is induced by $g$ and the inclusions. Finally, the bottom right square commutes, because both compositions are equal to the composition $H_q(\partial U_2') \rightarrow H_q(C) \rightarrow H_q(U_2)$ induced by the inclusions.
\end{proof}

\begin{rem}
The proof can be simplified by taking $U_2' = U_2$ if either $q>2$ or $c_1=c_2$. In the former case $C \approx \partial U_2' \times I$ by the h-cobordism theorem, while in the latter case Theorem \ref{thm:wall-emb-uniq} b) can be applied to the homotopy equivalences $K_1 \rightarrow U_1$ and $K_1 = K_2 \rightarrow U_2$.
\end{rem}

We used the following: 

\begin{thm}[Wall \cite{wall66-iv} and Hudson \cite{hudson70}] \label{thm:wall-emb-uniq}
a) Suppose that $K$ is a simply-connected finite CW-complex of dimension $k$, and $M$ is a simply-connected $d$-manifold for some $d \geq \max(2k+1,k+3)$. Let $f : K \rightarrow M$ be a continuous map. Then there is a compact codimension-$0$ submanifold $N \subset M$ with simply-connected boundary and a homotopy equivalence $h : K \rightarrow N$ such that the composition $K \stackrel{h}{\rightarrow} N \rightarrow M$ is homotopic to $f$. Moreover, $N$ has a handlebody decomposition that corresponds to the cell decomposition of $K$ in the sense of \cite{wall66-iv} (in particular, every handle has index at most $k$).

b) If $h : K \rightarrow N$ and $h' : K \rightarrow N'$ are two homotopy equivalences as above, then there is a diffeomorphism $g : N \rightarrow N'$ such that $g \circ h \simeq h' : K \rightarrow N'$ and $N \stackrel{g}{\rightarrow} N' \rightarrow M$ is isotopic to the embedding $N \rightarrow M$.
\end{thm}

\begin{proof}
Part a) is a special case of Wall's embedding theorem \cite{wall66-iv}. Part b) follows from the uniqueness part of the embedding theorem combined with results of Hudson \cite[Theorem 2.1 and Addendum 2.1.2]{hudson70}.
\end{proof}

\subsection{The definition of $\theta_{W,F}$} \label{ss:def:tw}

Now we will consider $W$ without a fixed CW-decomposition $c$.

\begin{defin}
Let $\theta_{W,F} \in \s^T_{2q+1}(B,\xi)$ be the stable isomorphism class of $\theta_{c,F}$ for any CW-decomposition $c$ of $W$. 
\end{defin}

\begin{thm} \label{thm:tw-welldef}
$\theta_{W,F}$ is well-defined. 
\end{thm}

\begin{proof}
Let $c_1$ and $c_2$ be arbitrary CW-decompositions of $W$. Then they have stabilisations (see Definition \ref{def:cw-stab}) $c_1'$ and $c_2'$ that satisfy the conditions of Lemma \ref{lem:cwstab-he}. By Lemmas \ref{lem:cwstab-isom} and \ref{lem:cwhe-isom} we have $\theta_{c_1,F} \si \theta_{c_1',F} \cong \theta_{c_2',F} \si \theta_{c_2,F}$. Therefore the stable isomorphism class of $\theta_{c,F}$ is independent of the choice of $c$. 
\end{proof}

\begin{defin} \label{def:cw-stab}
We define \emph{stabilisation}, an operation on CW-decompositions of $W$, as follows: Let $c$ be a CW-decomposition of $W$. Let $x$ be a $0$-cell, since $W$ is a $(2q+1)$-manifold, $x$ is in the boundary of some $(2q+1)$-cell $e^{2q+1}$. Consider $D^{q+1}$ with its decomposition into $3$ cells (of dimensions $0$, $q$ and $q+1$), and an embedding of $D^{q+1}$ into $x \cup \interior e^{2q+1}$ such that the $0$-cell goes to $x$. The new CW-decomposition $c'$ is constructed from $c$ by adding the $q$-cell and $(q+1)$-cell of the embedded $D^{q+1}$ and modifying the gluing map of the $(2q+1)$-cell $e^{2q+1}$.

Note that the stabilisation operation is not unique. 

We say that $c'$ is a ($k$-fold) stabilisation of $c$, if it can be obtained from $c$ by a sequence of ($k$) stabilisation steps.
\end{defin}

\begin{figure}[H]
\centering
\includegraphics[scale=0.5]{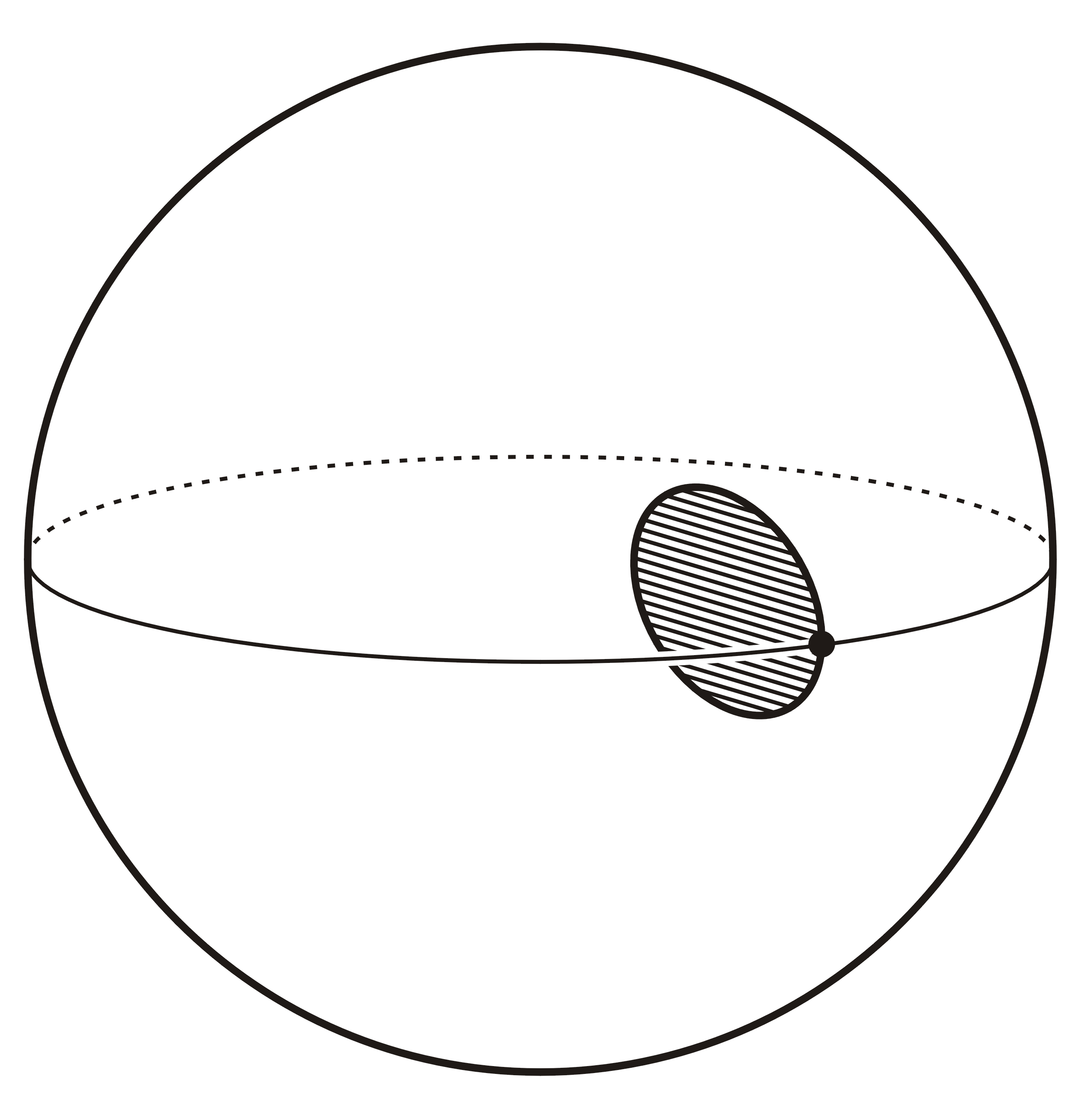}
\caption{A stabilisation of the standard CW-decomposition of $D^3$ (consisting of one $0$-cell, one $2$-cell and one $3$-cell) adds a $1$-cell and a $2$-cell, and modifies the $3$-cell.}
\end{figure}

The following three lemmas are needed for the proof of Theorem \ref{thm:tw-welldef}.

\begin{lem} \label{lem:cwstab-isom}
Let $c$ and $c'$ be CW-decompositions of $W$. If $c'$ is a $k$-fold stabilisation of $c$, then $\theta_{c',F} \cong \theta_{c,F} \oplus \HH_{2k}$. In particular, $\theta_{c',F} \si \theta_{c,F}$ for any stabilisation $c'$ of $c$.
\end{lem}

\begin{proof}
It is enough to prove this when $c'$ is obtained from $c$ by a single stabilisation step. 

Let $K$ and $K'$ denote the $q$-skeletons of $W$ with respect to $c$ and $c'$, then $K' = K \vee S^q$. Let $U$ and $U'$ be regular neighbourhoods of $K$ and $K'$, we can choose these such that $U' \approx U \natural (S^q \times D^{q+1})$. Let $(\UM; L, V) = \theta_{U,F}$ and $(\UMp; L', V') = \theta_{U',F}$. 

Since $\partial U' \approx U \# (S^q \times S^q)$, we have $H_q(\partial U') \cong H_q(\partial U) \oplus \Z^2$ and $\UMp \cong \UM \oplus \UH_2$, we will use these canonical isomorphisms to identify these groups and extended quadratic forms. The inclusion $H_q(\partial U) \rightarrow H_q(\partial U')$ under this identification is in fact the composition $H_q(\partial U) \cong H_q(\partial U \setminus \interior D^{2q}) \rightarrow H_q(\partial U')$, where the homomorphism $H_q(\partial U \setminus \interior D^{2q}) \rightarrow H_q(\partial U')$ is induced by the inclusion of spaces. Thus the inclusion $U \rightarrow U'$ induces commutative diagrams
\[
\xymatrix{
H_{q+1}(U, \partial U) \ar[r] \ar[d] & H_q(\partial U) \ar[d] \\
H_{q+1}(U', \partial U') \ar[r] & H_q(\partial U') 
}
\]
and
\[
\xymatrix{
H_{q+1}(W, U \sqcup M_0) \cong H_{q+1}(W \setminus \interior U, \partial U \sqcup M_0) \ar[r] \ar[d] & H_q(\partial U) \ar[d] \\
H_{q+1}(W, U' \sqcup M_0) \cong H_{q+1}(W \setminus \interior U', \partial U' \sqcup M_0) \ar[r] & H_q(\partial U') 
}
\]
hence $L \leq L'$ and $V \leq V'$. 

If we fix some basepoint $x \in S^q$, the submanifold $\{ x \} \times D^{q+1} \subset S^q \times D^{q+1} \subseteq U'$ represents an element of $H_{q+1}(U', \partial U')$, and its boundary is represented by the submanifold $\{ x \} \times S^q \subset S^q \times S^q \subseteq \partial U'$. This shows that $(0,(0,1)) \in L' \leq H_q(\partial U') \cong H_q(\partial U) \oplus \Z^2$, therefore $L \oplus (\{ 0 \} \times \Z) \leq L'$. Since $L'$ and $L \oplus (\{ 0 \} \times \Z)$ are both T-lagrangians in $\UMp$, we get that $L \oplus (\{ 0 \} \times \Z)$ is a direct summand in $L'$, and they have the same torsion subgroup and the same rank, and this implies that $L' = L \oplus (\{ 0 \} \times \Z)$. 

Similarly, the embedded $D^{q+1}$ that we used to construct the stabilisation $c'$ from $c$ represents an element of $H_{q+1}(W, U' \sqcup M_0) \cong H_{q+1}(W \setminus \interior U', \partial U' \sqcup M_0)$, and its boundary is represented by $S^q \times \{ x \} \subset S^q \times S^q \subseteq \partial U'$, therefore $V \oplus (\Z \times \{ 0 \}) \leq V'$.  Since $V'$ and $V \oplus (\Z \times \{ 0 \})$ are both free half-rank direct summands in $\UMp$, this implies that $V' = V \oplus (\Z \times \{ 0 \})$. 

Therefore $(\UMp; L', V') \cong (\UM \oplus \UH_2; L \oplus (\{ 0 \} \times \Z), V \oplus (\Z \times \{ 0 \})) = (\UM; L, V) \oplus \HH_2$. Since $\theta_{c,F}$ and $\theta_{c',F}$ are the isomorphism classes of $\theta_{U,F}$ and $\theta_{U',F}$ respectively, this means that $\theta_{c',F} \cong \theta_{c,F} \oplus \HH_2$.
\end{proof}

\begin{lem} \label{lem:cwstab-he}
Let $c_1$ and $c_2$ be any CW-decompositions of $W$. Then there is a stabilisation $c_i'$ of $c_i$ ($i=1,2$) such that if $K_i' \subset W$ denotes the $q$-skeleton of $W$ with respect to $c_i'$, then there is a homotopy equivalence $j : K_1' \rightarrow K_2'$ such that the diagram 
\[
\xymatrix{
K_1' \ar[rr]^-{j} \ar[rd] & & K_2' \ar[ld] \\
 & W & 
}
\]
commutes up to homotopy.
\end{lem}

\begin{proof}
Let $K_i$ denote the $q$-skeleton of $W$ with respect to $c_i$ and let $g_i : K_i \rightarrow W$ denote the inclusion. Then the homomorphisms $H_q(g_i) : H_q(K_i) \rightarrow H_q(W)$ are surjective, so by Lemma \ref{lem:hom-stab-isom} a) there are free abelian groups $A_i$ and an isomorphism $J : H_q(K_1) \oplus A_1 \rightarrow H_q(K_2) \oplus A_2$ such that $(H_q(g_1)+0) = (H_q(g_2)+0) \circ J : H_q(K_1) \oplus A_1 \rightarrow H_q(W)$. 

Let $k_i = \rk A_i$ and let $c_i'$ be a $k_i$-fold stabilisation of $c_i$. Then $K_i' \approx K_i \vee \bigvee^{k_i} S^q$, so $H_q(K_i') \cong H_q(K_i) \oplus A_i$, and $J$ can be regarded as an isomorphism $J : H_q(K_1') \rightarrow H_q(K_2')$ such that $H_q(g_1') = H_q(g_2') \circ J : H_q(K_1') \rightarrow H_q(W)$, where $g_i' : K_i' \rightarrow W$ is the inclusion. 

By Lemma \ref{lem:hom-real-map} there is a map $j : K_1' \rightarrow K_2'$ such that $g_1' \simeq g_2' \circ j : K_1' \rightarrow W$ and $H_q(j) = J$. Since $H_h(g_1') = H_h(g_2') \circ H_h(j)$ and $H_h(g_1')$ and $H_h(g_2')$ are isomorphisms if $h \leq q-1$, the same is true for $H_h(j)$. As $H_q(j) = J$ is also an isomorphism, and $K_1'$ and $K_2'$ are simply-connected and $q$-dimensional, $j$ is a homotopy equivalence. 
\end{proof}

\begin{lem} \label{lem:hom-real-map}
Let $X$ be a simply-connected topological space and $R$ and $S$ be CW-complexes with maps $f : R \rightarrow X$ and $g : S \rightarrow X$. Assume that $R$ is $q$-dimensional, $g$ is $q$-connected and $J : H_q(R) \rightarrow H_q(S)$ is a homomorphism such that $H_q(g) \circ J = H_q(f)$. Then there is a map $j : R \rightarrow S$ such that $g \circ j \simeq f$ and $H_q(j) = J$.
\end{lem}

\begin{proof}
The map $f : R \rightarrow X$ has a lift $\tilde{f} : R \rightarrow S$ such that $g \circ \tilde{f} \simeq f$, because the obstructions to the existence of such a lift are in the groups $H^i(R;\pi_{i-1}(\hofib(g)))$, which are trivial, because $g$ is $q$-connected and $R$ is $q$-dimensional. Let $J_0 = J - H_q(\tilde{f}) : H_q(R) \rightarrow H_q(S)$. 

The pair $(X, S)$ is $q$-connected (where $S$ is regarded as a subspace of $X$ via $g$), so $\pi_{q+1}(X, S) \cong H_{q+1}(X, S)$. The Hurewicz-homomorphisms induce the following commutative diagram: 
\[
\xymatrix{
 & H_q(R) \ar[d]^-{J_0} \ar@{-->}[dl] & \\
H_{q+1}(X, S) \ar[r] & H_q(S) \ar[r]^-{H_q(g)} & H_q(X) \\
\pi_{q+1}(X, S) \ar[u]^-{\cong} \ar[r] & \pi_q(S) \ar[u]_-{h} \ar[r]^-{\pi_q(g)} & \pi_q(X) \ar[u]
}
\]
We have $H_q(g) \circ J_0 = H_q(g) \circ (J - H_q(\tilde{f})) = H_q(g) \circ J - H_q(g) \circ H_q(\tilde{f}) = H_q(f) - H_q(f) = 0$ and $H_q(R)$ is free, so $J_0$ can be lifted to a homomorphism $H_q(R) \rightarrow H_{q+1}(X, S)$. By composing this homomorphism with the inverse of the isomorphism $\pi_{q+1}(X, S) \rightarrow H_{q+1}(X, S)$ and the boundary map, we get a homomorphism $\hat{J}_0 : H_q(R) \rightarrow \pi_q(S)$ such that $h \circ \hat{J}_0 = J_0$. Moreover, it follows from the construction of $\hat{J}_0$ that $\pi_q(g) \circ \hat{J}_0 = 0$, that is $\Image(\hat{J}_0) \leq \Ker(\pi_q(g))$.

Let $(C_*(R), \partial_*)$ be the cellular chain complex of $R$. Since $\Image(\partial_q) \cong C_q(R) / \Ker(\partial_q)$ is a subgroup of $C_{q-1}(R)$, it is free, hence $H_q(R) = \Ker(\partial_q)$ is a direct summand in $C_q(R)$. Therefore $\hat{J}_0 : H_q(R) \rightarrow \Ker(\pi_q(g))$ can be extended (arbitrarily) to a homomorphism $\bar{J}_0 : C_q(R) \rightarrow \Ker(\pi_q(g)) \leq \pi_q(S)$. 

Given a $q$-cell $e^q$ of $R$ and an element $[a] \in \pi_q(S)$, we can modify the map $\tilde{f}$ by adding $[a]$ to $\tilde{f} \big| _{e^q}$, ie.\ by replacing $\tilde{f} \big| _{e^q}$ with the composition $e^q \rightarrow e^q \vee S^q \stackrel{\tilde{f} \vee a}{\longrightarrow} S$. This operation changes the homomorphism $H_q(\tilde{f}) : H_q(R, \sk_{q-1} R) \rightarrow H_q(S, \sk_{q-1} S)$ by adding the image of $h([a]) \in H_q(S)$ in $H_q(S, \sk_{q-1} S)$ to the image of the generator $[e^q] \in H_q(R, \sk_{q-1} R) \cong C_q(R)$. Moreover, if $[a] \in \Ker(\pi_q(g))$, then the homotopy class of $g \circ \tilde{f}$ does not change, because $g \circ a$ is nullhomotopic, so the composition $e^q \rightarrow e^q \vee S^q \stackrel{\tilde{f} \vee a}{\longrightarrow} S \stackrel{g}{\rightarrow} X$ is homotopic to $g \circ \tilde{f} \big| _{e^q}$ (rel boundary). 

Let $j : R \rightarrow S$ be the map obtained from $\tilde{f}$ by adding $\bar{J}_0([e^q_i]) \in \pi_q(S)$ to $\tilde{f} \big| _{e^q_i}$ for every $q$-cell $e^q_i$ of $R$. This changes $H_q(\tilde{f}) : H_q(R, \sk_{q-1} R) \cong C_q(R) \rightarrow H_q(S, \sk_{q-1} S)$ by adding the composition $C_q(R) \stackrel{\bar{J}_0}{\rightarrow} \pi_q(S) \stackrel{h}{\rightarrow} H_q(S) \rightarrow H_q(S, \sk_{q-1} S)$. Therefore $H_q(j) = H_q(\tilde{f}) + h \circ \bar{J}_0 \big| _{H_q(R)} : H_q(R) \rightarrow H_q(S)$. Hence 
\[
H_q(j) = H_q(\tilde{f}) + h \circ \bar{J}_0 \big| _{H_q(R)} = H_q(\tilde{f}) + h \circ \hat{J}_0 = H_q(\tilde{f}) + J_0 = H_q(\tilde{f}) + J - H_q(\tilde{f}) = J \text{\,.}
\]
Moreover, since $\Image(\bar{J}_0) \leq \Ker(\pi_q(g))$, we have $g \circ j \simeq g \circ \tilde{f} \simeq f$, therefore $j$ satisfies both conditions. 
\end{proof}

We end this section with a consequence of the above results, which we will use later. Namely, any sufficiently stabilised isomorphism class in the stable isomorphism class $\theta_{W,F}$ can be realised as $\theta_{c,F}$ for some CW-decomposition $c$ of $W$:

\begin{prop} \label{prop:theta-cw-real}
If $\theta_{W,F} \si x$ for some $x \in \gs^T_{2q+1}(B,\xi)$, then there is a CW-decomposition $c$ of $W$ and an integer $k \geq 0$ such that $\theta_{c,F} \cong x \oplus \HH_{2k}$.
\end{prop}

\begin{proof}
Let $c_0$ be an arbitrary CW-decomposition of $W$. Then $\theta_{W,F}$ is the stable isomorphism class of $\theta_{c_0,F}$, so $\theta_{W,F} \si x$ means that there are integers $k,l \geq 0$ such that $x \oplus \HH_{2k} \cong \theta_{c_0,F} \oplus \HH_{2l}$. Let $c$ be an $l$-fold stabilisation of $c_0$. Then by Lemma \ref{lem:cwstab-isom} $\theta_{c,F} \cong \theta_{c_0,F} \oplus \HH_{2l} \cong x \oplus \HH_{2k}$. 
\end{proof}

\subsection{The invariance of $[\theta_{W,F}]$}

Now we consider the equivalence class $[\theta_{W,F}] \in \ell^T_{2q+1}(B,\xi)$ of $\theta_{W,F}$ and show that it is a normal bordism invariant.

\begin{thm} \label{thm:theta-bordism-inv}
Suppose that $F' : W' \rightarrow B$ is another $q$-connected normal bordism between $f_0$ and $f_1$ over $(B,\xi)$. If $F'$ is normally bordant to $F$ (rel boundary), then $[\theta_{W,F}] = [\theta_{W',F'}] \in \ell^T_{2q+1}(B,\xi)$. 
\end{thm}

\begin{proof}
Let $G : X \rightarrow B$ be a normal bordism between $F$ and $F'$. By applying surgery below the middle dimension (see eg.\ \cite[Proposition 4]{kreck99}), we can assume that $G$ is $(q+1)$-connected. Since the compositions $W \rightarrow X \rightarrow B$ and $W' \rightarrow X \rightarrow B$ are $q$-connected, this implies that the inclusions $W \rightarrow X$ and $W' \rightarrow X$ are $q$-connected too. Then there is a Morse-function on the triad $(X, W, W')$ having only critical points of index $q+1$ (see \cite[Theorem 8.1 and the proof of Theorem 7.8]{milnor65}). This means that $W'$ can be obtained from $W$ by a sequence of $q$-surgeries over $(B,\xi)$ (with trace $X$), so it is enough to prove the statement when $W'$ is obtained by a single $q$-surgery.

Fix an embedding $S^q \times D^{q+1} \rightarrow W$ where the surgery is done, and let $W_0 = W \setminus \interior(S^q \times D^{q+1})$. Then $W' = W_0 \cup_{S^q \times S^q} D^{q+1} \times S^q$. Take the standard CW-decomposition of $S^q \times S^q$ into $4$ cells (coming from the product structure) and extend it to a CW-decomposition $c_0$ of $W_0$. Define the equivalence relation $\underset{W}{\sim}$ on $S^q \times S^q$ by $(x,y_1) \underset{W}{\sim} (x,y_2)$ for any $x,y_1,y_2 \in S^q$, then ${W_0} / {\underset{W}{\sim}}$ is diffeomorphic to $W$ (because a neighbourhood of $S^q \times D^{q+1}$ in $W$ can be identified with $S^q \times \R^{q+1}$, and $(\R^{q+1} \setminus \interior D^{q+1}) / S^q \approx \R^{q+1}$). Moreover, $c_0$ induces a CW-decomposition $c$ on ${W_0} / {\underset{W}{\sim}} \approx W$, namely we have the standard decomposition of ${S^q \times S^q} / {\underset{W}{\sim}} \approx S^q$ into $2$ cells, and all other cells remain the same, with adjusted gluing maps. If we fix these CW-decompositions, then $\sk_q W$ is of the form $K_0 \vee S^q$, where $K_0$ is the $q$-skeleton of $W_0$ without the two $q$-cells of $S^q \times S^q$. Then a regular neighbourhood $U$ of $\sk_q W$ is of the form $U_0 \natural (S^q \times D^{q+1})$, where $U_0$ is a regular neighbourhood of $K_0$ and we can assume that the embedding $S^q \times D^{q+1} \rightarrow U \rightarrow W$ coincides with the embedding previously fixed for surgery (and hence $F \big| _{S^q \times D^{q+1}}$ is nullhomotopic). This means that $\partial U \approx \partial U_0 \# (S^q \times S^q)$ and $F \big| _{S^q \times S^q}$ is nullhomotopic. Let $(\UM; L, V) = \theta_{U,F}$. By Lemmas \ref{lem:qf-hyp} and \ref{lem:qf-sum} we have $\UM \cong \UM_0 \oplus \UH_2$, where $\UM_0 = (H_q(\partial U_0), \lambda_0, \mu_0)$, $\lambda_0$ is the intersection form of $\partial U_0$ and $\mu_0 = H_q(F \big| _{\partial U_0})$. Moreover $L \cong L_0 \oplus (\{ 0 \} \times \Z)$, where $L_0 = \Image(H_{q+1}(U_0, \partial U_0) \rightarrow H_q(\partial U_0))$. 

Similarly we define the equivalence relation $\underset{W'}{\sim}$ on $S^q \times S^q$ by $(x_1,y) \underset{W'}{\sim} (x_2,y)$ for any $x_1,x_2,y \in S^q$, then ${W_0} / {\underset{W'}{\sim}} \approx W'$. We get a CW-decomposition $c'$ of $W'$ such that $\sk_q W' \approx K_0 \vee S^q$, and $\sk_q W'$ has a regular neighbourhood $U' \approx U_0 \natural (D^{q+1} \times S^q)$, where $U_0$ is the same regular neighbourhood of $K_0$ that we used earlier. We can assume that the image of the embedding $D^{q+1} \times S^q \rightarrow U' \rightarrow W'$ is $\overline{W' \setminus W_0}$, and the restriction of this embedding to $S^q \times S^q$ coincides with the previously fixed embedding $S^q \times S^q \rightarrow W_0$. Let $(\UMp; L', V') = \theta_{U',F'}$. We have $\partial U' = \partial U \approx \partial U_0 \# (S^q \times S^q)$, therefore $\UMp = \UM \cong \UM_0 \oplus \UH_2$ and $L' \cong L_0 \oplus (\Z \times \{ 0 \})$. Moreover, $W \setminus \interior U = W_0 \setminus \interior U_0 = W' \setminus \interior U'$, so $V = \Image(H_{q+1}(W \setminus \interior U, \partial U \sqcup M_0) \rightarrow H_q(\partial U)) = \Image(H_{q+1}(W' \setminus \interior U', \partial U' \sqcup M_0) \rightarrow H_q(\partial U')) = V'$. By the definition of the equivalence relation $\sim$ (see Definition \ref{def:ell-t}), we have $\theta_{U,F} \cong (\UM_0 \oplus \UH_2; L_0 \oplus (\{ 0 \} \times \Z), V) \sim (\UM_0 \oplus \UH_2; L_0 \oplus (\Z \times \{ 0 \}), V) \cong \theta_{U',F'}$. Therefore $\theta_{W,F} \sim \theta_{W',F'}$.
\end{proof}

\subsection{Realising stable isomorphism classes} \label{ss:real-surg}

Next we show that every stable isomorphism class in the equivalence class $[\theta_{W,F}]$ can be realised by a normal bordism. This will be the key input to showing that $[\theta_{W,F}]$ is a complete obstruction, see Proposition \ref{prop:bord-hcob-elem}.

\begin{thm} \label{thm:theta-equiv-real}
If $\theta_{W,F} \sim x$ for some $x \in \s^T_{2q+1}(B,\xi)$, then there is a $q$-connected normal bordism $F' : W' \rightarrow B$ that is normally bordant to $F$ (rel boundary) such that $\theta_{W',F'} \si x$.
\end{thm}

\begin{proof}
If $\theta_{W,F} \si (\UM \oplus \UH_2; L \oplus ( \{0 \} \times \Z), V)$ and $x \si (\UM \oplus \UH_2; L \oplus (\Z \times \{ 0 \}), V)$ for some $\UM$, $L$ and $V$, then we can apply Lemma \ref{lem:theta-surg-real} below to get a suitable $F'$.

The case $\theta_{W,F} \si (\UM \oplus \UH_2; L \oplus (\Z \times \{ 0 \}), V)$, $x \si (\UM \oplus \UH_2; L \oplus ( \{0 \} \times \Z), V)$ can be reduced to the previous one by applying the automorphism $\id_M \oplus \sigma$ of $\UM \oplus \UH_2$ (where $\sigma : \Z^2 \rightarrow \Z^2$ is the flip map).

In general, $\theta_{W,F}$ is related to $x$ by a sequence of the above elementary equivalences, so we can get an $F' : W' \rightarrow B$ by a repeated application of Lemma \ref{lem:theta-surg-real}.
\end{proof}

\begin{lem} \label{lem:theta-surg-real}
Suppose that $\theta_{W,F} \si (\UM_0 \oplus \UH_2; L_0 \oplus ( \{0 \} \times \Z), V)$ for some $\UM_0$, $L_0$ and $V$. Then there is a $q$-connected normal bordism $F'' : W'' \rightarrow B$ which is normally bordant to $F$ such that $\theta_{W'',F''} \si (\UM_0 \oplus \UH_2; L_0 \oplus (\Z \times \{0 \} ), V)$.
\end{lem}

\begin{proof}
First we construct a normal bordism $F''$ by applying surgery to $F$. 

By Proposition \ref{prop:theta-cw-real} there is a CW-decomposition $c$ of $W$ and an integer $k \geq 0$ such that $\theta_{c,F} \cong \HH_{2k} \oplus (\UM_0 \oplus \UH_2; L_0 \oplus ( \{0 \} \times \Z), V)$. To simplify notation let $\UbM_0 = \UH_{2k} \oplus \UM_0$, $\bar{L}_0 = ( \{0 \} \times \Z^k) \oplus L_0$ and $\bar{V} = (\Z^k \times \{0 \} ) \oplus V$. Let $\bar{M}_0$ denote the underlying group of $\UbM_0$. Let $K$ denote the $q$-skeleton of $W$ with respect to $c$ and let $U$ be a regular neighbourhood of $K$, then by definition $\theta_{c,F}$ is the isomorphism class of $\theta_{U,F}$. We fix an isomorphism between $\theta_{U,F}$ and $(\UbM_0 \oplus \UH_2; \bar{L}_0 \oplus ( \{0 \} \times \Z), \bar{V})$. 

Let $\alpha \in H_q(W)$ be the image of $(0,(1,0)) \in \bar{M}_0 \oplus \Z^2$ under the homomorphism $H_q(\partial U) \rightarrow H_q(W)$ induced by the inclusion. Since $(0,(1,0))$ is in the hyperbolic component of $H_q(\partial U)$, we have $H_q(F \big| _{\partial U})(0,(1,0)) = 0$ and hence $H_q(F)(\alpha) = 0$.

Consider the commutative diagram
\[
\xymatrix{
H_{q+1}(B, W) \ar[r] & H_q(W) \ar[r]^-{H_q(F)} & H_q(B) \\
\pi_{q+1}(B, W) \ar[u]^-{\cong} \ar[r] & \pi_q(W) \ar[u] \ar[r]^-{\pi_q(F)} & \pi_q(B) \ar[u]
}
\]
induced by the Hurewicz-homomorphisms (where we regard $W$ as a subspace of $B$ via $F$). The pair $(B,W)$ is $q$-connected, hence $\pi_{q+1}(B, W) \cong H_{q+1}(B, W)$. 

Since $H_q(F)(\alpha) = 0$, $\alpha$ is the image of some element in $H_{q+1}(B, W)$. The image of the corresponding relative homotopy class in $\pi_{q+1}(B, W)$ under the boundary homomorphism is an element $\hat{\alpha} \in \pi_q(W)$, which is mapped to $\alpha$ by the Hurewicz-homomorphism, and it follows from the exactness of the sequence that $\pi_q(F)(\hat{\alpha}) = 0$. 

Since $q < \frac{1}{2} \dim W$, $\hat{\alpha}$ can be represented by an embedding $S^q \times D^{q+1} \rightarrow W$ along which we can perform a surgery over $(B, \xi)$ (see \cite[Lemma 2]{kreck99}). Let $F'' : W'' \rightarrow B$ be the result of this surgery, it is a normal bordism between $f_0$ and $f_1$ which is normally bordant to $F$ by construction. Since $F$ is $q$-connected and we did surgery along a $q$-sphere, $F''$ is $q$-connected too (\cite[Lemma 1]{kreck99}). 

In the rest of the proof we will show that the obstruction $\theta_{W'',F''}$ is stably isomorphic to $(\UbM_0 \oplus \UH_2; \bar{L}_0 \oplus (\Z \times \{0 \} ), \bar{V})$, and hence to $(\UM_0 \oplus \UH_2; L_0 \oplus (\Z \times \{0 \} ), V)$. We start by applying Theorem \ref{thm:theta-bordism-inv}, which describes the relationship between $\theta_{W'',F''}$ and $\theta_{W,F}$ (more precisely, between certain representatives thereof, which are obtained using specific CW-decompositions). 

We construct CW-decompositions $c'$ and $c''$ of $W$ and $W''$ respectively from a CW-decomposition of $W \setminus \interior(S^q \times D^{q+1})$ as in the proof of Theorem \ref{thm:theta-bordism-inv}. Let $K'$ denote the $q$-skeleton of $W$ with respect to $c'$, then $K' \approx K'_0 \vee S^q$ for some $K'_0$, so $H_q(K') \cong H_q(K'_0) \oplus \Z$. Since $S^q \subseteq K'$ represents the sphere along which we do surgery, the inclusion $K' \rightarrow W$ induces a homomorphism that sends $(0,1) \in H_q(K') \cong H_q(K'_0) \oplus \Z$ to $\alpha \in H_q(W)$. Let $U'$ be a regular neighbourhood of $K'$, then $U'$ is of the form $U'_0 \natural (S^q \times D^{q+1})$, and we can assume that the composition $S^q \times D^{q+1} \rightarrow U' \rightarrow W$ coincides with the embedding used for the surgery. Therefore $\theta_{U',F} \cong (\UMp_0 \oplus \UH_2; L'_0 \oplus (\{ 0 \} \times \Z), V')$, where $\UMp_0 = (H_q(\partial U'_0), \lambda'_0, \mu'_0)$, $\lambda'_0$ is the intersection form of $\partial U'_0$, $\mu'_0 = H_q(F \big| _{\partial U'_0})$ and $L'_0 = \Image(H_{q+1}(U'_0, \partial U'_0) \rightarrow H_q(\partial U'_0))$. We also have $\theta_{c'',F''} \cong (\UMp_0 \oplus \UH_2; L'_0 \oplus ( \Z \times \{0 \}), V')$, as we saw in the proof of Theorem \ref{thm:theta-bordism-inv}.

Now consider the following diagram: 
\[
\xymatrix{
\theta_{W,F} \si (\UMp_0 \oplus \UH_2; L'_0 \oplus (\{ 0 \} \times \Z), V') \ar@{<->}[r]^-{\sim} \ar@{<->}[d]_-{\si \;} & (\UMp_0 \oplus \UH_2; L'_0 \oplus ( \Z \times \{0 \}), V') \si \theta_{W'',F''} \\
\theta_{W,F} \si (\UbM_0 \oplus \UH_2; \bar{L}_0 \oplus ( \{0 \} \times \Z), \bar{V}) \ar@{<->}[r]^-{\sim} & (\UbM_0 \oplus \UH_2; \bar{L}_0 \oplus (\Z \times \{0 \} ), \bar{V})
}
\]
The top row summarises the information we just got from Theorem \ref{thm:theta-bordism-inv}. The bottom row contains the data appearing in the statement of the lemma (stabilised by adding $\HH_{2k}$). The stable isomorphism in the left column follows from the well-definedness of $\theta_{W,F}$ (Theorem \ref{thm:tw-welldef}). Our goal is to prove that $(\UMp_0 \oplus \UH_2; L'_0 \oplus ( \Z \times \{0 \}), V') \si (\UbM_0 \oplus \UH_2; \bar{L}_0 \oplus (\Z \times \{0 \} ), \bar{V})$. 

The stable isomorphism between $(\UMp_0 \oplus \UH_2; L'_0 \oplus (\{ 0 \} \times \Z), V')$ and $(\UbM_0 \oplus \UH_2; \bar{L}_0 \oplus ( \{0 \} \times \Z), \bar{V})$ is given by an isomorphism $I : \UH_{2l} \oplus \UMp_0 \oplus \UH_2 \rightarrow \UH_{2m} \oplus \UbM_0 \oplus \UH_2$ for some $l,m$ that is compatible with the (stabilised) quasi-formations. We will prove that there exists an $I$ that also determines a stable isomorphism between $(\UMp_0 \oplus \UH_2; L'_0 \oplus ( \Z \times \{0 \}), V')$ and $(\UbM_0 \oplus \UH_2; \bar{L}_0 \oplus (\Z \times \{0 \} ), \bar{V})$. The key to this is the observation that we have some control over the stable isomorphism produced by Theorem \ref{thm:tw-welldef}. This is due to the freedom in the choice of the isomorphism $J$ in Lemma \ref{lem:cwstab-he}, which is used in Lemma \ref{lem:hom-real-map} to construct a $j$ with $H_q(j)=J$, and related to $I$ via the diagram in Lemma \ref{lem:cwhe-isom}. So we will construct a suitable $I$ by going through the proof of Theorem \ref{thm:tw-welldef} again, this time choosing the isomorphism $J$ more carefully.

The homomorphism $H_q(K) \cong (\bar{M}_0 \oplus \Z^2) / (\bar{L}_0 \oplus ( \{0 \} \times \Z)) \rightarrow H_q(W)$ is surjective, because $K$ is a $q$-skeleton of $W$ (and we used Lemma \ref{lem:skel-hom} for the isomorphism), and the image of (the coset of) $(0,(1,0))$ is by definition $\alpha$, so there is an induced homomorphism $H_q(K) / \left< [0,(1,0)] \right> \cong \bar{M}_0/\bar{L}_0 \rightarrow H_q(W) / \left< \alpha \right>$, which is also surjective. 

The homomorphism $H_q(K') \cong H_q(K'_0) \oplus \Z \rightarrow H_q(W)$ is surjective too, and the image of $(0,1)$ is $\alpha$, so there is an induced homomorphism $H_q(K'_0) \rightarrow H_q(W) / \left< \alpha \right>$, which is also surjective.

By Lemma \ref{lem:hom-stab-isom} a) there are integers $l, m \geq 0$ and an isomorphism $\Z^l \oplus H_q(K'_0) \rightarrow \Z^m \oplus \bar{M}_0/\bar{L}_0$ which commutes with the maps to $H_q(W) / \left< \alpha \right>$ (where the $\Z^l$ and $\Z^m$ components are mapped to $0$). We fix such an isomorphism.

Consider the homomorphisms $\Z^l \oplus H_q(K'_0) \rightarrow H_q(W)$ and $\Z^l \oplus H_q(K'_0) \cong \Z^m \oplus \bar{M}_0/\bar{L}_0 \rightarrow H_q(W)$ (where the $\Z^l$ and $\Z^m$ components are again mapped to $0$). After composing with the projection $H_q(W) \rightarrow H_q(W) / \left< \alpha \right>$, these two homomorphisms become equal, therefore their difference is a homomorphism $\Z^l \oplus H_q(K'_0) \rightarrow \left< \alpha \right>$. Since $\left< \alpha \right>$ is cyclic and $\Z^l \oplus H_q(K'_0)$ is free, this homomorphism  factors as a composition $\Z^l \oplus H_q(K'_0) \rightarrow \Z \rightarrow \left< \alpha \right>$, where the second homomorphism maps $1$ to $\alpha$. By adding the first homomorphism of this composition to the previously fixed isomorphism, we get a homomorphism $\Z^l \oplus H_q(K'_0) \rightarrow \Z^m \oplus \bar{M}_0/\bar{L}_0 \oplus \Z \cong \Z^m \oplus H_q(K)$ which commutes with the maps to $H_q(W)$. We can extend it to an isomorphism $J : \Z^l \oplus H_q(K') \cong \Z^l \oplus H_q(K'_0) \oplus \Z \rightarrow \Z^m \oplus H_q(K) \cong \Z^m \oplus \bar{M}_0/\bar{L}_0 \oplus \Z$ by sending $(0,0,1)$ to $(0,0,1)$. This isomorphism $J$ again commutes with the maps to $H_q(W)$.

Let $\tilde{c}$ be an $m$-fold stabilisation of $c$ and $\tilde{c}'$ be an $l$-fold stabilisation of $c'$. Let $\tilde{K}$ and $\tilde{K}'$ denote the $q$-skeletons of $W$ with respect to $\tilde{c}$ and $\tilde{c}'$, then $H_q(\tilde{K}) \cong \Z^m \oplus H_q(K)$ and $H_q(\tilde{K}') \cong \Z^l \oplus H_q(K')$. Moreover, $J : H_q(\tilde{K}') \cong \Z^l \oplus H_q(K') \rightarrow H_q(\tilde{K}) \cong \Z^m \oplus H_q(K)$ is an isomorphism which commutes with the homomorphisms to $H_q(W)$. By applying Lemma \ref{lem:hom-real-map} as in the proof of Lemma \ref{lem:cwstab-he}, we get a homotopy equivalence $j : \tilde{K}' \rightarrow \tilde{K}$ that commutes with the inclusions in $W$ (up to homotopy) such that $H_q(j)=J$. 

Let $\tilde{U}$ and $\tilde{U}'$ be regular neighbourhoods of $\tilde{K}$ and $\tilde{K}'$ respectively. Since $\tilde{c}$ and $\tilde{c}'$ are stabilisations of $c$ and $c'$, we have $\theta_{\tilde{U},F} \cong \HH_{2m} \oplus \theta_{U,F}$ and $\theta_{\tilde{U}',F} \cong \HH_{2l} \oplus \theta_{U',F}$ by Lemma \ref{lem:cwstab-isom}. Let $\UtM_0 = \UH_{2m} \oplus \UbM_0$, $\tilde{L}_0 = ( \{0 \} \times \Z^m) \oplus \bar{L}_0$ and $\tilde{V} = (\Z^m \times \{0 \} ) \oplus \bar{V}$, then $\theta_{\tilde{U},F} \cong (\UtM_0 \oplus \UH_2; \tilde{L}_0 \oplus ( \{0 \} \times \Z), \tilde{V})$. Similarly, let $\UtMp_0 = \UH_{2l} \oplus \UMp_0$, $\tilde{L}'_0 = ( \{0 \} \times \Z^l) \oplus L'_0$ and $\tilde{V}' = (\Z^l \times \{0 \} ) \oplus V'$, then $\theta_{\tilde{U}',F} \cong (\UtMp_0 \oplus \UH_2; \tilde{L}'_0 \oplus ( \{0 \} \times \Z), \tilde{V}')$. 

We introduce the notation $(\tilde{M}_0, \tilde{\lambda}_0, \tilde{\mu}_0) = \UtM_0$ and $(\tilde{M}'_0, \tilde{\lambda}'_0, \tilde{\mu}'_0) = \UtMp_0$, and let $\tilde{\lambda} = \tilde{\lambda}_0 \oplus \bigl[ \begin{smallmatrix} 0 & 1 \\ 1 & 0 \end{smallmatrix} \bigr]$ and $\tilde{\lambda}' = \tilde{\lambda}'_0 \oplus \bigl[ \begin{smallmatrix} 0 & 1 \\ 1 & 0 \end{smallmatrix} \bigr]$ be the bilinear functions in $\UtM_0 \oplus \UH_2$ and $\UtMp_0 \oplus \UH_2$ respectively. By Lemma \ref{lem:skel-hom} we have $H_q(\tilde{U}) \cong (\tilde{M}_0 \oplus \Z^2) / (\tilde{L}_0 \oplus ( \{0 \} \times \Z))$ and $H_q(\tilde{U}') \cong (\tilde{M}'_0 \oplus \Z^2) / (\tilde{L}'_0 \oplus ( \{0 \} \times \Z))$, and the quotient homomorphisms $Q : \tilde{M}_0 \oplus \Z^2 \rightarrow (\tilde{M}_0 \oplus \Z^2) / (\tilde{L}_0 \oplus ( \{0 \} \times \Z))$ and $Q' : \tilde{M}'_0 \oplus \Z^2 \rightarrow (\tilde{M}'_0 \oplus \Z^2) / (\tilde{L}'_0 \oplus ( \{0 \} \times \Z))$ are induced by the inclusions $\partial \tilde{U} \rightarrow \tilde{U}$ and $\partial \tilde{U}' \rightarrow \tilde{U}'$. 

By applying Lemma \ref{lem:cwhe-isom} to the homotopy equivalence $j : \tilde{K}' \rightarrow \tilde{K}$, we get that  $\theta_{\tilde{U}',F} \cong \theta_{\tilde{U},F}$, via an isomorphism $I$ that fits into the commutative diagram
\[
\!\!\!\!
\xymatrix{
H_q(\tilde{K}') \cong \Z^l \oplus H_q(K'_0) \oplus \Z \ar[r]^-{\cong} \ar[d]_-{H_q(j) = J} & H_q(\tilde{U}') \cong (\tilde{M}'_0 \oplus \Z^2) / (\tilde{L}'_0 \oplus ( \{0 \} \times \Z)) \ar[d] & H_q(\partial \tilde{U}') \cong \tilde{M}'_0 \oplus \Z^2 \ar[d]^-{I} \ar[l]_-{Q'} \\
H_q(\tilde{K}) \cong \Z^m \oplus \bar{M}_0/\bar{L}_0 \oplus \Z \ar[r]^-{\cong} & H_q(\tilde{U}) \cong (\tilde{M}_0 \oplus \Z^2) / (\tilde{L}_0 \oplus ( \{0 \} \times \Z)) & H_q(\partial \tilde{U}) \cong \tilde{M}_0 \oplus \Z^2 \ar[l]_-{Q}
}
\]
In particular, $I$ is an isomorphism $I : \UtMp_0 \oplus \UH_2 \rightarrow \UtM_0 \oplus \UH_2$ such that $I(\tilde{L}'_0 \oplus ( \{0 \} \times \Z)) = \tilde{L}_0 \oplus ( \{0 \} \times \Z)$ and $I(\tilde{V}') = I(\tilde{V})$. 

It remains to prove that $I(\tilde{L}'_0 \oplus ( \Z \times \{0 \})) = \tilde{L}_0 \oplus (\Z \times \{0 \} )$, implying that $I$ determines an isomorphism between $(\UtMp_0 \oplus \UH_2; \tilde{L}'_0 \oplus ( \Z \times \{0 \}), \tilde{V}')$ and $(\UtM_0 \oplus \UH_2; \tilde{L}_0 \oplus (\Z \times \{0 \} ), \tilde{V})$, and hence a stable isomorphism between $(\UMp_0 \oplus \UH_2; L'_0 \oplus ( \Z \times \{0 \}), V')$ and $(\UbM_0 \oplus \UH_2; \bar{L}_0 \oplus (\Z \times \{0 \} ), \bar{V})$.

Under the identifications $H_q(\tilde{U}') \cong H_q(\tilde{K}')$ and $H_q(\tilde{U}) \cong H_q(\tilde{K})$, the isomorphism $J$ maps $[0,(1,0)]$, the coset of $(0,(1,0)) \in \tilde{M}'_0 \oplus \Z^2$ (which corresponds to $(0,0,1) \in \Z^l \oplus H_q(K'_0) \oplus \Z$) to $[0,(1,0)]$, the coset of $(0,(1,0)) \in \tilde{M}_0 \oplus \Z^2$ (that is $(0,0,1) \in \Z^m \oplus \bar{M}_0/\bar{L}_0 \oplus \Z$) by construction. Hence $Q \circ I(0,(1,0)) = J \circ Q'(0,(1,0)) = J([0,(1,0)]) = [0,(1,0)]$, so $I(0,(1,0)) \in Q^{-1}([0,(1,0)])$, therefore $I(0,(1,0)) = (a,(1,b))$ for some $a \in \tilde{L}_0$ and $b \in \Z$. 

Since $I$ is an isomorphism of extended quadratic forms, we have $0 = \tilde{\lambda}'((0,(1,0)),(0,(1,0))) = \tilde{\lambda}((a,(1,b))(a,(1,b))) = \tilde{\lambda}_0(a,a)+2b = 2b$ (because $a \in \tilde{L}_0$ and $\tilde{L}_0$ is a T-lagrangian in $\UtM_0$), hence $b=0$. Therefore $I(0,(1,0)) = (a,(1,0))$ for some $a \in \tilde{L}_0$. 

In $\UtMp_0 \oplus \UH_2$ we have $(\Z \times \{0 \})^{\perp} = \tilde{M}'_0 \oplus (\Z \times \{0 \})$, so $\tilde{L}'_0 = (\tilde{L}'_0 \oplus ( \{0 \} \times \Z)) \cap (\Z \times \{0 \})^{\perp}$. Hence its image under the isomorphism $I$ is $I(\tilde{L}'_0) = I(\tilde{L}'_0 \oplus ( \{0 \} \times \Z)) \cap I((\Z \times \{0 \})^{\perp}) = (\tilde{L}_0 \oplus ( \{0 \} \times \Z)) \cap (I(\Z \times \{0 \}))^{\perp} = (\tilde{L}_0 \oplus ( \{0 \} \times \Z)) \cap \left< (a,(1,0)) \right> ^{\perp}$. For $(x,(0,y)) \in \tilde{L}_0 \oplus ( \{0 \} \times \Z)$ we have $\tilde{\lambda}((a,(1,0))(x,(0,y))) = \tilde{\lambda}_0(a,x)+y = y$, so $(\tilde{L}_0 \oplus ( \{0 \} \times \Z)) \cap \left< (a,(1,0)) \right> ^{\perp} = \tilde{L}_0$. Therefore $I(\tilde{L}'_0) = \tilde{L}_0$.

This implies that $I(\tilde{L}'_0 \oplus ( \Z \times \{0 \})) = \tilde{L}_0 \oplus (\Z \times \{0 \} )$, as required. 
\end{proof}

\subsection{Elementary obstructions and h-cobordisms}

\begin{prop} \label{prop:hcob-elem}
$W$ is an h-cobordism if and only if $\theta_{W,F} \in \s^T_{2q+1}(B,\xi)$ is elementary. 
\end{prop}

\begin{proof}
Choose an arbitrary $c$ and $U$, we will prove that $W$ is an h-cobordism if and only if $\theta_{U,F}$ is elementary. 

By Proposition \ref{prop:fund-hom-groups} e) we have $H_i(W, M_0) \cong 0$ for $i \leq q-1$, and it also follows that if $i \geq q+2$, then $H_i(W, M_0) \cong H^{2q+1-i}(W, M_1) \cong 0$ (by Poincar\'e-duality and the universal coefficient formula). Since $W$, $M_0$ and $M_1$ are all simply-connected, this means that $W$ is an h-cobordism if and only if $H_q(W, M_0) \cong H_{q+1}(W, M_0) \cong 0$.

Consider the homological long exact sequence of the triple $(W, W \setminus \interior U, M_0)$. By Proposition \ref{prop:fund-hom-groups} d) we have $H_q(W, W \setminus \interior U) \cong 0$. From the proof of Lemma \ref{lem:hom-triv2} we get a commutative diagram
\[
\xymatrix{
H^{q-1}(M_1) \ar[rr]^-{0} \ar[d]_-{\cong} \ar[dr]^-{\cong} & & H^q(W, M_1) \\ 
H^q(W \setminus \interior U, \partial U \sqcup M_1) & H^q(W, U \sqcup M_1) \ar[l]_-{\cong} \ar[ru] 
}
\] 
so the homomorphism $H^q(W \setminus \interior U, \partial U \sqcup M_1) \cong H^q(W, U \sqcup M_1) \rightarrow H^q(W, M_1)$ is trivial. By Poincar\'e-duality, the homomorphism $H_{q+1}(W \setminus \interior U, M_0) \rightarrow H_{q+1}(W, M_0)$ is trivial too. So we have the following exact sequence (note that this sequence is a copy of the horizontal exact sequence in \cite[Figure 1]{kreck99}, but $U$ is defined differently):
\[
0 \rightarrow H_{q+1}(W, M_0) \rightarrow H_{q+1}(W, W \setminus \interior U) \rightarrow H_q(W \setminus \interior U, M_0) \rightarrow H_q(W, M_0) \rightarrow 0 
\]
It follows that $W$ is an h-cobordism if and only if the map $H_{q+1}(W, W \setminus \interior U) \rightarrow H_q(W \setminus \interior U, M_0)$ is an isomorphism. 

This map can be written as a composition $H_{q+1}(W, W \setminus \interior U) \cong H_{q+1}(U, \partial U) \rightarrow H_q(\partial U) \rightarrow H_q(W \setminus \interior U, M_0)$. Since $H_{q+1}(U) \cong H_{q+1}(\sk_q W) \cong 0$, the map $H_{q+1}(U, \partial U) \rightarrow H_q(\partial U)$ is injective, and its image is by definition $L$. By the exact sequence 
\[
\xymatrix{
0 \ar[r] & H_{q+1}(W \setminus \interior U, \partial U \sqcup M_0) \ar[r] & H_q(\partial U) \ar[r] & H_q(W \setminus \interior U, M_0) \ar[r] & 0
} 
\]
of Theorem \ref{thm:V-ses} and the definition of $V$, the restriction of the map $H_q(\partial U) \rightarrow H_q(W \setminus \interior U, M_0)$ to $L$ is an isomorphism if and only if $H_q(\partial U) \cong L \oplus V$. 

Therefore $W$ is an h-cobordism if and only if $\theta_{U,F}$ is elementary. 
\end{proof}

\begin{prop} \label{prop:bord-hcob-elem}
$F$ is normally bordant (rel boundary) to some normal bordism $F' : W' \rightarrow B$ such that $W'$ is an h-cobordism if and only if $[\theta_{W,F}] \in \ell^T_{2q+1}(B,\xi)$ is elementary.
\end{prop}

\begin{proof}
First suppose that $F$ is normally bordant to an $F' : W' \rightarrow B$ such that $W'$ is an h-cobordism. In this case $F'$ is $q$-connected (because the inclusion $M_0 \rightarrow W'$ is a homotopy equivalence, and the composition $M_0 \rightarrow W' \stackrel{F'}{\rightarrow} B$ is equal to $f_0$, which is $q$-connected), so $\theta_{W',F'}$ is defined. It follows from Proposition \ref{prop:hcob-elem} that $\theta_{W',F'}$ is elementary, hence $[\theta_{W',F'}]$ is elementary, and by Theorem \ref{thm:theta-bordism-inv} $[\theta_{W,F}] = [\theta_{W',F'}]$. 

Now suppose that $[\theta_{W,F}]$ is elementary. This means that there is an elementary quasi-formation $x \in \gc^T_{2q+1}(B,\xi)$ such that $\theta_{W,F} \sim x$. By Theorem \ref{thm:theta-equiv-real} there is a $q$-connected normal bordism $F' : W' \rightarrow B$ which is normally bordant to $F$ such that $\theta_{W',F'} \si x$. Then $\theta_{W',F'}$ is elementary, so $W'$ is an h-cobordism by Proposition \ref{prop:hcob-elem}.
\end{proof}

\section{The computation of $[\theta_{W,F}]$} \label{s:obstr-comp}

We use the same setting as in Section \ref{s:obstr-def}. The goal of this section is to compute the obstruction $[\theta_{W,F}]$ in terms of the Q-forms of $f_0$ and $f_1$.

\begin{defin}
For $i=0,1$ let $\UE_i = (H_q(M_i), \lambda_i, \mu_i) = E_q(M_i, f_i)$, this is an extended quadratic form over $H_q(B)$. 

Note that $\UE_i$ is geometric (with respect to $\hat{v}_q(\xi)$) and full, because $f_i$ is a normal $(q{-}1)$-smoothing (see Lemmas \ref{lem:qf-geom} and \ref{lem:qf-full}). 
\end{defin}

We choose a CW-decomposition $c$ of $W$ and a regular neighbourhood $U$ of $\sk_q W$ as before. These will be fixed for the rest of this section. 

\begin{defin}
Let $(\UM; L, V) = \theta_{U,F}$, where $\UM = (H_q(\partial U), \lambda, \mu)$ and let $\UV = (V, \lambda \big| _{V \times V}, \mu \big| _V)$, this is an extended quadratic form over $H_q(B)$, geometric (with respect to $\hat{v}_q(\xi)$) and free.
\end{defin}

Recall that we made the following orientation conventions. $W$, $M_0$ and $M_1$ are oriented manifolds such that $W$ is an oriented cobordism between $M_0$ and $M_1$, that is, $W$ induces the positive orientation on $M_0$ and the negative orientation on $M_1$. The regular neighbourhood $U$ is a codimension-$0$ submanifold in $W$, and its orientation is the restriction of that of $W$. We orient $\partial U$ as the boundary of $U$. Thus $W \setminus \interior U$ induces the negative orientation on $\partial U$. 

We will need the following lemma, which is an analogue of \cite[Proposition 8. i)]{kreck99}.

\begin{lem} \label{lem:proj-qf}
There is a morphism of extended quadratic forms $p : \UV \rightarrow \UE_0$ such that the underlying homomorphism $p : V \rightarrow H_q(M_0)$ is surjective. 
\end{lem}

\begin{proof}
Let $d_1 : H_{q+1}(W \setminus \interior U, \partial U \sqcup M_0) \rightarrow H_q(\partial U)$ and $d_2 : H_{q+1}(W \setminus \interior U, \partial U \sqcup M_0) \rightarrow H_q(M_0)$ denote the components of the boundary homomorphism $d : H_{q+1}(W \setminus \interior U, \partial U \sqcup M_0) \rightarrow H_q(\partial U \sqcup M_0) \cong H_q(\partial U) \oplus H_q(M_0)$.

By Theorem \ref{thm:V-ses} $d_1$ is injective and its image is $V$. In the homological long exact sequence of the triple $(W, U \sqcup M_0, U)$ we have $H_q(W, U) \cong 0$ (see Proposition \ref{prop:fund-hom-groups} c)), therefore $d_2$ is surjective. Let $p = - d_2 \circ d_1^{-1} : V \rightarrow H_q(M_0)$, this is a surjective homomorphism. 

Let $X$ denote the image of $\partial : H_{q+1}(W \setminus \interior U, \partial U \sqcup M_0 \sqcup M_1) \rightarrow H_q(\partial U \sqcup M_0 \sqcup M_1)$. Since $\Image d \leq X$, we have $((-\lambda) \oplus \lambda_0) \big| _{\Image d \times \Image d} = 0$ and $(\mu \oplus \mu_0) \big| _{\Image d} = 0$ by Lemma \ref{lem:boundary-lagr}, and this implies that $p$ is a morphism of extended quadratic forms. (For more details see \cite[Lemma 4.5.3]{csn-thesis}.)
\end{proof}

To be able to make computations, we will pass from the set $\ell^T_{2q+1}(B,\xi)$ to the monoid $\ell_{2q+1}(B,\xi)$ using the bijection of Proposition \ref{prop:ell-bij}. Recall the notation from Definition \ref{def:quot}.

\begin{defin}
Let $\bar{M} = H_q(\partial U) / \Tor H_q(\partial U)$. Since $L$ is a T-lagrangian in $\UM$, we always have $\mu \big| _{\Tor H_q(\partial U)} = 0$, hence we can define $\UbM = (\bar{M}, \bar{\lambda}, \bar{\mu})$, let $\pi : H_q(\partial U) \rightarrow \bar{M}$ denote the quotient homomorphism. 

Let $\bar{\theta}_{U,F} = (\UbM; \bar{L}, \bar{V}) \in \gc_{2q+1}(B,\xi)$. Let $\bar{\theta}_{c,F} \in \gs_{2q+1}(B,\xi)$ and $\bar{\theta}_{W,F} \in \s_{2q+1}(B,\xi)$ be the elements represented by $\bar{\theta}_{U,F}$. 
\end{defin}

Then $[\bar{\theta}_{W,F}] \in \ell_{2q+1}(B,\xi)$ is the element corresponding to $[\theta_{W,F}] \in \ell^T_{2q+1}(B,\xi)$ under the bijection of Proposition \ref{prop:ell-bij}.

\begin{defin}
Let $\bar{E}_i = H_q(M_i) / \Tor H_q(M_i)$. If $\mu_i(\Tor H_q(M_i))=0$, then we can also define $\UbE_i = (\bar{E}_i, \bar{\lambda}_i, \bar{\mu}_i)$.
\end{defin}

\begin{thm} \label{thm:theta-calc}
Assume that $\mu_i(\Tor H_q(M_i))=0$ for $i=0,1$ and that $\UbE_0 \oplus (-\UbEs_1)$ is metabolic. Then for every lagrangian $\bar{K}$ in $\UbE_0 \oplus (-\UbEs_1)$ there is an $x \in L_{2q+1}(B,\xi)$ such that 
\[
[\bar{\theta}_{W,F}] = [\UbE_0 \oplus (-\UbEs_1); \bar{K}, \bar{E}_0] \oplus x \in \ell_{2q+1}(B,\xi) \text{\,.}
\]
\end{thm}

\begin{rem} \label{rem:theta-calc}
The boundary of $W$ is $M_0 \sqcup (-M_1)$, and by Lemmas \ref{lem:qf-sum} and \ref{lem:qf-neg} we have $E_q(M_0 \sqcup (-M_1), f_0 \sqcup f_1) = \UE_0 \oplus (-\UEs_1)$. Thus by Lemma \ref{lem:bound-qf-met} if $H_q(B)$ is free, then $\UE_0 \oplus (-\UEs_1)$ is metabolic. The freeness of $H_q(B)$ also implies that $\mu_i(\Tor H_q(M_i))=0$, so $\UbE_0 \oplus (-\UbEs_1)$ is defined and it is also metabolic, hence both assumptions hold. Moreover, $L_{2q+1}(B,\xi) \cong 0$ by Theorem \ref{thm:l-group-zero}, so in this case Theorem \ref{thm:theta-calc} completely determines $[\bar{\theta}_{W,F}]$.
\end{rem}

\begin{proof}
We have $[\bar{\theta}_{W,F}] = [\UbM; \bar{L}, \bar{V}]$ by definition. To compute it, first we will define a subgroup $V' \leq M$ such that in $\UbM$ we have $\bar{V}' = \bar{V}^{\perp}$. Then we will use the morphism $\bar{p} : \UbV \rightarrow \UbE_0$ induced by the morphism $p : \UV \rightarrow \UE_0$ from Lemma \ref{lem:proj-qf} and an analogous morphism $\bar{p}' : \UbVp \rightarrow -\UbEs_1$ to relate $\UbM$ to $\UbE_0 \oplus (-\UbEs_1)$.

First we consider $-W$, ie.\ $W$ with the opposite orientation, this is an oriented cobordism between $M_1$ and $M_0$, ie.\ $\partial(-W) = M_1 \sqcup (-M_0)$. After changing the orientations $F$, $f_0$ and $f_1$ are still normal $(q{-}1)$-smoothings over $(B,\xi)$, and the obstruction $\theta_{-W,F}$ can be defined as before. The CW-decomposition $c$ of $W$ that we fixed earlier is also a CW-decomposition of $-W$, and we can choose the same regular neighbourhood $U$ of the $q$-skeleton, but the orientation of $-W$ restricts to the negative orientation of $U$. We will denote by $\theta_{-U,F}$ the quasi-formation determined by $-U$ ($U$ with the opposite orientation) that represents $\theta_{-W,F}$.

Let $(\UMp; L', V') = \theta_{-U,F}$. Then $\UMp = E_q(-\partial U, F \big| _{\partial U}) = -\UMs = (H_q(\partial U), -\lambda, \mu)$ by Lemma \ref{lem:qf-neg}. By definition we have $L' = \Image(H_{q+1}(-U, -\partial U) \rightarrow H_q(-\partial U)) = \Image(H_{q+1}(U, \partial U) \rightarrow H_q(\partial U)) = L$ and $V' = \Image(H_{q+1}(-W \setminus \interior(-U), -\partial U \sqcup M_1) \rightarrow H_q(-\partial U)) = \Image(H_{q+1}(W \setminus \interior U, \partial U \sqcup M_1) \rightarrow H_q(\partial U))$. Note that $M_1$ takes the place of $M_0$, because $V$ was defined using the positive boundary component of $W$ (ie.\ $M_0$) and the positive boundary component of $-W$ is $M_1$.

Let $\UtVp = (V', -\lambda \big| _{V' \times V'}, \mu \big| _{V'})$. By Lemma \ref{lem:proj-qf} there is a morphism of extended quadratic forms $p' : \UtVp \rightarrow \UE_1$ such that the underlying homomorphism $p' : V' \rightarrow H_q(M_1)$ is surjective. Let $\UVp = (V', \lambda \big| _{V' \times V'}, \mu \big| _V') = -(\UtVp)^*$, then $p'$ is also a morphism $p' : \UVp = -(\UtVp)^* \rightarrow -\UEs_1$. 

We have $p' = - d'_2 \circ (d'_1)^{-1}$, where $d'_1$ and $d'_2$ denote the components of the boundary homomorphism $d' : H_{q+1}(W \setminus \interior U, \partial U \sqcup M_1) \rightarrow H_q(\partial U \sqcup M_1)$ as in the proof of Lemma \ref{lem:proj-qf}. We also get that $\Image d' \leq X$, hence $((-\lambda) \oplus \lambda_0 \oplus \lambda_1) \big| _{\Image d \times \Image d'} = 0$. This implies that $\lambda(x,y) = 0$ for every $x \in V$, $y \in V'$.

Now we consider $\UbM$. Since $V$ is free, we have $\UbV = \UV$ by definition, and $\bar{V} = \pi(V) \cong V$, therefore $\UbV \cong (\bar{V}, \bar{\lambda} \big| _{\bar{V} \times \bar{V}}, \bar{\mu} \big| _{\bar{V}})$. We will in fact identify these two extended quadratic forms, ie.\ we will regard $\bar{V} \leq \bar{M}$ as the underlying group of $\UbV$ (instead of $V$). Similarly we have $\UbVp \cong (\bar{V}', \bar{\lambda} \big| _{\bar{V}' \times \bar{V}'}, \bar{\mu} \big| _{\bar{V}'})$.

The surjective morphism $p : \UV \rightarrow \UE_0$ induces a surjective morphism $\bar{p} : \UbV \rightarrow \UbE_0$. Since $\bar{E}_0$ is free, there is a homomorphism $\bar{i} : \bar{E}_0 \rightarrow \bar{V}$ such that $\bar{p} \circ \bar{i} = \id_{\bar{E}_0}$. This is a morphism of extended quadratic forms $\bar{i} : \UbE_0 \rightarrow \UbV$, because $\bar{p}$ is a morphism of extended quadratic forms. Similarly, $p' : \UVp \rightarrow -\UEs_1$ induces a surjective morphism $\bar{p}' : \UbVp \rightarrow -\UbEs_1$, which has a right inverse $\bar{i}' : -\UbEs_1 \rightarrow \UbVp$. 

Let $Z = p^{-1}(\Tor H_q(M_0)) \leq V$, then $\bar{Z} = \Ker \bar{p} \leq \bar{V}$. Let $\bar{Y} = \Image \bar{i}$, then $\bar{V} = \bar{Y} \oplus \bar{Z}$. We have $\bar{\mu}(x) = \bar{\mu}_0(\bar{p}(x)) = \bar{\mu}_0(0) = 0$ and $\bar{\lambda}(x,y) = \bar{\lambda}_0(\bar{p}(x), \bar{p}(y)) = 0$ for every $x \in \bar{Z}$, $y \in \bar{V}$. Therefore $\bar{\lambda} \big| _{\bar{Z} \times \bar{Z}} = 0$ and $\bar{\mu} \big| _{\bar{Z}} = 0$, and $\UbV = \UbY \oplus \UbZ$, where $\UbY = (\bar{Y}, \bar{\lambda} \big| _{\bar{Y} \times \bar{Y}}, \bar{\mu} \big| _{\bar{Y}})$ and $\UbZ = (\bar{Z},0,0)$. Moreover, $\bar{i} : \UbE_0 \rightarrow \UbY$ is an isomorphism of extended quadratic forms. 

Similarly, let $Z' = (p')^{-1}(\Tor H_q(M_1)) \leq V'$ so that $\bar{Z}' = \Ker \bar{p}'$, and let $\bar{Y}' = \Image \bar{i}'$. If $\UbYp = (\bar{Y}', \bar{\lambda} \big| _{\bar{Y}' \times \bar{Y}'}, \bar{\mu} \big| _{\bar{Y}'})$ and $\UbZp = (\bar{Z}',0,0)$, then $\UbVp = \UbYp \oplus \UbZp$ and $\bar{i}' : -\UbEs_1 \rightarrow \UbYp$ is an isomorphism of extended quadratic forms. 

Since $\lambda(x,y) = 0$ for every $x \in V$, $y \in V'$, we get that $\bar{\lambda}(x,y) = 0$ for every $x \in \bar{V}$, $y \in \bar{V}'$, in particular $\bar{V}' \leq \bar{V}^{\perp}$. Since $\bar{V}'$ is a direct summand in $\bar{M}$ and $\rk \bar{V}' = \frac{1}{2} \rk \UM = \rk \bar{V}^{\perp}$, this implies that $\bar{V}' = \bar{V}^{\perp}$. Similarly $\bar{V} = (\bar{V}')^{\perp}$.

Since $\UbV = \UbY \oplus \UbZ$ and $\bar{\lambda}(x,y) = 0$ for every $x \in \bar{Z}$, $y \in \bar{V}$, we have $\bar{Z} \leq \bar{V} \cap \bar{V}^{\perp}$. On the other hand, $\UbY \cong \UbE_0$ is nonsingular, so $\bar{Y} \cap \bar{Y}^{\perp} = \{ 0 \}$, hence $\bar{Y} \cap \bar{V}^{\perp} = \{ 0 \}$. Therefore $\bar{V} \cap \bar{V}^{\perp} = \bar{Z}$. Similarly, $\bar{V}' \cap (\bar{V}')^{\perp} = \bar{Z}'$. Together with our earlier observations this means that $\bar{Z} = \bar{V} \cap \bar{V}' = \bar{Z}'$. So we have $\bar{V} = \bar{Y} \oplus \bar{Z}$, $\bar{V}' = \bar{Y}' \oplus \bar{Z}$ and $\bar{V} \cap \bar{V}' = \bar{Z}$, which implies that $\bar{Y} \cap \bar{Y}' = \{ 0 \}$. 

Let $\bar{N} = \bar{Y} \oplus \bar{Y}'$ and $\UbN = (\bar{N}, \bar{\lambda} \big| _{\bar{N} \times \bar{N}}, \bar{\mu} \big| _{\bar{N}}) = \UbY \oplus \UbYp$. Then $\bar{i} \oplus \bar{i}' : \UbE_0 \oplus (-\UbEs_1) \rightarrow \UbN$ is an isomorphism, so $\UbN$ is nonsingular, hence by Lemma \ref{lem:nonsing-summand} we have $\bar{M} = \bar{N} \oplus \bar{N}^{\perp}$. Let $\bar{N}' = \bar{N}^{\perp}$ and $\UbNp = (\bar{N}', \bar{\lambda} \big| _{\bar{N}' \times \bar{N}'}, \bar{\mu} \big| _{\bar{N}'})$, then $\UbM = \UbN \oplus \UbNp$. 

We already saw that $\bar{\lambda} \big| _{\bar{Z} \times \bar{Z}} = 0$ and $\bar{\mu} \big| _{\bar{Z}} = 0$. Also, $\bar{Z}$ is a direct summand in $\bar{V}$, hence in $\bar{M}$, hence in $\bar{N}'$. Finally, $\rk \bar{Z} = \rk \bar{V} - \rk \bar{Y} = \frac{1}{2}(\rk \UbM - \rk \UbN) = \frac{1}{2} \rk \UbNp$. Therefore $\bar{Z}$ is a lagrangian in $\UbNp$. 

By assumption $\bar{K}$ is a lagrangian in $\UbE_0 \oplus (-\UbEs_1)$, so $(\bar{i} \oplus \bar{i}')(\bar{K})$ is a lagrangian in $\UbN$. Therefore $(\bar{i} \oplus \bar{i}')(\bar{K}) \oplus \bar{Z}$ is a lagrangian in $\UbN \oplus \UbNp = \UbM$. 

Now we can compute $[\bar{\theta}_{W,F}]$. 
\[
\begin{aligned}
\bar{\theta}_{W,F} &\si (\UbM; \bar{L}, \bar{V}) \sim \\
 &\sim (\UbM; \bar{L}, (\bar{i} \oplus \bar{i}')(\bar{K}) \oplus \bar{Z}) \oplus (\UbM; (\bar{i} \oplus \bar{i}')(\bar{K}) \oplus \bar{Z}, \bar{V}) = \\
 &= (\UbM; \bar{L}, (\bar{i} \oplus \bar{i}')(\bar{K}) \oplus \bar{Z}) \oplus (\UbN \oplus \UbNp; (\bar{i} \oplus \bar{i}')(\bar{K}) \oplus \bar{Z}, \bar{Y} \oplus \bar{Z}) = \\
 &= (\UbM; \bar{L}, (\bar{i} \oplus \bar{i}')(\bar{K}) \oplus \bar{Z}) \oplus (\UbN; (\bar{i} \oplus \bar{i}')(\bar{K}), \bar{Y}) \oplus (\UbNp; \bar{Z}, \bar{Z}) \sim \\
 &\sim (\UbM; \bar{L}, (\bar{i} \oplus \bar{i}')(\bar{K}) \oplus \bar{Z}) \oplus (\UbN; (\bar{i} \oplus \bar{i}')(\bar{K}), \bar{Y}) \cong \\
 &\cong (\UbM; \bar{L}, (\bar{i} \oplus \bar{i}')(\bar{K}) \oplus \bar{Z}) \oplus (\UbE_0 \oplus (-\UbEs_1); \bar{K}, \bar{E}_0)
\end{aligned}
\]
where we used Theorem \ref{thm:jacobi}, Lemma \ref{lem:nkk-zero} and the isomorphism $\bar{i} \oplus \bar{i}'$. Let $x = [\UbM; \bar{L}, (\bar{i} \oplus \bar{i}')(\bar{K}) \oplus \bar{Z}] \in L_{2q+1}(B,\xi)$, then $[\bar{\theta}_{W,F}] = x \oplus [\UbE_0 \oplus (-\UbEs_1); \bar{K}, \bar{E}_0]$.
\end{proof}

\section{Proof of the Q-form conjecture when $H_q(B)$ is free} \label{s:qfc}

\begin{proof}[Proof of Theorem \ref{thm:qfc}]
First we apply surgery below the middle dimension to get a $q$-connected normal bordism $F : W \rightarrow B$ which is normally bordant to $F_0$ (see \cite[Proposition 4]{kreck99}). We have $H_i(M_0) \cong H_i(M_1)$ for every $i$ (by the assumptions that $f_0$ and $f_1$ are $q$-connected and $E_q(M_0,f_0) \cong E_q(M_1,f_1)$, and Poincar\'e-duality), hence $\chi(M_0)=\chi(M_1)$. In this setting the surgery obstruction $\theta_{W,F} \in \s^T_{2q+1}(B,\xi)$ is defined. Let $[\bar{\theta}_{W,F}] \in \ell_{2q+1}(B,\xi)$ denote the element corresponding to $[\theta_{W,F}] \in \ell^T_{2q+1}(B,\xi)$ under the bijection of Proposition \ref{prop:ell-bij}.

We define $\UE_i$, $\bar{E}_i$ and $\UbE_i$ as before (using that $H_q(B)$ is free). By assumption there is an isomorphism $I : \UE_0 \rightarrow \UE_1$, it induces an isomorphism $\bar{I} : \UbE_0 \rightarrow \UbE_1$. By Lemma \ref{lem:diag-lagr} b) $\Delta^*_{\bar{I}} \leq \bar{E}_0 \oplus \bar{E}_1$ is a lagrangian in $\UbE_0 \oplus (-\UbEs_1)$. By Theorem \ref{thm:theta-calc} and Remark \ref{rem:theta-calc} we have $[\bar{\theta}_{W,F}] = [\UbE_0 \oplus (-\UbEs_1); \Delta^*_{\bar{I}}, \bar{E}_0]$. 

We have $\Delta^*_{\bar{I}} \oplus \bar{E}_0 = \bar{E}_0 \oplus \bar{E}_1$, because $\Delta^*_{\bar{I}}$ is the anti-diagonal in $\bar{E}_0 \oplus \bar{E}_1$ (with respect to the isomorphism $\bar{I}$). This means that $[\UbE_0 \oplus (-\UbEs_1); \Delta^*_{\bar{I}}, \bar{E}_0]$, and hence $[\bar{\theta}_{W,F}]$ is elementary. Therefore $[\theta_{W,F}]$, the corresponding element of $\ell^T_{2q+1}(B,\xi)$, is elementary too (see Remark \ref{rem:elem-bij}). Then by Proposition \ref{prop:bord-hcob-elem} $F$ is normally bordant to a normal bordism $F' : W' \rightarrow B$ such that $W'$ is an h-cobordism. 
\end{proof}

\section{Inertia groups} \label{s:inertia}

In this section we prove Theorem \ref{thm:inert-ker}. The proof is a simplified version of the argument given in \cite[Section 5.2.3]{csn-thesis}, which identifies the extended inertia group of a polarised $8$-dimensional E-manifold with the kernel of a bordism map (noting that in the case of a $3$-connected $8$-manifold the polarisation is trivial). As explained in Section \ref{ss:ig}, the $I(M) \leq \Ker(\eta_{2q}^A)$ direction of Theorem \ref{thm:inert-ker} follows from a classical argument, and the $I(M) \geq \Ker(\eta_{2q}^A)$ direction uses the special case of the Q-form conjecture.

Recall that $\nu_M$ denotes the stable normal bundle of a manifold $M$, as well as its classifying map $\nu_M : M \rightarrow BSO$.

\begin{defin} \label{def:BSO-qa}
For a positive integer $q$ and a subgroup $A \leq \pi_q(BSO)$ let $g : BSO\left< q, A\right> \rightarrow BSO$ be a fibration such that $\pi_i(BSO\left< q, A\right>) \cong 0$ for $i < q$, $\pi_q(g) : \pi_q(BSO\left< q, A\right>) \rightarrow \pi_q(BSO)$ is an injective map with image $A$ and $\pi_i(g)$ is an isomorphism for $i > q$. Let $\xi_{q,A} = g^*(\gamma) \rightarrow BSO\left< q, A\right>$ be the pullback of the universal oriented vector bundle $\gamma \rightarrow BSO$.
\end{defin}

The conditions of Definition \ref{def:BSO-qa} determine $BSO\left< q, A\right>$ up to homotopy equivalence (compatible with $g$). To see that a fibration satisfying the conditions exists, we can take $BSO\left< q, A\right>$ to be eg.\  the homotopy fibre of the map $BSO\left< q \right> \rightarrow K(\pi_q(BSO) / A, q)$ inducing the quotient homomorphism $\pi_q(BSO) \rightarrow \pi_q(BSO) / A$ on $\pi_q$ (where $BSO\left< q \right>$ is the $(q{-}1)$-connected cover of $BSO$), with $g$ the composition $BSO\left< q, A\right> \rightarrow BSO\left< q \right> \rightarrow BSO$.

\begin{prop} \label{prop:nm-unique}
Suppose that $M$ is a $(q{-}1)$-connected manifold (of any dimension, possibly with boundary), and let $A \leq \pi_q(BSO)$ be a subgroup such that $\Image(\pi_q(\nu_M)) \leq A$. Then up to homotopy there is a unique normal map 
\[
\xymatrix{
\nu_M \ar[d] \ar[r] & \xi_{q,A} \ar[d] \\
M \ar[r] & BSO\left< q, A\right>
}
\]
Moreover, if $A = \Image(\pi_q(\nu_M))$, then this normal map is a normal $(q{-}1)$-smoothing.
\end{prop}

In particular, if $A = \Image(\pi_q(\nu_M))$, then $(BSO\left< q, A\right>,\xi_{q,A})$ is the normal $(q{-}1)$-type of $M$.

\begin{proof}
For the existence of such a normal map we need to show that $\nu_M : M \rightarrow BSO$ can be lifted to a map $f : M \rightarrow BSO\left< q, A\right>$. Since $M$ is $(q{-}1)$-connected, it is homotopy equivalent to a CW complex whose $q$-skeleton is a wedge of $q$-spheres. So the restriction of $\nu_M$ to the $q$-skeleton amounts to a choice of some elements in $\pi_q(BSO)$, which are in $\Image(\pi_q(\nu_M))$. Since $\Image(\pi_q(\nu_M))$ is contained in $A = \Image(\pi_q(g) : \pi_q(BSO\left< q, A\right>) \rightarrow \pi_q(BSO))$, this restriction can be lifted to $BSO\left< q, A\right>$. The obstructions to lifting the map on the higher-dimensional cells are in the groups $H^i(M;\pi_{i-1}(F))$ for $i > q$, where $F$ denotes the fibre of $g : BSO\left< q, A\right> \rightarrow BSO$. By the definition of $BSO\left< q, A\right>$ we have $\pi_i(F) \cong 0$ for $i \geq q$, so all of the obstructions vanish. Therefore there is a map $f : M \rightarrow BSO\left< q, A\right>$ which lifts $\nu_M$ (and hence it is covered by a bundle map). Moreover, $\pi_i(M) \cong \pi_i(BSO\left< q, A\right>) \cong 0$ for $i < q$, and if $A = \Image(\pi_q(\nu_M))$, then $\pi_q(f)$ is surjective (because $\Image(\pi_q(g) \circ \pi_q(f)) = \Image(\pi_q(\nu_M)) = A$ and $\pi_q(g) : \pi_q(BSO\left< q, A\right>) \rightarrow A$ is an isomorphism), so $f$ is $q$-connected.

For the uniqueness we need to show that the space of normal maps $\nu_M \rightarrow \xi_{q,A}$ is path-connected. This space is homeomorphic to the space of bundle maps $F(\nu_M) \rightarrow F(\xi_{q,A})$, where $F(\nu_M)$ and $F(\xi_{q,A})$ denote the oriented frame bundles of $\nu_M$ and $\xi_{q,A}$ respectively. Since $F(\nu_M)$ and $F(\xi_{q,A})$ are principal $SO$-bundles, a continuous map $F(\nu_M) \rightarrow F(\xi_{q,A})$ is a bundle map if and only if it is $SO$-equivariant. 

By \cite[Theorem 4 (8.1)]{husemoller94} the space of $SO$-equivariant maps has the following equivalent description. There is a well-defined map $F(\nu_M) \times_{SO} F(\xi_{q,A}) \rightarrow F(\nu_M) / SO$ given by $[x,y] \mapsto [x]$, and it is the projection of a locally trivial bundle over $F(\nu_M) / SO = M$ with fibre $F(\xi_{q,A})$ (which is associated to the principal $SO$-bundle $F(\nu_M)$). A section $s : M \rightarrow F(\nu_M) \times_{SO} F(\xi_{q,A})$ determines a well-defined $SO$-equivariant map $\bar{f}_s : F(\nu_M) \rightarrow F(\xi_{q,A})$ given by $\bar{f}_s(x)=y$ whenever $[x,y] \in \Image(s)$. This correspondence is a homeomorphism between the space of sections of $F(\nu_M) \times_{SO} F(\xi_{q,A})$ and the space of $SO$-equivariant maps $F(\nu_M) \rightarrow F(\xi_{q,A})$.

So it is enough to show that the space of sections of $F(\nu_M) \times_{SO} F(\xi_{q,A})$ is path-connected. Given any two sections, the obstructions to the existence of a homotopy between them are in the groups $H^i(\Sigma M;\pi_{i-1}(F(\xi_{q,A}))) \cong H^{i-1}(M;\pi_{i-1}(F(\xi_{q,A})))$. Since $M$ is $(q{-}1)$-connected, these groups vanish if $i \leq q$. Hence it is enough to show that $\pi_i(F(\xi_{q,A})) \cong 0$ for $i \geq q$.

Since $\xi_{q,A}$ is the pullback of the universal oriented vector bundle $\gamma \rightarrow BSO$, there is a commutative diagram between the homotopical long exact sequences of the associated principal $SO$-bundles:  
\[
\xymatrix{
\pi_{i+1}(BSO\left< q, A\right>) \ar[r] \ar[d] & \pi_i(SO) \ar[r] \ar[d]_-{\cong} & \pi_i(F(\xi_{q,A})) \ar[r] \ar[d] & \pi_i(BSO\left< q, A\right>) \ar[r] \ar[d] & \pi_{i-1}(SO) \ar[d]_-{\cong} \\
\pi_{i+1}(BSO) \ar[r]^-{\cong} & \pi_i(SO) \ar[r] & \pi_i(F(\gamma)) \ar[r] & \pi_i(BSO) \ar[r]^-{\cong} & \pi_{i-1}(SO) 
}
\]
Here $F(\gamma)=ESO$ is contractible, and by the definition of $BSO\left< q, A\right>$ the map $\pi_i(BSO\left< q, A\right>) \rightarrow \pi_i(BSO)$ is an isomorphism for $i > q$ and injective for $i=q$. Therefore by the $4$-lemma we have $\pi_i(F(\xi_{q,A})) \cong 0$ for $i \geq q$.
\end{proof}

\begin{defin}
Suppose that $V$ is a $(q{-}1)$-connected $2q$-manifold with $\partial V \approx S^{2q-1}$. Let $C^{2q}_V$ be the set of diffeomorphism classes of $2q$-manifolds $M$ such that $M \setminus \interior D^{2q} \approx V$. 
\end{defin}

The group $\Theta_{2q}$ acts transitively on $C^{2q}_V$ via connected sum. If $M \in C^{2q}_V$, then $M$ is $(q{-}1)$-connected, $\pi_q(M) \cong \pi_q(V)$ and $\Image(\pi_q(\nu_M)) = \Image(\pi_q(\nu_V)) \leq \pi_q(BSO)$.

\begin{defin}
Suppose that $V$ is a $(q{-}1)$-connected $2q$-manifold with $\partial V \approx S^{2q-1}$ and $A = \Image(\pi_q(\nu_V))$. Let $\eta_V : C^{2q}_V \rightarrow \Omega_{2q}(BSO\left< q, A\right>; \xi_{q,A})$ denote the map which sends a manifold $M$ to the bordism class of the normal map $M \rightarrow BSO\left< q, A\right>$. 
\end{defin}

By Proposition \ref{prop:nm-unique} $\eta_V$ is well-defined. In the special case $V = D^{2q}$ we have $C^{2q}_{D^{2q}} = \Theta_{2q}$, so we get a map $\eta_{D^{2q}} : \Theta_{2q} \rightarrow \Omega_{2q}(BSO\left< q, 0\right>; \xi_{q,0})$. This map is equal to $\eta_{2q}^0$ from Definition \ref{def:eta-A}, and in general $\eta_{2q}^A$ can be equivalently defined as the composition of $\eta_{D^{2q}}$ and the homomorphism $\Omega_{2q}(BSO\left< q, 0\right>; \xi_{q,0}) \rightarrow \Omega_{2q}(BSO\left< q, A\right>; \xi_{q,A})$ induced by the fibration $BSO\left< q, 0\right> \rightarrow BSO\left< q, A\right>$.

\begin{prop} \label{prop:eta1}
Suppose that $V$ is a $(q{-}1)$-connected $2q$-manifold with $\partial V \approx S^{2q-1}$ and $A = \Image(\pi_q(\nu_V))$. If $M \in C^{2q}_V$ and $\Sigma \in \Theta_{2q}$, then $\eta_V(M \# \Sigma) = \eta_V(M) + \eta_{2q}^A(\Sigma) \in \Omega_{2q}(BSO\left< q, A\right>; \xi_{q,A})$.
\end{prop}

\begin{proof}
The bordism class $\eta_V(M)$ is represented by a normal map $f : M \rightarrow BSO\left< q, A\right>$ and $\eta_{2q}^A(\Sigma)$ is represented by a normal map $f_0 : \Sigma \rightarrow BSO\left< q, A\right>$. The connected sum $M' := M \# \Sigma$ is formed using embeddings $D^{2q} \rightarrow M$ and $D^{2q} \rightarrow \Sigma$. When composed with $f$ and $f_0$ respectively, these embeddings give homotopic maps $D^{2q} \rightarrow BSO\left< q, A\right>$ (covered by bundle maps). The homotopy $D^{2q} \times I \rightarrow BSO\left< q, A\right>$, together with the maps $f \circ \pi_1 : M \times I \rightarrow BSO\left< q, A\right>$ and $f_0 \circ \pi_1 : \Sigma \times I \rightarrow BSO\left< q, A\right>$ (where $\pi_1$ denotes the projection to the first component), defines a normal bordism $W := ((M \sqcup \Sigma) \times I) \cup (D^{2q} \times I) \rightarrow BSO\left< q, A\right>$ between $f \sqcup f_0 : M \sqcup \Sigma \rightarrow BSO\left< q, A\right>$ and $f' := f \# f_0 : M' \rightarrow BSO\left< q, A\right>$. This shows that $f'$ represents $\eta_V(M) + \eta_{2q}^A(\Sigma)$. By Proposition \ref{prop:nm-unique} $f'$ also represents $\eta_V(M')$, therefore $\eta_V(M \# \Sigma) = \eta_V(M) + \eta_{2q}^A(\Sigma)$.
\end{proof}

In the special case $V = D^{2q}$ we get that $\eta_{D^{2q}}$ is a homomorphism, and hence the same is true for $\eta_{2q}^A$. So we can use it to define an action of $\Theta_{2q}$ on $\Omega_{2q}(BSO\left< q, A\right>; \xi_{q,A})$ by $x^{\Sigma} = x + \eta_{2q}^A(\Sigma)$, where $x \in \Omega_{2q}(BSO\left< q, A\right>; \xi_{q,A})$ and $\Sigma \in \Theta_{2q}$. Then Proposition \ref{prop:eta1} means that $\eta_V : C^{2q}_V \rightarrow \Omega_{2q}(BSO\left< q, A\right>; \xi_{q,A})$ is $\Theta_{2q}$-equivariant.

\begin{prop} \label{prop:inert-upper}
Suppose that $q \geq 3$, $M$ is a $(q{-}1)$-connected $2q$-manifold and $A = \Image(\pi_q(\nu_M))$. Then $I(M) \leq \Ker(\eta_{2q}^A : \Theta_{2q} \rightarrow \Omega_{2q}(BSO\left< q, A\right>; \xi_{q,A}))$.
\end{prop}

\begin{proof}
Let $V = M \setminus \interior D^{2q}$ and suppose that $\Sigma \in I(M)$. Then $M \# \Sigma \approx M$, so $\eta_V(M \# \Sigma) = \eta_V(M)$. By Proposition \ref{prop:eta1}, $\eta_V(M \# \Sigma) = \eta_V(M) + \eta_{2q}^A(\Sigma)$, hence $\eta_{2q}^A(\Sigma)=0$.
\end{proof}

The following proposition is the key to computing inertia groups. Part b) uses the special case of the Q-form conjecture (Theorem \ref{thm:qfc}) to show that two manifolds, $M$ and $M'$ are diffeomorphic (while part a) uses \cite{kreck99} and has been known in equivalent forms). This will be applied in Theorem \ref{thm:inert-lower} for $\Sigma \in \Theta_{2q}$ and $M' = M \# \Sigma$.

\begin{prop} \label{prop:eta2}
Suppose that $q \geq 3$, $V$ is a $(q{-}1)$-connected $2q$-manifold with $\partial V \approx S^{2q-1}$ and $A = \Image(\pi_q(\nu_V))$. If either 

a) $q$ is odd, or 

b) $q \equiv 0$, $4$ or $6$ mod $8$,

\noindent
then $\eta_V : C^{2q}_V \rightarrow \Omega_{2q}(BSO\left< q, A\right>; \xi_{q,A})$ is injective.
\end{prop}

\begin{proof}
Suppose that $M,M' \in C^{2q}_V$ and $\eta_V(M)=\eta_V(M')$. This means that the unique normal maps $f : M \rightarrow BSO\left< q, A\right>$ and $f' : M' \rightarrow BSO\left< q, A\right>$ are normally bordant. By Proposition \ref{prop:nm-unique} these maps are normal $(q{-}1)$-smoothings. Since $M,M' \in C^{2q}_V$, there are diffeomorphisms $j : V \rightarrow M \setminus \interior D^{2q}$ and $j' : V \rightarrow M' \setminus \interior D^{2q}$. 

a) Since $M \setminus \interior D^{2q} \approx V \approx M' \setminus \interior D^{2q}$, we have $\chi(M)=\chi(M')$. So all conditions of \cite[Theorem D]{kreck99} are satisfied, which implies that $M \approx M'$.

b) By excision $j$ and $j'$, regarded as embeddings $j : V \rightarrow M$ and $j' : V \rightarrow M'$, are $(2q-1)$-connected. In particular, $H_q(j)$ and $H_q(j')$ are isomorphisms. Since $\partial V \approx S^{2q-1}$, the map $H_q(V) \rightarrow H_q(V, \partial V)$ is an isomorphism, so there is an intersection form $H_q(V) \times H_q(V) \rightarrow \Z$. The isomorphism $H_q(j) : H_q(V) \rightarrow H_q(M)$ preserves the intersection form, because every homology class in $H_q(M)$ has a representative in $j(V) \subset M$. We get similarly that $H_q(j') : H_q(V) \rightarrow H_q(M')$ also preserves the intersection form. Moreover, $f \circ j$ and $f' \circ j'$ are both normal maps $V \rightarrow BSO\left< q, A\right>$, so by Proposition \ref{prop:nm-unique} they are homotopic. This implies that $H_q(f \circ j) = H_q(f' \circ j')$, hence $H_q(f) = H_q(f') \circ H_q(j') \circ H_q(j)^{-1}$. Therefore $H_q(j') \circ H_q(j)^{-1} : H_q(M) \rightarrow H_q(M')$ is an isomorphism between the Q-forms $E_q(f)$ and $E_q(f')$. 

We have $H_q(BSO\left< q, A\right>) \cong \pi_q(BSO\left< q, A\right>) \cong A \leq \pi_q(BSO)$. If $q \equiv 0$, $4$ or $6$ mod $8$, then $\pi_q(BSO)$ is free, so $H_q(BSO\left< q, A\right>)$ is also free. So in these cases all conditions of Theorem \ref{thm:qfc} are satisfied, therefore $M \approx M'$. 
\end{proof}

\begin{thm} \label{thm:inert-lower}
Suppose that $q \geq 3$, $M$ is a $(q{-}1)$-connected $2q$-manifold and $A = \Image(\pi_q(\nu_M))$. If either 

a) $q$ is odd, or 

b) $q \equiv 0$, $4$ or $6$ mod $8$,

\noindent
then $I(M) \geq \Ker(\eta_{2q}^A : \Theta_{2q} \rightarrow \Omega_{2q}(BSO\left< q, A\right>; \xi_{q,A}))$.
\end{thm}

\begin{proof}
Let $V = M \setminus \interior D^{2q}$ and suppose that $\Sigma \in \Ker(\eta_{2q}^A)$. Then $\eta_V(M) = \eta_V(M) + \eta_{2q}^A(\Sigma) = \eta_V(M \# \Sigma)$ by Proposition \ref{prop:eta1}. By Proposition \ref{prop:eta2} a) or b) this implies that $M \approx M \# \Sigma$.
\end{proof}

\begin{proof}[Proof of Theorem \ref{thm:inert-ker}]
We have $I(M) \leq \Ker(\eta_{2q}^A)$ by Proposition \ref{prop:inert-upper} and $I(M) \geq \Ker(\eta_{2q}^A)$ by Theorem \ref{thm:inert-lower} b).
\end{proof}

\begin{proof}[Proof of Theorem \ref{thm:inert-pontr}]
Suppose that $q = 4$ or $8$, and let $M$ be a $(q{-}1)$-connected $2q$-manifold. Let $A = \Image(\pi_q(\nu_M): \pi_q(M) \rightarrow \pi_q(BSO)) \leq \pi_q(BSO) \cong \Z$, then $A = \alpha_M\Z$ for a unique non-negative integer $\alpha_M$. By Theorems \ref{thm:inert-ker} and \ref{thm:bordism-ker} $I(M) = \Ker(\eta_{2q}^A)$, which is $0$ if $4 | \alpha_M$ and $\Theta_{2q}$ otherwise. 

Let $j = \frac{q}{4}$ and let $d_M$ denote the divisibility of the Pontryagin class $p_j(M) \in H^{4j}(M)$. By Bott-Milnor \cite{bott-milnor58} we have (for any $(4j{-}1)$-connected manifold) that $d_M = a_j(2j{-}1)!\alpha_M$, where $a_j=1$ if $j$ is even and $a_j=2$ if $j$ is odd. Therefore $4 | \alpha_M$ if and only if $8 | d_M$ (for $q=4$) or $24 | d_M$ (for $q=8$).
\end{proof}

\section{The stable smoothing set and stable class} \label{s:sss}

Now we consider the stable smoothing set and stable class for simply-connected $4k$-manifolds. We will prove Theorem \ref{thm:sc-bij} in Section \ref{ss:alg-bij}. Then in Section \ref{ss:alg-comp} we compute the set $\SI(E_q(M,f))$ and the action of $A^{(B,\xi)}_{[M,f]}$ on it for manifolds $M$ with rank-$2$ hyperbolic intersection form, thereby proving Theorem \ref{thm:hyp-sss}.

\subsection{Algebraic description of the stable class} \label{ss:alg-bij}

First we prove Theorem \ref{thm:sc-bij}, which reduces the problem of determining the stable smoothing set and stable class of a manifold to algebra. 

\begin{proof}[Proof of Theorem \ref{thm:sc-bij}]
a) First suppose that $g : N \rightarrow B$ represents an element of $\Str^{\st}(M,f)$. Then there is a normal bordism between $f$ and $g$, and the proof of \cite[Theorem 2]{kreck99} shows that it can be turned into an s-cobordism by the compatible subtraction of some tori (and, since $\chi(N)=\chi(M)$, the same number of tori are subtracted at both ends). The subtraction of one torus changes the Q-form of $f$ or $g$ by adding the hyperbolic form $\UH_2$, so by Remark \ref{rem:nec} we have $E_q(N,g) \oplus \UH_{2k} \cong E_q(M,f) \oplus \UH_{2k}$ for some $k \geq 0$, ie.\ $E_q(N,g)$ represents an element of $\SI(E_q(M,f))$. It also follows from Remark \ref{rem:nec} that the isomorphism class of $E_q(N,g)$ is independent of the choice of the representative $g : N \rightarrow B$, so taking the Q-form induces a well-defined map $E_q : \Str^{\st}(M,f) \rightarrow \SI(E_q(M,f))$. 

The argument from the proof of \cite[Theorem 10.2]{CCPS21} shows that this map is surjective. Namely, suppose that $\UE$ represents an element of $\SI(E_q(M,f))$, ie.\ $\UE \oplus \UH_{2k} \cong E_q(M,f) \oplus \UH_{2k}$ for some $k \geq 0$. By performing $k$ trivial $(q{-}1)$-surgeries on $f$ we get a normal $(q{-}1)$-smoothing with Q-form $E_q(M,f) \oplus \UH_{2k}$, and after performing $k$ further $q$-surgeries along a basis of a lagrangian in the $\UH_{2k}$ component of $\UE \oplus \UH_{2k}$ (which has been identified with $E_q(M,f) \oplus \UH_{2k}$), we get a $g : N \rightarrow B$ representing an element of $\Str^{\st}(M,f)$, with $E_q(N,g) \cong \UE$.

Finally suppose that $H_q(B)$ is free. If $g : N \rightarrow B$ and $g' : N' \rightarrow B$ represent elements of $\Str^{\st}(M,f)$, then they are normally bordant. So by Theorem \ref{thm:qfc}, if $E_q(N,g) \cong E_q(N',g')$, then $g$ and $g'$ represent the same element. This means that the map $E_q : \Str^{\st}(M,f) \rightarrow \SI(E_q(M,f))$ is also injective. 

b) Suppose that $h \in A^{(B,\xi)}_{[M,f]}$ and $(X,\lambda,\mu)$ represents an element of $\SI(E_q(M,f))$. We define the action of $h$ on $(X,\lambda,\mu)$ to be given by $(X,\lambda,h \circ \mu)$. The isomorphism class of $(X,\lambda,h \circ \mu)$ is determined by that of $(X,\lambda,\mu)$, so we only need to check that it is in $\SI(E_q(M,f))$. By the definition of $A^{(B,\xi)}_{[M,f]}$, the automorphism $h$ of $H_q(B)$ is induced by some $\bar{h} \in \Aut(B,\xi)_{[M,f]}$. By the surjectivity of $E_q$, we have $(X,\lambda,\mu) \cong E_q(N,g)$ for some $g : N \rightarrow B$ representing an element of $\Str^{\st}(M,f)$. Recall that $\Aut(B,\xi)_{[M,f]}$ acts on $\Str^{\st}(M,f)$, in particular $\bar{h}$ sends $g$ to $\bar{h} \circ g : N \rightarrow B$. Therefore $(X,\lambda,h \circ \mu) \cong E_q(N,\bar{h} \circ g)$ represents an element of $\SI(E_q(M,f))$, as required. 

In fact, we can define an action of $\Aut(B,\xi)_{[M,f]}$ on $\SI(E_q(M,f))$, where an element $\bar{h} \in \Aut(B,\xi)_{[M,f]}$ acts via the action of its induced automorphism $h \in A^{(B,\xi)}_{[M,f]}$. Then we have $\SI(E_q(M,f)) / \Aut(B,\xi)_{[M,f]} = \SI(E_q(M,f)) / A^{(B,\xi)}_{[M,f]}$. Moreover, $E_q : \Str^{\st}(M,f) \rightarrow \SI(E_q(M,f))$ is $\Aut(B,\xi)_{[M,f]}$-equivariant. Hence it induces a map 
\[
\Str^{\st}(M,f) / \Aut(B,\xi)_{[M,f]} \cong \Str^{\st}(M) \rightarrow \SI(E_q(M,f)) / \Aut(B,\xi)_{[M,f]} = \SI(E_q(M,f)) / A^{(B,\xi)}_{[M,f]}
\]
Since $E_q$ is surjective in general, and bijective if $H_q(B)$ is free, the same is true for the induced map. 
\end{proof}

\subsection{The computation of $\SI(\UM)$ for certain $\UM$} \label{ss:alg-comp}

In this section we prove Theorem \ref{thm:hyp-sss}. By Theorem \ref{thm:sc-bij} this is now a purely algebraic problem. First we need to compute $\SI(\UM)$ for those extended quadratic forms $\UM$ that can occur as Q-forms of the manifolds in Theorem \ref{thm:hyp-sss}, equivalently, whose underlying bilinear function is hyperbolic of rank $2$. We will do this in Theorem \ref{thm:si-hyp}. Then we prove that for these $\UM$ no automorphism of the abelian group $Q$ (which corresponds to $H_q(B)$ in the geometric setting) can act non-trivially on $\SI(\UM)$, see Proposition \ref{prop:si-action}.

By the assumptions of Theorem \ref{thm:hyp-sss} the groups $Q$ we have to consider are free and have rank at most $2$ (as a normal $(q{-}1)$-smoothing induces a surjection on $H_q$). The main part of the proof concerns the $\rk Q = 1$ (ie.\ $Q \cong \Z$) case. Up to isomorphism every free extended quadratic form over $\Z$ with rank-$2$ hyperbolic bilinear function is of the form $\UE_{a,b}$ (see Definition \ref{def:eab} below), and we will determine the size of $\SI(\UE_{a,b})$ in Proposition \ref{prop:si-eab}. The remaining cases will be taken care of by Lemmas \ref{lem:si-free}--\ref{lem:si-mu-inj}.

Note that the $\rk H_q(B) = 0$ and $2$ cases of Theorem \ref{thm:hyp-sss} could alternatively be proved using Kreck \cite[Corollary 4]{kreck99} and \cite[Proposition 8]{kreck99} respectively.

We start by defining an invariant $\kappa$.

\begin{defin}
For an extended quadratic form $\UM = (M, \lambda, \mu)$ over an abelian group $Q$ we define $\kappa(\UM) = (N, \lambda \big| _{N \times N}, \mu \big| _N)$, where $N = (\Ker \mu)^{\perp} \leq M$. 
\end{defin}

\begin{lem} \label{lem:kappa}
Suppose that $\UM = (M, \lambda, \mu)$ and $\UMp = (M', \lambda', \mu')$ are extended quadratic forms over an abelian group $Q$. 

a) If $\mu$ is injective, then $\kappa(\UM) = \UM$. 

b) If $\UM \cong \UMp$, then $\kappa(\UM) \cong \kappa(\UMp)$. 

c) If $\UMp$ is nonsingular and $\mu' = 0$, then $\kappa(\UM \oplus \UMp) \cong \kappa(\UM)$. 

d) If $\UM \oplus \UH_{2k} \cong \UMp \oplus \UH_{2l}$ for some $k,l \geq 0$, then $\kappa(\UM) \cong \kappa(\UMp)$. 
\end{lem}

\begin{proof}
a) In this case $\Ker \mu = 0$ and $0^{\perp}=M$. 

b) Any isomorphism $\UM \rightarrow \UMp$ restricts to an isomorphism $\Ker \mu \rightarrow \Ker \mu'$, and hence an isomorphism $(\Ker \mu)^{\perp} \rightarrow (\Ker \mu')^{\perp}$, which preserves the restrictions of $\lambda$ and $\mu$. 

c) Since $\mu' = 0$, we have $\Ker(\mu \oplus \mu') = (\Ker \mu) \oplus M'$, and since $\UMp$ is nonsingular, $((\Ker \mu) \oplus M')^{\perp} = ((\Ker \mu) \oplus 0)^{\perp} \cap (0 \oplus M')^{\perp} = ((\Ker \mu) \oplus 0)^{\perp} \cap (M \oplus 0) = (\Ker \mu)^{\perp}\oplus 0 \leq M \oplus 0 \leq M \oplus M'$. 

d) follows immediately from parts b) and c).
\end{proof}

Next we will consider free extended quadratic forms over $\Z$ with rank-$2$ hyperbolic bilinear function.

\begin{defin} \label{def:eab}
For $a,b \in \Z$ let $\UE_{a,b}$ denote the extended quadratic form $(\Z^2, \bigl[ \begin{smallmatrix} 0 & 1 \\ 1 & 0 \end{smallmatrix} \bigr], [a,b] )$.
\end{defin}

For example, $\UE_{0,0} = \UH_2$. Our goal is to determine the size of $\SI(\UE_{a,b})$, see Proposition \ref{prop:si-eab}. First we will characterise the extended quadratic forms $\UE$ that satisfy $\UE \oplus \UH_{2k} \cong \UE_{a,b} \oplus \UH_{2k}$ for some $k$ in Proposition \ref{prop:si1}. Then we will use Proposition \ref{prop:si2} to classify them up to isomorphism.

\begin{prop} \label{prop:si1}
Let $a,b \in \Z$, and let $\UE$ be an extended quadratic form over $\Z$. There is a $k \geq 0$ such that $\UE \oplus \UH_{2k} \cong \UE_{a,b} \oplus \UH_{2k}$ if and only if $\UE \cong \UE_{c,d}$ for some $c,d \in \Z$ such that $\gcd(a,b)=\gcd(c,d)$ and $ab=cd$. 
\end{prop}

Note that for $a,b \in \Z$ we define $\gcd(a,b)$ to be the unique non-negative integer such that $a\Z + b\Z = \gcd(a,b)\Z$ as subgroups of $\Z$. In particular, $\gcd(a,b)=0$ if and only if $a=b=0$. Similarly, $\lcm(a,b)$ is the unique integer such that $a\Z \cap b\Z = \lcm(a,b)\Z$ and $\lcm(a,b)$ has the same sign as $ab$ (hence $ab=\lcm(a,b)\gcd(a,b)$). 

\begin{proof}
If $a=b=0$, then let $\bar{a}=\bar{b}=1$, otherwise define $\bar{a} = \frac{a}{\gcd(a,b)}$ and $\bar{b} = \frac{b}{\gcd(a,b)}$. Then in all cases $a = \bar{a} \gcd(a,b)$, $b = \bar{b} \gcd(a,b)$, $a\bar{b} = \bar{a}b = \lcm(a,b)$ and $\gcd(\bar{a},\bar{b}) = 1$.

If $\UE \oplus \UH_{2k} \cong \UE_{a,b} \oplus \UH_{2k}$ for some $k \geq 0$, then $\UE$ has rank $2$, and its bilinear function has signature $0$, hence it is indefinite. Moreover, it is even, so by the classification of indefinite forms (see \cite[Ch.\ II.\ Theorem (5.3)]{milnor-husemoller73}), it is isomorphic to $\bigl[ \begin{smallmatrix} 0 & 1 \\ 1 & 0 \end{smallmatrix} \bigr]$. Therefore $\UE \cong \UE_{c,d}$ for some $c,d \in \Z$. 

For any pair of extended quadratic forms $(M,\lambda,\mu)$ and $(M',\lambda',\mu')$ over an abelian group $Q$, if $(M,\lambda,\mu) \oplus \UH_{2k} \cong (M',\lambda',\mu') \oplus \UH_{2k}$ for some $k \geq 0$, then $\Image(\mu)=\Image(\mu') \leq Q$. In the case of $\UE_{a,b}$ we have $\Image [a,b] = a\Z + b\Z = \gcd(a,b)\Z \leq \Z$, so we get that $\gcd(a,b)=\gcd(c,d)$. 

By Lemma \ref{lem:kappa} d) we have $\kappa(\UE_{a,b}) \cong \kappa(\UE_{c,d})$. By Lemma \ref{lem:kappa-ab} below, $a=b=0 \Leftrightarrow \kappa(\UE_{a,b}) \cong (0,0,0) \Leftrightarrow \kappa(\UE_{c,d}) \cong (0,0,0) \Leftrightarrow c=d=0$, and in this case $ab=0=cd$. Otherwise $\kappa(\UE_{a,b}) \cong (\Z,2\bar{a}\bar{b},2\lcm(a,b))$ and $\kappa(\UE_{c,d}) \cong (\Z,2\bar{c}\bar{d},2\lcm(c,d))$ (where $\bar{c}$ and $\bar{d}$ are defined analogously to $\bar{a}$ and $\bar{b}$). An isomorphism between $(\Z,2\bar{a}\bar{b},2\lcm(a,b))$ and $(\Z,2\bar{c}\bar{d},2\lcm(c,d))$  sends $1$ to either $1$ or $-1$, and since it preserves the bilinear function, we get that in both cases $2\bar{a}\bar{b} = 2\bar{c}\bar{d}$. This implies that $2ab = 2\bar{a}\bar{b}\gcd(a,b)^2 = 2\bar{c}\bar{d}\gcd(c,d)^2 = 2cd$, hence $ab=cd$. 

In the other direction, if $c$ and $d$ are integers such that $\gcd(a,b)=\gcd(c,d)$ and $ab=cd$, then also $\lcm(a,b) = \lcm(c,d)$. Therefore it is enough to show that $\UE_{a,b} \oplus \UH_2 \cong \UE_{\lcm(a,b),\gcd(a,b)} \oplus \UH_2$. 

Since $\gcd(\bar{a},\bar{b}) = 1$, there are $\alpha, \beta \in \Z$ such that $\alpha\bar{a} + \beta\bar{b} = 1$. Note that 
\[
\begin{bmatrix}
\beta\bar{b}^2 & \alpha\bar{a}^2 & -\alpha\beta\bar{a}\bar{b} & \bar{a}\bar{b} \\ 
\alpha & \beta & \alpha\beta & -1 \\
\beta\bar{b} & -\beta\bar{a} & \beta^2\bar{b} & \bar{a} \\
-\alpha\bar{b} & \alpha\bar{a} & \alpha^2\bar{a} & \bar{b}
\end{bmatrix}
\begin{bmatrix}
0 & 1 & 0 & 0 \\ 
1 & 0 & 0 & 0 \\
0 & 0 & 0 & 1 \\
0 & 0 & 1 & 0
\end{bmatrix}
\begin{bmatrix}
\beta\bar{b}^2 & \alpha & \beta\bar{b} & -\alpha\bar{b} \\ 
\alpha\bar{a}^2 & \beta & -\beta\bar{a} & \alpha\bar{a} \\
-\alpha\beta\bar{a}\bar{b} & \alpha\beta & \beta^2\bar{b} & \alpha^2\bar{a} \\
\bar{a}\bar{b} & -1 & \bar{a} & \bar{b}
\end{bmatrix}
= 
\begin{bmatrix}
0 & 1 & 0 & 0 \\ 
1 & 0 & 0 & 0 \\
0 & 0 & 0 & 1 \\
0 & 0 & 1 & 0
\end{bmatrix}
\]
and 
\[
[a,b,0,0]
\begin{bmatrix}
\beta\bar{b}^2 & \alpha & \beta\bar{b} & -\alpha\bar{b} \\ 
\alpha\bar{a}^2 & \beta & -\beta\bar{a} & \alpha\bar{a} \\
-\alpha\beta\bar{a}\bar{b} & \alpha\beta & \beta^2\bar{b} & \alpha^2\bar{a} \\
\bar{a}\bar{b} & -1 & \bar{a} & \bar{b}
\end{bmatrix}
= 
[\lcm(a,b),\gcd(a,b),0,0]
\]
So if $I : \Z^4 \rightarrow \Z^4$ is the homomorphism determined by the square matrix in the second equality, then $I$ is a morphism $\UE_{\lcm(a,b),\gcd(a,b)} \oplus \UH_2 \rightarrow \UE_{a,b} \oplus \UH_2$. Moreover, since $I^*(\UE_{a,b} \oplus \UH_2) = \UE_{\lcm(a,b),\gcd(a,b)} \oplus \UH_2$ is nonsingular (and $\rk \UE_{\lcm(a,b),\gcd(a,b)} \oplus \UH_2 = \rk \UE_{a,b} \oplus \UH_2$), $I$ is an isomorphism. Therefore $\UE_{a,b} \oplus \UH_2 \cong \UE_{\lcm(a,b),\gcd(a,b)} \oplus \UH_2$, as required.
\end{proof}

We used the following lemma:

\begin{lem} \label{lem:kappa-ab}
If $a=b=0$, then $\kappa(\UE_{a,b}) \cong (0,0,0)$, otherwise $\kappa(\UE_{a,b}) \cong (\Z,2\bar{a}\bar{b},2\lcm(a,b))$.
\end{lem}

\begin{proof}
If $a=b=0$, then we have $\Ker [0,0] = \Z^2$ and $(\Z^2)^{\perp} = 0$. Otherwise, $\Ker [a,b] \cong \Z$, generated by $\bigl( \begin{smallmatrix} \bar{b} \\ -\bar{a} \end{smallmatrix} \bigr)$, and $(\Ker [a,b])^{\perp} \cong \Z$, generated by $\bigl( \begin{smallmatrix} \bar{b} \\ \bar{a} \end{smallmatrix} \bigr)$. Choosing this as the single basis element of $(\Ker [a,b])^{\perp}$, the restriction of $\bigl[ \begin{smallmatrix} 0 & 1 \\ 1 & 0 \end{smallmatrix} \bigr]$ is given by $2\bar{a}\bar{b}$, and the restriction of $[a,b]$ is given by $\bar{b}a + \bar{a}b = 2\lcm(a,b)$. 
\end{proof}

\begin{lem}
We have $\Aut(\UH_2) = \left\{ \bigl[ \begin{smallmatrix} 1 & 0 \\ 0 & 1 \end{smallmatrix} \bigr], \bigl[ \begin{smallmatrix} 0 & 1 \\ 1 & 0 \end{smallmatrix} \bigr], \bigl[ \begin{smallmatrix} -1 & 0 \\ 0 & -1 \end{smallmatrix} \bigr], \bigl[ \begin{smallmatrix} 0 & -1 \\ -1 & 0 \end{smallmatrix} \bigr] \right\} \cong \Z_2 \oplus \Z_2$. 
\end{lem}

\begin{proof}
The hyperbolic form $\UH_2$ has only two lagrangians, $\Z \times \{ 0 \}$ and $\{ 0 \} \times \Z$, so for any $I \in \Aut(\UH_2)$ we have $I\bigl( \begin{smallmatrix} 1 \\ 0 \end{smallmatrix} \bigr), I\bigl( \begin{smallmatrix} 0 \\ 1 \end{smallmatrix} \bigr) \in \{ \bigl( \begin{smallmatrix} 1 \\ 0 \end{smallmatrix} \bigr),\bigl( \begin{smallmatrix} 0 \\ 1 \end{smallmatrix} \bigr),\bigl( \begin{smallmatrix} -1 \\ 0 \end{smallmatrix} \bigr),\bigl( \begin{smallmatrix} 0 \\ -1 \end{smallmatrix} \bigr)\}$. For each possible value of $I\bigl( \begin{smallmatrix} 1 \\ 0 \end{smallmatrix} \bigr)$ there is a unique choice of $I\bigl( \begin{smallmatrix} 0 \\ 1 \end{smallmatrix} \bigr)$ such that $I$ preserves the bilinear function of $\UH_2$.
\end{proof}

We will use the action of $\Aut(\UH_2)$ on $\Z^2$ via the inclusion $\Aut(\UH_2) \rightarrow \Aut(\Z^2)$.

\begin{prop} \label{prop:si2}
Let $a,b,c,d \in \Z$. Then $\UE_{a,b} \cong \UE_{c,d}$ if and only if $\bigl( \begin{smallmatrix} a \\ b \end{smallmatrix} \bigr)$ and $\bigl( \begin{smallmatrix} c \\ d \end{smallmatrix} \bigr)$ are in the same orbit of the action of $\Aut(\UH_2)$ (equivalently, if $\bigl( \begin{smallmatrix} c \\ d \end{smallmatrix} \bigr) \in \{ \bigl( \begin{smallmatrix} a \\ b \end{smallmatrix} \bigr), \bigl( \begin{smallmatrix} b \\ a \end{smallmatrix} \bigr), \bigl( \begin{smallmatrix} -a \\ -b \end{smallmatrix} \bigr), \bigl( \begin{smallmatrix} -b \\ -a \end{smallmatrix} \bigr) \}$). 
\end{prop}

\begin{proof}
An automorphism $I \in \Aut(\Z^2)$ is an isomorphism $\UE_{c,d} \rightarrow \UE_{a,b}$ if and only if it preserves the bilinear function $\bigl[ \begin{smallmatrix} 0 & 1 \\ 1 & 0 \end{smallmatrix} \bigr]$ (equivalently, $I \in \Aut(\UH_2)$) and $[c,d]=I^*([a,b])$ (ie.\ $[c,d] = [a,b] \circ I$). Since every element of $\Aut(\UH_2)$ is symmetric, the second condition can be replaced by $\bigl( \begin{smallmatrix} c \\ d \end{smallmatrix} \bigr) = I \bigl( \begin{smallmatrix} a \\ b \end{smallmatrix} \bigr)$, that is, $\bigl( \begin{smallmatrix} c \\ d \end{smallmatrix} \bigr)$ is the image of $\bigl( \begin{smallmatrix} a \\ b \end{smallmatrix} \bigr)$ under the action of $I$.
\end{proof}

\begin{prop} \label{prop:si-eab}
Let $a,b \in \Z$. Then $|\SI(\UE_{a,b})|=1$ if $ab=0$ or $|a|=|b|$, and $|\SI(\UE_{a,b})|=2^{r-1}$ otherwise, where $r$ is the number of primes dividing $\frac{|ab|}{\gcd(a,b)^2}$.
\end{prop}

\begin{proof}
First suppose that $a=0$. If $c,d \in \Z$, then $\gcd(c,d)=\gcd(a,b)=b$ and $cd=ab=0$ if and only if $\bigl( \begin{smallmatrix} c \\ d \end{smallmatrix} \bigr) \in \{ \bigl( \begin{smallmatrix} 0 \\ b \end{smallmatrix} \bigr), \bigl( \begin{smallmatrix} b \\ 0 \end{smallmatrix} \bigr), \bigl( \begin{smallmatrix} 0 \\ -b \end{smallmatrix} \bigr), \bigl( \begin{smallmatrix} -b \\ 0 \end{smallmatrix} \bigr) \}$. By Propositions \ref{prop:si1} and \ref{prop:si2} $\UE \oplus \UH_{2k} \cong \UE_{0,b} \oplus \UH_{2k}$ for some $k \geq 0$ if and only if $\UE \cong \UE_{0,b}$. Therefore $|\SI(\UE_{0,b})|=1$. Similarly $|\SI(\UE_{a,0})|=1$. 

Next note that for $c,d \in \Z$ we have $|cd|=\gcd(c,d)^2$ if and only if $|c|=|d|=\gcd(c,d)$. So if $|a|=|b|$, then by Proposition \ref{prop:si1} $\UE \oplus \UH_{2k} \cong \UE_{a,b} \oplus \UH_{2k}$ for some $k \geq 0$ if and only if $\UE \cong \UE_{a,b}$ or $\UE \cong \UE_{-a,-b}$. In both cases $\UE \cong \UE_{a,b}$ by Proposition \ref{prop:si2}, showing that $|\SI(\UE_{a,b})|=1$. 

If $ab \neq 0$ and $|a| \neq |b|$, then define $\bar{a} = \frac{a}{\gcd(a,b)}$ and $\bar{b} = \frac{b}{\gcd(a,b)}$. By Proposition \ref{prop:si1} $\UE \oplus \UH_{2k} \cong \UE_{a,b} \oplus \UH_{2k}$ for some $k \geq 0$ if and only if $\UE \cong \UE_{\bar{c}\gcd(a,b),\bar{d}\gcd(a,b)}$ for some $\bar{c}, \bar{d} \in \Z$ such that $\gcd(\bar{c},\bar{d})=1$ and $\bar{c}\bar{d} = \bar{a}\bar{b}$. Equivalently, $\UE \cong \UE_{\bar{c}\gcd(a,b),\lcm(a,b)/\bar{c}}$ for some $\bar{c} \in \Z \setminus \{ 0 \}$ such that $\bar{c}$ divides $\bar{a}\bar{b}$ and $\gcd(\bar{c},\frac{\bar{a}\bar{b}}{\bar{c}})=1$. If the prime factorisation of $|\bar{a}\bar{b}| = \frac{|ab|}{\gcd(a,b)^2}$ is $|\bar{a}\bar{b}|=p_1^{k_1} \ldots p_r^{k_r}$, then $\bar{c}$ satisfies the conditions if and only if $\bar{c} = \epsilon \prod_{i \in I} p_i^{k_i}$ for some $\epsilon \in \{ \pm 1 \}$ and $I \subseteq \{ 1, \ldots , r \}$. So there are $2^{r+1}$ possible values of $\bar{c}$. By our assumptions $|\bar{a}\bar{b}|>1$, so $|\bar{c}| \neq \left| \frac{\bar{a}\bar{b}}{\bar{c}} \right|$ if $\bar{c}$ satisfies the conditions, hence $|\bar{c}\gcd(a,b)| \neq \left| \frac{\lcm(a,b)}{\bar{c}} \right|$. Therefore $\Aut(\UH_2)$ acts freely on the set of possible pairs $\bigl( \begin{smallmatrix} \bar{c}\gcd(a,b) \\ \lcm(a,b)/\bar{c} \end{smallmatrix} \bigr)$. So it follows from Proposition \ref{prop:si2} that $|\SI(\UE_{a,b})|=\frac{2^{r+1}}{|\Aut(\UH_2)|} = 2^{r-1}$.
\end{proof}

\begin{rem}
Conway-Crowley-Powell-Sixt \cite{CCPS23} constructed two families of manifolds, $N_{a,b}$ and $M_{a,b}$, for certain pairs of positive integers $a,b$. Both $N_{a,b}$ and $M_{a,b}$ admit normal smoothings over a suitable base space such that their Q-forms are isomorphic to $\UE_{a,b}$, and this was used to estimate the size of their stable classes, see \cite[Proposition 2.2, Theorem 3.3]{CCPS23}. Proposition \ref{prop:si-eab} can be regarded as an algebraic analogue of those results. 
\end{rem}

To prove Theorem \ref{thm:si-hyp} about $\SI(\UM)$ for the types of extended quadratic forms $\UM = (M, \lambda, \mu)$ not covered in Proposition \ref{prop:si-eab}, we will need the next three lemmas. First, we show that allowing $M$ to contain torsion makes no difference. Next, when $\rk Q = 0$, an extended quadratic form over $Q \cong 0$ can be regarded as a form over $\Z$ via the inclusion $0 \rightarrow \Z$, and applying this (or any) inclusion of groups has no effect on $\SI(\UM)$. Finally, if $\rk Q = 2$ and $\UM$ is free, full and has rank $2$, then $\mu$ is injective, and we will prove that this implies that $\SI(\UM)$ is trivial.

\begin{lem} \label{lem:si-free}
Suppose that $\UM = (M, \lambda, \mu)$ is an extended quadratic form over an abelian group $Q$ such that $\mu \big| _{\Tor M} = 0$ (hence $\UbM$ is defined, see Definition \ref{def:quot}). Then $\SI(\UM) \cong \SI(\UbM)$. 
\end{lem}

\begin{proof}
Let $R = \Tor M$ and $\UR = (R, 0, 0)$, then $\UM \cong \UbM \oplus \UR$. If $\UMp \oplus \UH_{2k} \cong \UbM \oplus \UH_{2k}$ for some $k \geq 0$, then $(\UMp \oplus \UR) \oplus \UH_{2k} \cong \UM \oplus \UH_{2k}$, and if $\UMp \cong \UMpp$, then $\UMp \oplus \UR \cong \UMpp \oplus \UR$. Therefore the assignment $\UMp \mapsto \UMp \oplus \UR$ induces a well-defined map $F : \SI(\UbM) \rightarrow \SI(\UM)$. 

Next consider an isomorphism $\UMp \oplus \UH_{2k} \cong \UM \oplus \UH_{2k}$ for some $k \geq 0$ and $\UMp = (M', \lambda', \mu')$. It restricts to an isomorphism $\Tor(M' \oplus \Z^{2k}) \cong \Tor M' \rightarrow \Tor(M \oplus \Z^{2k}) \cong R$, and since it is a morphism of extended quadratic forms, $\mu' \big| _{\Tor M'} = 0$. So $\UbMp$ is defined, $\UMp \cong \UbMp \oplus \UR$, and there is an induced isomorphism $\UbMp \oplus \UH_{2k} \cong \UbM \oplus \UH_{2k}$. If $\UMp \cong \UMpp$, then $\UbMp \cong \UbMpp$, so the assignment $\UMp \mapsto \UbMp$ induces a well-defined map $G : \SI(\UM) \rightarrow \SI(\UbM)$, which is the inverse of $F$. 
\end{proof}

\begin{lem} \label{lem:si-incl}
Let $\iota : Q_1 \rightarrow Q_2$ be a homomorphism between abelian groups, and let $(M,\lambda,\mu)$ be an extended quadratic form over $Q_1$, so that $(M,\lambda,\iota \circ \mu)$ is an extended quadratic form over $Q_2$. If $\iota$ is injective, then it induces a bijection $\SI(M,\lambda,\mu) \cong \SI(M,\lambda,\iota \circ \mu)$. 
\end{lem}

\begin{proof}
If $(M',\lambda',\mu') \oplus \UH_{2k} \cong (M,\lambda,\mu) \oplus \UH_{2k}$ for some $k \geq 0$, then $(M',\lambda',\iota \circ \mu') \oplus \UH_{2k} \cong (M,\lambda,\iota \circ \mu) \oplus \UH_{2k}$, and if $(M',\lambda',\mu') \cong (M'',\lambda'',\mu'')$, then $(M',\lambda',\iota \circ \mu') \cong (M'',\lambda'',\iota \circ \mu'')$. Therefore the assignment $(M',\lambda',\mu') \mapsto (M',\lambda',\iota \circ \mu')$ induces a well-defined map $\iota_* : \SI(M,\lambda,\mu) \rightarrow \SI(M,\lambda,\iota \circ \mu)$. 

If $(M',\lambda',\mu') \oplus \UH_{2k} \cong (M,\lambda,\iota \circ \mu) \oplus \UH_{2k}$ for some $k \geq 0$, then $\Image(\mu') = \Image(\iota \circ \mu) \leq \Image(\iota) \leq Q_2$. So if $\iota$ is injective and $\iota^{-1} : \Image(\iota) \rightarrow Q_1$ denotes its inverse, then $(M',\lambda',\mu') = (M',\lambda',\iota \circ \iota^{-1} \circ \mu')$, showing that $\iota_*$ is surjective. And if $(M',\lambda',\iota \circ \mu') \cong (M'',\lambda'',\iota \circ \mu'')$, then $(M',\lambda',\mu') = (M',\lambda',\iota^{-1} \circ \iota \circ \mu') \cong (M'',\lambda'',\iota^{-1} \circ \iota \circ \mu'') = (M'',\lambda'',\mu'')$, so $\iota_*$ is injective. 
\end{proof}

\begin{lem} \label{lem:si-mu-inj}
Suppose that $\UM = (M, \lambda, \mu)$ is an extended quadratic form over an abelian group $Q$. If $\mu$ is injective, then $|\SI(\UM)|=1$. 
\end{lem}

\begin{proof}
We need to show that if $\UM \oplus \UH_{2k} \cong \UMp \oplus \UH_{2k}$ for some $\UMp = (M', \lambda', \mu')$ and $k \geq 0$, then $\UM \cong \UMp$. The isomorphism $\UM \oplus \UH_{2k} \cong \UMp \oplus \UH_{2k}$ restricts to an isomorphism between $\Ker (\mu \oplus 0) = 0 \oplus \Z^{2k}$ and $\Ker (\mu' \oplus 0) = (\Ker \mu') \oplus \Z^{2k}$, which is only possible if $\Ker \mu' \cong 0$, ie.\ $\mu'$ is injective. Then by Lemma \ref{lem:kappa} a) and d) we have $\UM \cong \kappa(\UM) \cong \kappa(\UMp) \cong \UMp$.
\end{proof}

\begin{thm}  \label{thm:si-hyp}
Let $\UM = (M, \lambda, \mu)$ be an extended quadratic form over a free abelian group $Q$. Suppose that $\mu$ is surjective, $\rk \UM=2$ and the induced bilinear function on $M / \Tor(M) \cong \Z^2$ is isomorphic to $\bigl[ \begin{smallmatrix} 0 & 1 \\ 1 & 0 \end{smallmatrix} \bigr]$. Then 
\[
|\SI(\UM)| = 
\begin{cases}
1 & \text{if $\rk Q=0$ or $2$} \\
1 & \text{if $\rk Q=1$ and $|ab| \leq 1$} \\
2^{r-1} & \text{if $\rk Q=1$ and $|ab| \geq 2$}
\end{cases}
\]
where in the $\rk Q=1$ case we fix an identification $Q \cong \Z$ and elements $x,y \in M$ such that $[x],[y]$ is a basis of $M / \Tor M$ in which the induced bilinear function is given by the matrix $\bigl[ \begin{smallmatrix} 0 & 1 \\ 1 & 0 \end{smallmatrix} \bigr]$, we define $a = \mu(x), b=\mu(y) \in Q \cong \Z$, and $r$ is the number of primes dividing $|ab|$.
\end{thm}

\begin{proof}
Since $Q$ is free, we have $\mu \big| _{\Tor M} = 0$, and by Lemma \ref{lem:si-free} $|\SI(\UM)| = |\SI(\UbM)|$ (where $\UbM = (\bar{M}, \bar{\lambda}, \bar{\mu})$ is defined as in Definition \ref{def:quot}), so we need to prove the corresponding statement for $|\SI(\UbM)|$. By assumption there is an isomorphism $\bar{M} \cong \Z^2$ under which $\bar{\lambda}$ corresponds to $\bigl[ \begin{smallmatrix} 0 & 1 \\ 1 & 0 \end{smallmatrix} \bigr]$, and $\bar{\mu}$ is surjective (hence $\rk Q \leq 2$). 

If $Q \cong 0$, then $\UbM \cong (\Z^2, \bigl[ \begin{smallmatrix} 0 & 1 \\ 1 & 0 \end{smallmatrix} \bigr], 0 )$. By applying Lemma \ref{lem:si-incl} to the inclusion $0 \rightarrow \Z$ and Proposition \ref{prop:si-eab}, we have $|\SI(\UbM)| = |\SI(\UE_{0,0})| = 1$. 

If $Q \cong \Z^2$, then since $\bar{\mu}$ is surjective, it is also injective, so by Lemma \ref{lem:si-mu-inj} $|\SI(\UbM)| = 1$. 

If $Q \cong \Z$, then the basis $[x],[y]$ determines an isomorphism $\UbM \cong \UE_{a,b}$. Since $\bar{\mu}$ is surjective, $\gcd(a,b)=1$, which also implies that $|a|=|b|$ holds if and only if $|ab|=1$. Therefore the remaining cases follow from the corresponding parts of Proposition \ref{prop:si-eab}. 
\end{proof}

\begin{prop} \label{prop:si-action}
Let $\UM = (M, \lambda, \mu)$ and $Q$ be as in Theorem \ref{thm:si-hyp}. Suppose that $\UN = (N, \lambda_N, \mu_N)$ represents an element of $\SI(\UM)$. If $h$ is an automorphism of $Q$ such that $(N, \lambda_N, h \circ \mu_N)$ also represents an element of $\SI(\UM)$, then $(N, \lambda_N, h \circ \mu_N) \cong \UN$.
\end{prop}

\begin{proof}
Since $\mu \big| _{\Tor M} = 0$, by the proof of Lemma \ref{lem:si-free} it is enough to consider the case when $\UM$ is free. If $|\SI(\UM)|=1$, then any extended quadratic form that represents an element of $\SI(\UM)$ is isomorphic to $\UN$. Otherwise $\rk Q = 1$, and if $Q$ is identified with $\Z$, then $\UM \cong \UE_{a,b}$ for some $a,b \in \Z$. By Proposition \ref{prop:si1} $\UN \cong \UE_{c,d}$ for some $c,d \in \Z$. The only non-trivial automorphism of $\Z$ is $-\id_{\Z}$, and if $h = -\id_{\Z}$, then $(N, \lambda_N, h \circ \mu_N) \cong \UE_{-c,-d}$. By Proposition \ref{prop:si2} $\UE_{-c,-d} \cong \UE_{c,d}$.
\end{proof}

\begin{proof}[Proof of Theorem \ref{thm:hyp-sss}]
By Theorem \ref{thm:sc-bij} $|\Str^{\st}(M)| = |\SI(E_q(M,f)) / A^{(B,\xi)}_{[M,f]}|$. By Proposition \ref{prop:si-action} any automorphism of $H_q(B)$ that acts on $\SI(E_q(M,f))$, in particular any element of $A^{(B,\xi)}_{[M,f]}$, acts trivially. Therefore $|\Str^{\st}(M)| = |\SI(E_q(M,f))|$, which was determined in Theorem \ref{thm:si-hyp}.
\end{proof}

\bibliographystyle{amsinitial}
\bibliography{ref}

\end{document}